\documentclass[11pt, preprint]{imsart}
\usepackage{xr}
\usepackage{amsmath,amssymb}
\usepackage{parskip}
\usepackage{marvosym}
\usepackage[hang,small,bf]{caption}
\usepackage{mathtools,bm}
\usepackage[colorlinks]{hyperref}
\hypersetup{
	colorlinks,%
	citecolor=blue,%
	filecolor=black,%
	linkcolor=blue,%
	urlcolor=blue
}
\usepackage{enumerate}
\usepackage{multirow}
\RequirePackage[OT1]{fontenc}

\usepackage{amsthm}
\usepackage{amsmath}
\usepackage{amssymb}
\usepackage{thmtools} 
\usepackage{subcaption}
\usepackage{textcomp}

\usepackage{graphicx}
\usepackage{verbatim}
\usepackage{array, float}
\usepackage{fontenc}
\usepackage[toc,page]{appendix}
\usepackage[nottoc]{tocbibind}
\usepackage[english]{babel}
\usepackage{marvosym}

\usepackage[pagewise]{lineno}

\usepackage{mathtools}
\allowdisplaybreaks
\usepackage{bbm}
\usepackage[noabbrev,capitalize]{cleveref}

\newtheoremstyle{exampstyle}
{8pt} 
{8pt} 
{\it} 
{} 
{\bfseries} 
{.} 
{.5em} 
{} 

\theoremstyle{exampstyle}
\newtheorem{theorem}{Theorem}

\newtheorem{lemma}{Lemma}
\newtheorem{corollary}{Corollary}
\newtheorem{remark}{Remark}

\newtheorem{prop}{Proposition}

\newtheorem{defn}{Definition}

\numberwithin{equation}{section}
\numberwithin{example}{section}
\numberwithin{theorem}{section}
\numberwithin{lemma}{section}
\numberwithin{corollary}{section}
\numberwithin{prop}{section}
\numberwithin{defn}{section}
\numberwithin{remark}{section}

\usepackage{tikz}
\tikzset{
	treenode/.style = {shape=rectangle, rounded corners,
		draw, align=center,
		top color=white, bottom color=blue!20},
	root/.style     = {treenode, font=\Large, bottom color=yellow},
	env/.style      = {treenode, font=\ttfamily\normalsize},
	con/.style      = {treenode, font=\ttfamily, bottom color=green!25},
	nocon/.style    = {treenode, font=\ttfamily, bottom color=red!30},
	dummy/.style    = {circle,draw}
}

\startlocaldefs

\newcommand{\linn}{N_n^{\mathrm{lin}}}

\newcommand{\eat}[1]{}

\usepackage{enumitem}
\DeclareMathOperator*{\argmin}{\arg\!\min}

\renewcommand{\bar}[1]{\overline{#1}}
\newcommand{\hk}{\mathcal{H}_K}
\newcommand{\kmac}{\hat{\eta}_{n}}
\renewcommand{\hat}[1]{\widehat{#1}}
\renewcommand{\tilde}[1]{\widetilde{#1}}
\newcommand{\E}{\mathbbm{E}}
\renewcommand{\P}{\mathbbm{P}}
\newcommand{\X}{{\mathcal{X}}}
\newcommand{\Y}{{\mathcal{Y}}}
\newcommand{\tgn}{T^{\tilde{\mathcal{G}}_n}}

\newcommand{\emgn}{\mathcal{E}(\mathcal{G}_n)}

\newcommand{\R}{\mathbb{R}}
\newcommand{\indep}{\mathop{\rotatebox[origin=c]{90}{$\models$}}}

\newcommand\independent{\protect\mathpalette{\protect\independenT}{\perp}}
\def\independenT#1#2{\mathrel{\rlap{$#1#2$}\mkern2mu{#1#2}}}

\newcommand{\mcg}{\mathcal{G}}

\newcommand{\h}{\mathcal{H}}

\newcommand{\B}{\mathcal{B}}

\newcommand{\Z}{\mathcal{Z}}

\newcommand{\tmk}{{\mathcal{M}}_K}

\@addtoreset{proofpart}{theorem}

\makeatletter
\newcommand*{\rom}[1]{\expandafter\@slowromancap\romannumeral #1@}
\makeatother

\def\argmin{\mathop{\rm argmin}}

\newcommand{\rgn}{\mathcal{G}_n^{\textsc{rank}}}
\newcommand{\rmgn}{\mathcal{E}(\mathcal{G}_n^{\textsc{rank}})}

\newcommand{\trgn}{\tilde{\mathcal{G}}_n^{\textsc{rank}}}
\newcommand{\rkmac}{\hat{\eta}_n^{\textsc{rank}}}

\newcommand{\xs}{x^*}

\newcommand{\tX}{\tilde{\mathcal{X}}}

\newcommand{\nst}{(\tilde{G},\tilde{\mu})\in\mathcal{J}_{\theta}}
\newcommand{\N}{\mathcal{N}}

\newcommand{\dC}{\mathrm{dCor}}

\newcommand{\mam}{\mathcal{M}}

\newcommand{\hS}{\tilde{S}}
\newcommand{\td}{\tilde{d}}

\newcommand{\nrk}{N_n^{\textsc{rank}}}
\newcommand{\nrp}{N_n^{\textsc{pop}}}

\newcommand{\mme}{\mathbbm{E}}

\newcommand{\klin}{\hat{\eta}_n^{\mathrm{lin}}}

\newcommand{\kd}{\mathrm{K}_{\mathrm{D}}}
\newcommand{\kg}{\mathrm{K}_{\mathrm{G}}}
\newcommand{\mst}{\mathrm{MST}}
\newcommand{\onn}{1\mathrm{NN}}
\newcommand{\tnn}{20\mathrm{NN}}

\newcommand{\rpk}{\eta_K^{\textsc{rank}}(\mu)}

\newcommand{\ml}{\mathrm{lin}}

\definecolor{LightCyan}{rgb}{0.88,1,1}
\definecolor{Gray}{gray}{0.9}

\endlocaldefs

\usepackage{abstract}
\usepackage{xr}
\RequirePackage[numbers]{natbib}
\begin{document}
	
	\renewcommand{\abstractname}{}    
	\renewcommand{\absnamepos}{empty}
	
	\begin{frontmatter}
		\title{Measuring Association on Topological Spaces Using Kernels and Geometric Graphs}
		
		\runtitle{Measuring Association on Topological Spaces}

		\begin{aug}
			\author{\fnms{Nabarun} 
				\snm{Deb}, \ead[label=e1]{nd2560@columbia.edu}}\author{\fnms{Promit} \snm{Ghosal}
			\ead[label=e2]{promit@mit.edu}}
			\and
			\author{\fnms{Bodhisattva} \snm{Sen}\thanksref{t3}\ead[label=e3]{bodhi@stat.columbia.edu}}
			\affiliation{
				Columbia University\thanksmark{a1},
				Massachusetts Institute of Technology\thanksmark{a2}, and
			Columbia University\thanksmark{a3}
			}
			\thankstext{t3}{Supported by NSF grant DMS-2015376.}
			
			\runauthor{Deb, Ghosal and Sen}
			
			\address{1255 Amsterdam Avenue \\
				New York, NY 10027\\
				\printead{e1} \\
				}
				\address{77 Massachusetts Avenue \\
				Cambridge, MA\\
				\printead{e2} \\			
			    }
		        \address{1255 Amsterdam Avenue \\
		        	New York, NY 10027\\
		        	\printead{e3} \\
		        }
		\end{aug}
		\vspace{0.2in}
		\begin{abstract}
In this paper we propose and study a class of simple, nonparametric, yet interpretable measures of association between two random variables $X$ and $Y$ taking values in general topological spaces. These nonparametric measures --- defined using the theory of reproducing kernel Hilbert spaces --- capture the strength of dependence between $X$ and $Y$ and have the property that they are 0 if and only if the variables are independent and 1 if and only if one variable is a measurable function of the other. Further, these population measures can be consistently estimated using the general framework of geometric graphs which include $k$-nearest neighbor graphs and minimum spanning trees. Moreover, a sub-class of these estimators are also shown to adapt to the intrinsic dimensionality of the underlying distribution. Some of these empirical measures can also be computed in near linear time. Under the hypothesis of independence between $X$ and $Y$, these empirical measures (properly normalized) have a standard normal limiting distribution. Thus, these measures can also be readily used to test the hypothesis of mutual independence between $X$ and $Y$. 
			In fact, as far as we are aware, these are the only procedures that possess all the above mentioned desirable properties. Furthermore, when restricting to Euclidean spaces, we can make these sample measures of association finite-sample distribution-free, under the hypothesis of independence, by using multivariate ranks defined via the theory of optimal transport. The correlation coefficient proposed in~\citet{dette2013copula},~\citet{chatterjee2019new} and~\citet{azadkia2019simple} can be seen as a special case of this general class of measures.
		\end{abstract}

		
		
	\end{frontmatter}

	\section{Introduction}
	
	Suppose that $Z=(X,Y)\sim\mu$ where $\mu$ is supported on a subset of some topological space $\Z=\X\times \Y$ and has nondegenerate marginal distributions $\mu_{X}$ and $\mu_{Y}$, supported on $\X$ and $\Y$ respectively. Assume that we have i.i.d.~data $\{Z_i=({X}_i,{Y}_i)\}_{i = 1}^{n}$ from $\mu$. In this paper, we propose and study a class of simple yet interpretable empirical measures $T_n\equiv T_n(Z_1,\ldots ,Z_n)$ and their population counterparts, that yield a family of {\it nonparametric measures of association} between ${X}$ and ${Y}$. In addition, these empirical measures can be readily used as test statistics for testing the hypothesis of mutual independence between $X$ and $Y$.

	To explain our motivation, consider the case when $\X=\Y=\R$ and suppose that $\{Z_i\}_{i=1}^n$ are i.i.d. with a bivariate normal distribution. In this setting, the empirical Pearson's correlation coefficient  (see e.g.,~\cite{pearson1920notes}) captures the strength of association between $X$ and $Y$, i.e., it converges to a population measure which equals 0 if and only if $X$ and $Y$ are independent and 1 if and only if they are perfectly dependent (i.e., one variable is a function of the other). Moreover, any value between 0 and 1 of the correlation coefficient conveys an idea of the strength of the relationship between $X$ and $Y$. In addition, it has a simple limit distribution theory when $X$ and $Y$ are independent, and can consequently be used to test for the hypothesis of independence. Unfortunately, the Pearson's correlation coefficient ceases to have the aforementioned nice properties when the joint distribution of $X$ and $Y$ is not normal. The Spearman's rank correlation (see~\cite{spearman1904proof}) overcomes this shortcoming, but can only capture monotone relationships between $X$ and $Y$; also see Kendall's $\tau$~\cite{kendall1938new,Kendall1990}. 
	
	The above discussion raises a natural question: {\it ``Is it possible to define a simple empirical measure $T_n$ which provides a nonparametric measure of association between the variables $X$ and $Y$ under reasonable assumptions on $\X$, $\Y$ and $\mu$}?'' In this paper, we answer the above question in the affirmative. Towards that direction, it is perhaps instructive to first throw some light on the term {\it nonparametric measure of association} which we have been using informally so far. In this context, we will adhere to the criteria presented in~\citet{chatterjee2019new}, which is actually the main motivation behind our current work. Below we quote the relevant part from the abstract of~\cite{chatterjee2019new}:
	
	\noindent	``\textit{Is it possible to define a coefficient of correlation which is}: 
	\begin{enumerate}
		\item[(I)] \textit{As simple as the classical coefficients like Pearson's correlation or Spearman's correlation}, \textit{and yet} 
		
		\item[(II)] \textit{Consistently estimates some simple and interpretable measure of the degree of dependence between the variables, which is 0 if and only if the variables are independent and 1 if and only if one is a measurable function of the other, and} 
		
		\item[(III)] \textit{Has a simple
			asymptotic theory under the hypothesis of independence, like the classical coefficients?}''
	\end{enumerate}
	\noindent Although the above three properties seem quite natural and intuitive, we do not know of any nonparametric measure of association obeying (I)-(III), when $\X=\R^{d_1}$ and $\Y=\R^{d_2}$, for $d_1,d_2>1$ (in~\cite{chatterjee2019new, azadkia2019simple} the authors provide such a measure when $d_2 =1$), let alone for more general topological spaces.

	A plethora of nonparametric procedures have been proposed that can detect nonlinear dependencies between the variables $X$ and $Y$ over the last 60 years; see e.g.,~\cite{Renyi59, blum1961distribution, Rosenblatt75, friedman1983, Gabor2007, szekely2017energy, gieser1997, taskinen2005, gretton2008kernel, Oja-2010,Reshef11,Lyons13, heller2013, SS-2014-Bio, JH16, munmun2016, berrett2017nonparametric} and the references therein, however none of them satisfy (II). While these coefficients are indeed useful in practice, they have one common problem: They are all designed primarily for testing independence, and not for measuring the strength of the relationship between the variables.

	We now provide one concrete example of a measure satisfying (I)-(III) when $\X=\R^{d_1}$ and $\Y=\R^{d_2}$ for some $d_1,d_2\geq 1$. Note that this is just one example from the class of measures we will propose later. In order to motivate this measure, we begin by constructing the $k$-\emph{nearest neighbor graph} ($k$-NNG) of the data points $\{X_i\}_{i=1}^n \subset \R^{d_1}$, i.e., a graph with vertices $X_1,\ldots ,X_n$ where every vertex $X_i$ shares an edge with its $k$-nearest neighbors. 
	The $k$-NNG has the following property which makes it useful for our application --- the node pairs defining the edges represent points that tend to be `close' together (small distance or dissimilarity). 
	
	For $i = 1, \ldots, n$, let $\N_i$ denote the set of indices of the neighbors of $X_i$ in the corresponding $k$-NNG. Let $d_i$ denote the cardinality of the set $\N_i$. Consider the following statistic:
	\begin{equation}\label{eq:firststat}
	T_n:=1-\frac{\frac{1}{n}\sum_{i=1}^n d_i^{-1}\sum_{j\in \N_i} \lVert Y_i-Y_j\rVert_2}{\frac{1}{n(n-1)}\sum_{i\neq j} \lVert Y_i-Y_j\rVert_2},
	\end{equation}
	where $\|\cdot \|_2$ denotes the usual Euclidean norm. Note that $T_n$ has a simple form. We now present a result (see~\cref{pf:energyeq} for a proof) which shows that $T_n$ in~\eqref{eq:firststat} estimates a measure of the strength of dependence between $X$ and $Y$.
	\begin{prop}\label{prop:energyeq}
		Suppose $(X_1,Y_1),\ldots ,(X_n,Y_n)\overset{i.i.d.}{\sim} \mu$, a Borel probability measure with marginals $\mu_{X}$ and $\mu_Y$ (both nondegenerate). Let $\mu_{Y|x}$ denote the regular conditional distribution of $Y$ given $X=x$. Suppose $\lVert X_1-X_2\rVert_2$ has a continuous distribution, $\E \lVert Y_1\rVert_2^{2+\epsilon}<\infty$ for some $\epsilon>0$, and $k \le c n^{1-\delta}$ for some $c >0$ and $\delta \in (0,1]$. Then the following holds:
		\begin{equation}\label{eq:assocRd}
		T_n\overset{\mathbb{P}}{\longrightarrow}1-\frac{\E \lVert Y'-\tilde{Y'}\rVert_2}{\E \lVert Y_1-Y_2\rVert_2}:=T \ \qquad \mbox{as } n \to \infty.
		\end{equation}
		Here $(X',Y',\tilde{Y'})$ is generated as: draw $X'\sim \mu_X$ and then $Y'|X'\sim \mu_{Y|X'}$, $\tilde{Y'}|X'\sim \mu_{Y|X'}$ such that $Y'$ and $\tilde{Y'}$ are conditionally independent given $X'$. Moreover, $T\in [0,1]$ and $T$ equals $0$ if and only if $X$ and $Y$ are independent, and equals $1$ if and only if $Y$ is a noiseless measurable function of $X$.
	\end{prop}
	In the above discussion, we have attempted to view $T_n$ in light of properties (I) and (II). As it turns out, $\sqrt{n}T_n$, after suitable scaling, has a limiting standard normal law when $X$ and $Y$ are independent, as warranted in property (III); see~\cref{theo:nullclt}. This immediately yields a simple, easily computable and consistent method for testing independence between $X$ and $Y$.
	
	Among properties (I), (II) and (III), we believe that property (II) is perhaps most difficult to parse. In~\cref{sec:kmacintuit}, we provide an intuitive explanation as to why $T_n$ converges to $0$ ($1$) when $Y$ is independent of $X$ (a noiseless function of $X$). The converse directions which are a part of~\cref{prop:energyeq}, are considerably harder to prove. In fact, for the converse, the choice of the function $\lVert \cdot-\cdot\rVert_2$ is crucial; for instance the choice $\lVert\cdot-\cdot\rVert_2^2$ does not work. This raises a rather intriguing question:
	
	{\it ``Can we replace $\lVert\cdot -\cdot \rVert_2$ with a more general class of functions? Can we choose functions which do not rely on the structure of $\R^{d_2}$?"}
	
	We answer the above question using the framework of \emph{reproducing kernel Hilbert spaces} (RKHSs); see~\cref{def:rkhs}. Towards this direction, let us consider the case where $Z_i$'s take values in a general topological space $\X\times\Y$. Suppose there exists a symmetric, nonnegative definite kernel function $K(\cdot,\cdot):\Y\times\Y\to\R$ which is {\it characteristic}  (see~\cref{defn:Char}). Then we construct our {\it kernel measure of association} (abbreviated as KMAc) in a similar fashion as $T_n$ (from~\eqref{eq:firststat}) as follows:
	\begin{align}\label{eq:tempkmac}
	\kmac:=\frac{\frac{1}{n} \sum_{i=1}^n d_i^{-1}\sum_{j\in \N_i} K(Y_i,Y_j)- \frac{1}{n(n-1)}\sum_{i\neq j} K(Y_i,Y_j)}{\frac{1}{n}\sum_{i=1}^n K(Y_i,Y_i)-\frac{1}{n(n-1)}\sum_{i\neq j} K(Y_i,Y_j)}.
	\end{align}
	Note that the construction of $k$-NNGs can be carried out in rather general spaces provided there is some metric or a notion of ``similarity" between two elements in that space (see~\cite{boytsov2013learning,athitsos2004boostmap,miranda2013very,jacobs2000classification}). In fact, in~\cref{sec:kmacestimate}, we will go beyond $k$-NNGs and instead work with a more general class of \emph{geometric graphs} (see~\cref{sec:kmacestimate} for a definition), which includes the minimum spanning tree among others; see~\eqref{eq:statest} for the general version of our measure. On the other hand, the existence of {characteristic} kernels is a subject of active research in the machine learning community. Examples of such kernels are known for separable Hilbert spaces, certain non-Euclidean domains featuring texts, video/image and histogram-valued objects, etc. (see e.g.,~\cite{Lyons13,christmann2010universal,danafar2010characteristic}; also see~\cref{sec:exkergraph} for a discussion). Consequently, $\kmac$ (as in~\eqref{eq:tempkmac}) can be constructed in a very general setting. Even while working on $\R^{d_1+d_2}$, $\kmac$ provides a lot of flexibility as there are a number of characteristic kernels known in this case (see~\cref{rem:excharker}), some of which may have better properties than others, depending on the application at hand.
	
	In Theorem~\ref{theo:consis}, we show that  $\kmac$ consistently estimates a measure of dependency between $X$ and $Y$ which can be conveniently expressed as 
	\begin{equation}\label{eq:tempkac}
	\eta_{K}(\mu):= \frac{\mme[\mme [K(Y',\tilde{Y'})|X']]-\mme [K(Y_1,Y_2)]}{\mme [K(Y_1,Y_1)]-\mme[ K(Y_1,Y_2)]};
	\end{equation}
	here $Y_1,Y_2, Y', \tilde{Y'}$ are defined as in~\eqref{eq:assocRd}. A term by term comparison between~\eqref{eq:tempkac}~and~\eqref{eq:tempkmac} shows that $\kmac$ is indeed a natural estimator of $\eta_K(\mu)$. Note that $d_i^{-1}\sum_{j:(i,j)\in\emgn} K(Y_i,Y_j)$ can be \emph{informally} viewed as an empirical analogue of $\E [K(Y',\tilde{Y'}|X'=X_i)]$. Thus, compared to~\cite{chatterjee2019new,azadkia2019simple} where the authors claimed that the connection between their empirical and population measure of association was hard to motivate without getting into the technicalities of the proof, our approach makes the connection (between our empirical and population measures) more transparent.
	

	Below we summarize some of the key features of $\kmac$:
	
	\begin{enumerate}
		\item It can be computed in a broad variety of topological spaces $\X$ and $\Y$ (see~\cref{sec:exkergraph}). This is particularly useful in functional regression models (see~\cite{morris2015functional} for a survey), in real-life machine learning and human behavior recognition  (see~\cite{danafar2010characteristic}), measuring association between two stochastic processes (see~\cite{chwialkowski2014kernel}), etc.
		
		\item It has a simple, interpretable form as the classical correlation coefficients but is fully nonparametric. There is no estimation of conditional densities involved.

		\item It converges to a limit in $[0,1]$ which equals $0$ if and only if $X$ and $Y$ are independent; and to $1$ if and only if $Y$ is a noiseless measurable function of $X$ (see Theorems~\ref{theo:kacmas}~and~\ref{theo:consis}). Further, this limit is closely related to the notion of maximum mean discrepancy (see~\cref{prop:wdf}) and energy distance (see~\cref{lem:equivform} in the Appendix), which are popular and widely studied discrepancy measures between probability distributions in both the machine learning and statistics communities.
		
		\item It satisfies a moment concentration inequality around a population limit under mild assumptions on the kernel $K$ (see~\cref{cor:bdiffcon}). We further establish rates of convergence for $\kmac$ (constructed using $k$-NNG) to $\eta_K(\mu)$ which shows that $\kmac$ adapts to the {\it intrinsic dimensionality} of $X$ (see~\cref{theo:rateres}~and~\cref{cor:rateuclid}).
		
		\item When $X$ and $Y$ are independent, $\kmac$ (suitably normalized) satisfies a central limit theorem (CLT) with the standard normal limiting distribution, uniformly over a large class of graphs (see~\cref{theo:nullclt}). Thus, $\kmac$ readily yields a consistent method for testing independence between $X$ and $Y$, without resorting to permutation techniques. Moreover, the uniformity in the CLT justifies the use of a data-driven choice of $k$ in the $k$-NNG while constructing $\kmac$.
		
		
		
		\item In some cases, $\kmac$ can be computed in $\mathcal{O}(n\log{n})$ time (see~\cref{sec:compute}). We also propose a related estimator $\klin$ (see~\eqref{eq:fastkmac}) for $\eta_K(\mu)$, which is always computable in $\mathcal{O}(n\log{n})$ time, whenever a $k$-NNG is used for its construction. In addition, $\klin$ shares all the nice statistical properties of $\kmac$ (see points 1-5 above). For example,  in~\cref{prop:nullcltesscomp} we state a CLT, similar to~\cref{theo:nullclt}, that can be used to construct a consistent,  nonparametric test of independence between $X$ and $Y$ having near linear time complexity.

		\item When $\X=\R^{d_1}$ and $\Y=\R^{d_2}$, for $d_1,d_2 \ge 1$, by a suitable choice of kernel, $\kmac$ satisfies (I)-(III) under \emph{no assumptions} on the distribution of $(X,Y)$; see~\cref{prop:energyeqfree}. Also the limit of $\kmac$ (i.e., $\eta_K(\mu)$) has some additional features if the kernel $K(\cdot,\cdot)$ is chosen suitably. In particular, it exhibits invariance, equitability and continuity (see~\cite{Reshef2016,Mori2019}; also see~\cref{prop:otprop}). In fact, in many cases, $\eta_K(\mu)$ can be shown to be a continuous and monotonic function of the noise level (see~\cref{prop:spGauss}~and~\cref{sec:ice}). This justifies $\eta_K(\mu)$ as a measure of the strength of association between $X$ and $Y$ (see~\cite{dette2013copula}).
		
		\item $\kmac$ is asymmetric in $X$ and $Y$. Similar to the coefficient proposed in~\cite{chatterjee2019new,azadkia2019simple}, we are interested in understanding whether $Y$ is a measurable function of $X$, and not just if one variable is a function of the other. If one wants the second option, a natural idea would be to switch the roles of $X$ and $Y$ in $\kmac$ and take the maximum of the two resulting measures. This new measure would converge to $0$ if and only if $X$ and $Y$ are independent, and $1$ if and only if one of the variables is a function of the other. This symmetrizability feature is not available in the measure proposed in~\cite{azadkia2019simple} as it heavily relies on $Y$ being univariate.
		
	\end{enumerate}
	\subsection{Finite-sample distribution-freeness when  $\mu=\mu_X\otimes\mu_Y$} 
	Although our proposed empirical measure $\kmac$ (in~\eqref{eq:tempkmac}) extends many of the important properties of the classical correlation coefficients beyond $d_1=d_2=1$, but unlike the Spearman's correlation coefficient and Kendall's $\tau$ (and the recently proposed Chatterjee's correlation in~\cite{chatterjee2019new}) $\kmac$ is {\it not} finite-sample distribution-free when $\mu=\mu_X\otimes\mu_Y$. This leads us to the following question:
	
	{\it ``Can we find a measure of association that satisfies properties (I)-(III) and moreover, is finite-sample distribution-free when $X$ and $Y$ are independent?"}
	
	We answer the above question in the affirmative when $\X=\R^{d_1}$ and $\Y=\R^{d_2}$, with $d_1, d_2 \ge 1$, building on our work in Sections~\ref{sec:popkmac}--\ref{sec:massoceuclid}. The key observation here is that the distribution-free measures discussed above (when $d_1=d_2 = 1$) are all based on the (univariate) ``ranks'' of $Y_i$'s. In~\cref{sec:R-Q-Intro} we use the recently proposed idea of {\it multivariate ranks} based on the theory of \emph{optimal transport} (see \cite{Hallin17, del2018center, Cher17, Deb19, Shi19}) to develop measures of association that satisfy properties (I)-(III) and are also finite-sample distribution-free when $\mu=\mu_X\otimes\mu_Y$; see Sections~\ref{sec:overviewOT} and~\ref{sec:defn} for details. 
	
	
	Having defined the multivariate ranks (via optimal transport) we construct our family of distribution-free (when $\mu=\mu_X\otimes\mu_Y$) measures of association based on a very simple and classical analogy between Pearson's correlation and Spearman's correlation. Note that when $d_1=d_2=1$, Spearman's correlation is equivalent to the classical Pearson's correlation coefficient computed between the one-dimensional ranks of the $X_i$'s and the $Y_i$'s, instead of using the observations themselves. We mimic the same approach here, i.e., instead of computing $\kmac$ using the $X_i$'s and $Y_i$'s themselves, we instead use their empirical multivariate ranks. 
	
	We propose the ``rank" version of $\kmac$, namely $\rkmac$, in~\eqref{eq:rankkmac}. In~\cref{thm:Consistency}, we show that $\rkmac$ satisfies (I)-(III), is distribution-free when $\mu=\mu_X\otimes\mu_Y$, and consistently estimates a measure of dependency between $X$ and $Y$. As $\rkmac$ is based on multivariate ranks, a test for independence of $X$ and $Y$ based on $\rkmac$ will generally be more powerful against heavy-tailed alternatives and more robust to outliers and contaminations (see~\cite{Huber-2nd-E, Oja-2010, Deb19} for related discussions). Further, the corresponding test, being distribution-free, also avoids asymptotic approximations or permutation ideas for determining rejection thresholds. In~\cref{prop:Chacon}, we prove that the limit of $\rkmac$ exactly coincides with the limit of the coefficient in~\cite{azadkia2019simple} (denoted by $T_n(Y,Z)$ in their paper) when $d_2=1$. Note that, unlike $\rkmac$, the measure in~\cite{azadkia2019simple} does not have the finite sample distribution-free property.
	
	Finally,~\cref{theo:ranknullclt} proves a CLT for $\rkmac$ which is once again uniform over a large class of graphs. We would like to point out that unlike $X_i$'s and $Y_i$'s, their multivariate ranks are no longer independent among themselves which makes the CLT harder to prove. We circumvent this by proving a ``H\'ajek representation" (see~\cref{lem:hajekrep} in the Appendix; also see~\cite[Theorem 5.1]{shi2020rate}) which is a popular technique used for analyzing rank based statistics in the univariate setting. This result may be of independent interest. 
	
	
	\subsection{Related works} 
	In~\cite{dette2013copula}, the authors use the term ``measure of regression dependence" for the three properties mentioned above and show that it is possible to define such a measure satisfying (I)-(III) when $\mathcal{X}=\Y=\R$. The same population measure was rediscovered in~\cite{chatterjee2019new} where the author proposes a tuning parameter-free estimator of the same measure from empirical observations that can be computed in near linear time. Since then, the estimator in~\cite{chatterjee2019new} has attracted a lot of attention (see~\cite{shi2020power,cao2020correlations}). Further, in~\cite{azadkia2019simple}, the authors propose a similar measure when $\mathcal{X}=\R$ and $\Y=\R^{d_2}$, $d_2\geq 1$. However, all these measures crucially use the canonical ordering of $\R$ and hence do not extend to the multivariate setting (where $\mathcal{X}=\R^{d_1}$ and $\Y=\R^{d_2}$ with $d_1,d_2 \ge 2$), let alone more general topological spaces. Some multivariate measures of association satisfying (I)-(III) have been proposed in~\cite{siburg2010measure,tasena2016measures,boonmee2016measure}, following similar copula-based ideas as in~\cite{dette2013copula}; however, to the best of our knowledge, none of these papers provide a consistent empirical estimate of their proposed measures of association. 
	
	\subsection{Organization}
	The rest of our paper is organized as follows. In~\cref{sec:popkmac}, we formally introduce the population versions of our family of kernel measures of association $\eta_K(\mu)$ and study their properties. The new class of empirical kernel measures of association, i.e., KMAc ($\kmac$), is presented in full generality using the theory of geometric graphs in~\cref{sec:kmacestimate} and the flexibility in the construction of $\kmac$ is illustrated using examples of different geometric graphs in~\cref{sec:exkergraph}. In~\cref{sec:kmacclt}, we state a CLT for $\kmac$ when $X$ and $Y$ are independent, which holds uniformly over a large class of graphs.~\cref{sec:ratescon} shows that when $X$ and $Y$ are not independent, $\kmac$ converges to $\eta_K(\mu)$ at a rate which adapts to the intrinsic dimension (see~\cref{def:covernum}) of the measure $\mu_X$. Further, in~\cref{sec:linalt} we propose an estimator $\klin$ of $\eta_K(\mu)$ which is closely related to $\kmac$ but has the advantage of being computable in $\mathcal{O}(n\log{n})$ time in broad generality.~\cref{sec:massoceuclid} focuses on some additional properties of $\kmac$ and $\eta_K(\mu)$ when restricted to Euclidean spaces. In~\cref{sec:R-Q-Intro}, we introduce the multivariate rank version of $\kmac$ and describe its properties such as distribution-freeness when $\mu=\mu_X\otimes\mu_Y$, consistency, connection with correlation coefficients from~\cite{chatterjee2019new,azadkia2019simple} and asymptotic normality. 
	The supplement begins with~\cref{sec:gendis} which contains some general discussions which were deferred from the main text of the paper. All our simulation studies are featured in~\cref{sec:sim}. In~\cref{sec:pfmain}, we provide the proofs of our main results. Some technical lemmas used in our proofs are provided in~\cref{sec:techlem}, while~\cref{sec:auxres} gives some known auxiliary results from analysis, concentration of measures and CLTs.
	\section{A kernel measure of association $\eta_K$ --- the population version}\label{sec:popkmac}
	In this section, we formally present the population version of our kernel measure of association, i.e., $\eta_K$ (see~\eqref{eq:tempkac}) and show that it satisfies the desirable properties of a nonparametric measure of association (as in (II)) under certain assumptions. Let us first breakdown (II) into three explicit properties that we want the population measure $\eta_K$ to satisfy. Consider two topological spaces $\X, \Y$ equipped with Borel complete probability measures and let $\X\times \Y$ be the completion of the product space. Let $\mam(\X\times\Y)$ and $\mam(\Y)$ be the set of all Borel probability measures on $\X\times\Y$ and $\Y$ respectively. We are interested in defining a function $h:\mam(\mathcal{X}\times\Y)\to\R$ such that given a random element $(X,Y)\sim \mu\in \mam(\mathcal{X}\times\Y)$, with nondegenerate marginals $\mu_X$ and $\mu_Y$ (an assumption we make throughout the paper), the following properties hold:
	\begin{itemize}
		\item[(P1)] $h(\mu)\in [0,1]$.
		\item[(P2)] $h(\mu)=0$ if and only if $\mu=\mu_X\otimes \mu_Y$ (i.e., $X$ and $Y$ are independent).
		\item[(P3)] $h(\mu)=1$ if and only if $Y=g(X)$, $\mu$ almost everywhere (a.e.), for some measurable function $g:\mathcal{X}\to\Y$.
	\end{itemize} 
	

	%


	Let $\mu_{Y|x}$ be the regular conditional distribution of $Y$ given $X=x$ which we assume exists for all $x \in \X$. Assume that a {\it kernel} $K(\cdot,\cdot)$  --- a symmetric, nonnegative definite function on $\Y\times \Y$ --- exists on $\Y\times\Y$ and let $\hk$ denote the induced RKHS (see~\cref{def:rkhs}). Suppose that $\hk$ is separable (this can be ensured under mild conditions\footnote{For example, if $\Y$ is a separable space and $K(\cdot,\cdot)$ is continuous.}, see e.g.,~\cite[Lemma 4.33]{SVM}) and let $\langle \cdot,\cdot\rangle_{\hk}: \hk \times \hk\to \R$ and $\|\cdot \|_{\hk}$ denote the inner product and induced norm on $\hk$.

	Generate $(X',Y',\tilde{Y'})$ as follows: $X'\sim\mu_X$, $Y'|X\sim \mu_{Y|X'}$, $\tilde{Y'}|X'\sim \mu_{Y|X'}$ and $Y'\independent\tilde{Y'}|X'$. Also let $Y_1,Y_2$ be i.i.d. from $\mu_Y$. Note that $Y'$ and $\tilde{Y'}$ are dependent (via $X'$), unlike $Y_1$ and $Y_2$. Let us also recall the definition of $\eta_K$ from~\eqref{eq:tempkac}~and present it using the notation defined above:
	\begin{equation}\label{eq:kac}
	\eta_{K}(\mu) = 1-\frac{\mme\lVert K(\cdot,Y')-K(\cdot,\tilde{Y'})\rVert_{\mathcal{H}_K}^2}{\mme\lVert K(\cdot,Y_1)-K(\cdot,Y_2)\rVert_{\mathcal{H}_K}^2}.
	\end{equation}
	
	In order to ensure that $\eta_{K}(\cdot)$ is well-defined, we need certain moment assumptions. By the reproducing property of $K(\cdot,\cdot)$,
	\begin{equation}\label{eq:K-1-2}	
	\lVert K(\cdot,Y_1)-K(\cdot,Y_2)\rVert_{\mathcal{H}_{K}}^2=K(Y_1,Y_1)+K(Y_2,Y_2)-2K(Y_1,Y_2). 
	\end{equation}
	Suppose that $\mu_Y\in\tmk^1(\Y)$ where $$\tmk^{\theta}(\Y)\coloneqq \left\{\nu\in\mathcal{M}(\Y):\int_\Y K^{\theta}(y,y)\,d\nu(y)<\infty \right\}, \qquad \mbox{for } \, \theta>0.$$ 
	Then the first two terms in~\eqref{eq:K-1-2} have finite moments. The third term is also finite by an application of the Cauchy-Schwartz inequality, the reproducing property as in~\eqref{eq:hilnorm}, combined with the observation that $\tmk^1(\Y)\subset \tmk^{1/2}(\Y)$. Thus, we assume that $\mu_Y\in \tmk^1(\Y)$ in the sequel. The following result (see~\cref{pf:wdf} for a proof) presents an alternate expression of $\eta_K$.
	\begin{prop}\label{prop:wdf}
		If $\mu_Y\in\tmk^1(\Y)$, then the following relation holds:
		\begin{equation}\label{eq:MMD-kernel}\eta_K(\mu)=\frac{\mme_{\mu_X}[\mathrm{MMD}_{K}^2(\mu_{Y|X},\mu_Y)]}{\mme\lVert K(\cdot,Y)-\mme K(\cdot,Y)\rVert_{\hk}^2},
		\end{equation} where for $Q_1, Q_2 \in \tmk^{1}(\Y)$, ${\rm MMD}_K(Q_1,Q_2)$ denotes the {\it maximum mean discrepancy} (MMD) between $Q_1$ and $Q_2$ (see~\cref{defn:MMD}). Further, $\mme_{\mu_X}[\mathrm{MMD}_{K}^2(\mu_{Y|X},\mu_Y)] = \E\lVert \E[K(\cdot,Y)|X]\rVert_{\hk}^2-\E\lVert K(\cdot,Y)\rVert_{\hk}^2$.
	\end{prop}
	
	\begin{remark}\label{rem:relmmd}
		One of the most popular measures of dependence in machine learning is the Hilbert Schmidt independence criterion or HSIC (see~\cite{GrettonHSIC2005,gretton2008kernel}). In~\cite[Corollary 26]{Sejdinovic2013}, the authors proved that the HSIC of $\mu\in\mathcal{M}(\mathcal{X}\times\Y)$ is equivalent to the MMD between $\mu$ and $\mu_X\otimes\mu_Y$ (which is $0$ if and only if $\mu=\mu_X\otimes \mu_Y$). Similarly,~\cref{prop:wdf} shows that $\eta_{K}$ is also equivalent to the ``averaged" squared MMD, this time between $\mu_{Y|X}$ and $\mu_Y$ (as the denominator is only a function of the marginal of $Y$). Therefore, like HSIC, $\eta_K$ is intuitively a natural measure of dependence, but unlike HSIC, $\eta_K$ (as we will see later) does indeed characterize noiseless functional relationship.
	\end{remark}
	\begin{remark}\label{rem:Euclid}
		Consider the following kernel on $\R^d$ (for $\alpha\in (0,2)$):
		\begin{equation*}
		K(y,\tilde{y}):= {2}^{-1}\Big(\lVert y\rVert^{\alpha}_2+\lVert \tilde{y}\rVert^{\alpha}_2-\lVert y-\tilde{y}\rVert^{\alpha}_2\Big), \qquad \mbox{for all } \,y,y'\in\R^d.
		\end{equation*}
		It is the covariance function of the fractional Brownian motion with exponent $\alpha\in (0,2)$. Simple calculations show that $\eta_K$ computed using the above kernel coincides exactly with $T$ from~\eqref{eq:assocRd} for $\alpha=1$. Moreover, the kernel corresponding to $\alpha=1$ is also the distance covariance kernel, which is very popular in statistics (see~\cite{Gabor2007,Lyons13}). 
	\end{remark}
	The following result (see~\cref{pf:kacmas} for a proof) shows that $\eta_K$ is a valid measure of association. 
	\begin{theorem}\label{theo:kacmas}
		Suppose $\mu_Y\in \tmk^1(\Y)$, $\Y$ is Hausdorff\,\footnote{A topological space where for any two distinct points there exist neighborhoods of each which are disjoint from each other.} and $K(\cdot,\cdot)$ is a characteristic kernel (see~\cref{defn:Char}). Then $\eta_K(\mu)$, as defined in~\eqref{eq:kac}, satisfies $\mathrm{(P1)}$-$\mathrm{(P3)}$.
	\end{theorem}
	\begin{remark}[Examples of characteristic kernels]\label{rem:excharker}A number of popular characteristic kernels have been studied in the literature. Some popular ones in $\R^d$ include the Gaussian kernel  $K(y,\tilde{y}) :=\exp(-\lVert y-\tilde{y}\rVert^2_2)$, the Laplace kernel  $K(y,\tilde{y}) :=\exp(-\lVert y-\tilde{y}\rVert_1)$ where $\lVert\cdot\rVert_1$ denotes the standard $L^1$ norm, and the kernel in~\cref{rem:Euclid} ($\alpha\in (0,2)$); see~\cite{fukumizu2009characteristic,danafar2010characteristic} for other examples of characteristic kernels on more general topological spaces. Sufficient conditions for a kernel to be characteristic are discussed in~\cite{fukumizu2009characteristic,sriperumbudur2008injective,sriperumbudur2010,Szabo2017}.
	\end{remark}
	\begin{remark}[Why is a characteristic kernel necessary?]\label{rem:onlyneg}
		In order to show that $\eta_{K}(\cdot)$ satisfies $\mathrm{(P3)}$, we only need the map $y\mapsto K(\cdot,y)$ to be injective, which is a much weaker requirement than the kernel being characteristic (see~\cite[Proposition 14]{Sejdinovic2013}). The characteristic requirement on $K(\cdot,\cdot)$ is only necessary while establishing $\mathrm{(P2)}$.
	\end{remark}
	
	During the final stages of preparing this paper, we came across~\cite{Ke2020} where the authors present the population version $\eta_K(\mu)$ in a slightly different form (see~\cref{prop:wdf}). However their proposed estimator (which is very different from ours) does not consistently estimate $\eta_K(\mu)$ but a ``weighted" version thereof (see~\cite[Theorem 8]{Ke2020}). This weighted version does not equal $1$ if and only if $Y$ is a measurable function of $X$. Therefore, the proposed measure in~\cite{Ke2020} does not satisfy (II) as stated in the Introduction, which is the main goal of this paper. Moreover, the estimator proposed in~\cite{Ke2020} does not extend beyond $\R^{d_1+d_2}$ as it relies on kernel bandwidth selection techniques. In addition, the analysis of the population version $\eta_K(\mu)$ in~\cite{Ke2020} ignores technical issues regarding measurability and existence of regular conditional distributions, which are crucial if $\X$ and $\Y$ are non-Euclidean. Consequently, we believe that our analysis provides a more technically rigorous understanding of $\eta_K(\mu)$.
	
	\section{KMAc --- an estimate of $\eta_K$}\label{sec:kmacestimate}
	Having defined $\eta_K$ in~\eqref{eq:kac}, the next natural question is how to estimate it given empirical observations. Towards this direction, consider the setup below. Suppose $(X_1,Y_1),\ldots ,(X_n,Y_n)$ are i.i.d. $\mu$. Also assume that $\X$ is endowed with a metric $\rho$. Note that the denominator in~\eqref{eq:kac} can be estimated easily using empirical averages (from standard U-statistics theory; see~\cite[Chapter 12]{van1998}), for instance, with the following estimator: 
	$$\frac{1}{n(n-1)}\sum_{i\neq j}\lVert K(\cdot,Y_i)-K(\cdot,Y_j)\rVert_{\hk}^2.$$
	The numerator in~\eqref{eq:kac} is trickier to estimate. The main difficulty arises because of the term:
	$$\E \lVert K(\cdot,Y')-K(\cdot,\tilde{Y'})\rVert_{\hk}^2=2\E\lVert K(\cdot,Y)\rVert_{\hk}^2-2\E\big[\E (K(Y',\tilde{Y'})|X')\big].$$
	In particular, the term $\E\big[\E (K(Y',\tilde{Y'})|X')\big]$ in the above display is the hardest to estimate. To motivate our estimator, let us consider a simple case, where $X_i$'s are categorical, i.e., take values in a finite set. A natural estimate for the aforementioned term in that case would be the following:
	\begin{equation}\label{eq:targeterm}
	\frac{1}{n}\sum_{i=1}^n \frac{1}{\#\{j:X_j=X_i\}} \sum_{j:X_j=X_i} K(Y_i,Y_j).
	\end{equation}
	However for continuously distributed $X_i$'s, the inner sum in the above display is vacuous. To circumvent this difficulty, we replace the inner sum in the above display over $\{j:X_j=X_i\}$ by a sum over $\{j:X_j\mbox{ is ``close" to }X_i\}$. We formalize this idea of ``closeness" using \emph{geometric graph functionals} (as in~\cite{bbb2019}), which we describe below.
	
	$\mcg$ is a \emph{geometric graph functional} on $\mathcal{X}$ if, given any finite subset $S \subset \mathcal{X}$, $\mcg(S)$ defines a graph with vertex set $S$ and corresponding edge set, say $\mathcal{E}(\mcg(S))$. 
	Note that the graph $\mathcal{G}(S)$ can be \emph{directed/undirected}. In this paper, we will restrict ourselves to simple graphs (i.e., those without multiple edges and self loops) with no isolated vertices. Accordingly, we will often drop the qualifier geometric and simple. 
	
	Next we define $\mathcal{G}_n\coloneqq \mathcal{G}(X_1,\ldots ,X_n)$ where $\mathcal{G}$ is some graph functional on $\mathcal{X}$. Now consider the following analogue of~\eqref{eq:targeterm}:
	\begin{equation}\label{eq:analogtarget}
	\frac{1}{n}\sum_{i=1}^n\frac{1}{d_i}\sum_{j:(i,j)\in\emgn} K(Y_i,Y_j)
	\end{equation}
	where $\emgn$ denotes the set of (directed/undirected) edges of $\mathcal{G}_n$, i.e., $(i,j)\in\emgn$ if and only if there is an edge from $i\to j$ or $j\to i$ in $\mathcal{G}_n$, and $d_i$ denotes the degree of $X_i$ in $\mathcal{G}_n$. To be specific $d_i:=\sum_{j:(i,j)\in\emgn} 1$. Paralleling~\eqref{eq:targeterm}, we would like to define graph functionals for which $(i,j)\in\emgn$ implies $X_i$ and $X_j$ are ``close". 
	
	Using~\eqref{eq:analogtarget}, we can now propose our {\it kernel measure of association}, i.e., KMAc (also see~\eqref{eq:tempkmac}) in full generality:
	\begin{align}\label{eq:statest}
	\kmac:=\frac{\frac{1}{n}\sum_{i=1}^n d_i^{-1}\sum_{j:(i,j)\in\emgn} K(Y_i,Y_j)- \frac{1}{n(n-1)}\sum_{i\neq j} K(Y_i,Y_j)}{\frac{1}{2n(n-1)}\sum_{i\neq j} \lVert K(\cdot,Y_i)-K(\cdot,Y_j)\rVert_{\hk}^2}.
	\end{align}
	
	\begin{remark}[Directed graphs] Although~\eqref{eq:statest} is well-defined for both directed and undirected graphs, in the specific case of directed graphs, an alternative to~\eqref{eq:analogtarget} would be to only consider the ``outgoing" edges in the inner sum, i.e., we replace $\{j:(i,j)\in \emgn\}$ by $\{j:i\to j\in \emgn\}$. In that case, we would replace $d_i$ by the ``out-degree", i.e., $d_i^+:=\sum_{j:i\to j\in\emgn} 1$. However, all our results go through verbatim under this alternative. So we stick to the convention $\{j:(i,j)\in\emgn\}$ for the ease of exposition.
	\end{remark}
	
	The next natural question is --- ``does $\kmac$ consistently estimate $\eta_K$" as the sample size grows to infinity? We will answer this question in the affirmative under the following assumptions on the graph functional:
	
	\begin{enumerate}
		\item[(A1)] Given the graph $\mathcal{G}_n$, let $N(1),\ldots ,N(n)$ be independent random variables where $N(i)$ is a uniformly sampled index from among the neighbors of $X_i$ in $\mathcal{G}_n$. We will assume that,
		$$\rho(X_1,X_{N(1)})\overset{\P}{\longrightarrow}0 \qquad \mbox{as }n\to\infty.$$
		By exchangeability of the $X_i$'s, the above display implicitly means that ``most" neighboring vertices are ``close" (in terms of $\rho$) stochastically.
		\item[(A2)] Assume that there exists a deterministic positive sequence $r_n\geq 1$ (may or may not be bounded), such that $$\min_{1\leq i\leq n} d_i\geq r_n\qquad \mbox{w.p. }1.$$
		Let $\mathcal{G}_{n,i}$ denote the graph obtained from $\mathcal{G}_n$ by replacing $X_i$ with an i.i.d. random element $X_i'$. Assume that there exists a deterministic positive sequence $q_n$ (may or may not be bounded), such that $$\max_{1\leq i\leq n}\max\{|\mathcal{E}(\mathcal{G}_n)\setminus \mathcal{E}(\mathcal{G}_{n,i})|,|\mathcal{E}(\mathcal{G}_{n,i})\setminus \mathcal{E}(\mathcal{G}_{n})|\}\leq q_n\qquad \mbox{w.p.}~1.$$
		The above display means that there are at most $q_n$ edges in $\emgn$ that are not in $\mathcal{G}_{n,i}$ (and vice versa).
		There are trivial choices possible for $r_n$ and $q_n$, but we will assume that there exists one choice satisfying
		\begin{equation}\label{eq:locsensit}
		\limsup\limits_{n\to\infty}\frac{q_n}{r_n}\leq C,
		\end{equation}
		for some constant $C>0$. This assumption means that the graph functional is {\it local} in the sense that if one vertex of the graph is changed, then the number of edges affected is asymptotically of the same order as the degree of the vertex removed. 
		
		\item[(A3)] The final assumption states that asymptotically all vertices of $\mathcal{G}_n$ have degrees of the same order, i.e., there exists a deterministic sequence $t_n$ (may or may not be bounded) such that:
		$$\max_{1\leq i\leq n} d_i\leq t_n\quad \mbox{w.p.}~1, \qquad \mbox{and} \qquad  \limsup_{n\to\infty}\frac{t_n}{r_n}\leq C, $$ for some constant $C>0$.
	\end{enumerate}
	Similar assumptions were also used in~\cite[cf. conditions N1 and N2]{bbb2019}. We will see examples of graph functionals satisfying (A1)-(A3) in~\cref{sec:exkergraph}. The following result (see~\cref{pf:theoconsis} for a proof) shows that under mild conditions, $\kmac$ provides a consistent estimate of $\eta_K$.
	
	\begin{theorem}\label{theo:consis}
		Suppose $\mathcal{G}_n$ satisfies (A1)-(A3), $\hk$ is separable and  $\mu_Y\in\tmk^{2+\epsilon}(\Y)$ for some fixed $\epsilon>0$. Then $$\kmac\overset{\mathbb{P}}{\longrightarrow} \eta_{K}(\mu).$$ Further, if $\mu_Y\in\tmk^{4+\epsilon}(\Y)$ for some fixed $\epsilon>0$, then $\kmac\overset{a.s.}{\longrightarrow} \eta_{K}(\mu).$
	\end{theorem}
	In fact, a sub-Gaussian concentration bound can be proved for the term in~\eqref{eq:analogtarget} under additional assumptions on the kernel $K(\cdot,\cdot)$. The following result (see~\cref{pf:bdiffcon} for a proof) makes this precise.
	\begin{prop}\label{cor:bdiffcon}
		Under the same assumptions as in~\cref{theo:consis} (except assumption (A1) on $\mathcal{G}_n$) and provided $\sup_{y\in\Y} K(y,y)\leq M$ for some $M>0$, there exists a fixed positive constant $C^*$ (free of $n$ and $t$), such that for any $t>0$, the following holds:
		\begin{equation}\label{eq:bdiffcon1}
		\P\left[\bigg|\frac{1}{n}\sum_{i=1}^n \frac{1}{d_i}\sum_{j:(i,j)\in\emgn} K(Y_i,Y_j)-\E[K(Y_1,Y_{N(1)})]\bigg|\geq t\right]\leq 2\exp(-C^* nt^2),
		\end{equation}
		and consequently, 
		\begin{equation}\label{eq:bdiffcon2}
		\sqrt{n}\left(\kmac-\frac{2\E [K(Y_1,Y_{N(1)})]-2\lVert \E K(\cdot,Y_1)\rVert_{\hk}^2}{\E\lVert K(\cdot,Y_1)-\E K(\cdot,Y_1)\rVert_{\hk}^2}\right)=\mathcal{O}_{\P}(1),
		\end{equation}
		where $N(1)$ is as defined in assumption (A1).
	\end{prop}
	Observe that~\cref{cor:bdiffcon} does not use assumption (A1). Moreover, the centering in~\eqref{eq:bdiffcon2} is not exactly equal to $\eta_K$. With some elementary simplifications, it is not hard to show that if $\E [K(Y_1,Y_{N(1)})]$ in the above display was replaced by $\E\big[\E (K(Y',\tilde{Y'})|X')\big]$, then the centering in~\eqref{eq:bdiffcon2} would reduce to $\eta_K$. In fact, in the proof of~\cref{theo:consis}, assumption (A1) is only used to prove that $\E [K(Y_1,Y_{N(1)})]\to \E\big[\E (K(Y',\tilde{Y'})|X')\big]$ as $n\to\infty$.

	\subsection{Examples of graph functionals and kernels for KMAc}\label{sec:exkergraph}
	In this section, we will provide examples of graph functionals and kernels which can be used to construct $\kmac$ for estimating $\eta_K$ consistently. Let us begin with two popular graph functionals that have found widespread applications in the graph-based hypothesis testing literature (see~\cite{friedman1983,bbb2019,bbbm2019,heller2012,munmun2016,berrett2017nonparametric}):
	
	\begin{enumerate}
		\item \emph{Minimum spanning tree (MST)}: A MST is a subset of edges of an edge-weighted undirected graph which connects all the vertices with the least possible sum of edge weights and contains no cycles. For instance, in a metric (say $\rho(\cdot,\cdot)$) space, given a set of points $X_1,\ldots ,X_n$, one can construct a MST for the complete graph with vertices as $X_i$'s and edge weights $\rho(X_i,X_j)$. An example of particular interest is the Euclidean MST which is defined similarly on $\R^d$ with $\rho(x,y):= \lVert x-y\rVert_2$ for all $x,y\in\R^d$. The Euclidean MST can be computed in $\mathcal{O}(n^{2-o(d)}(\log{n})^{1-o(d)})$ time complexity (see~\cite{yao1982constructing}).
		\item \emph{Nearest neighbor graph (NNG)}: A $k$-NNG of a set of points in a metric space is a graph where each point is joined with its $k$-nearest neighbors (with respect to some metric, say $\rho(\cdot,\cdot)$). One can choose the edges to be either directed or undirected. Accordingly, the Euclidean $k$-NNG is defined on $\R^d$ with $\rho(x,y):= \lVert x-y \rVert_2$ for all $x,y\in\R^d$; and it can be computed in $\mathcal{O}(kn\log{n})$ time complexity (see~\cite{friedman1977algorithm} for details).
	\end{enumerate}
	
	The following result (see~\cref{pf:vercond} for a proof) shows that the two popular examples above satisfy (A1)-(A3).
	
	\begin{prop}\label{prop:vercond}
		(i) If $X_1,\ldots ,X_n\overset{i.i.d.}{\sim} \mu_X$, an absolutely continuous distribution on $\R^d$, then the corresponding Euclidean MST on $X_1,\ldots ,X_n$ is unique and satisfies (A1)-(A3). \newline
		(ii) If $X_1,\ldots ,X_n\overset{i.i.d.}{\sim} \mu_X$ (over $\R^d$) such that $\lVert X_1-X_2\rVert_2$ has a continuous distribution, then the $k$-NNG (both directed and undirected) is uniquely defined and satisfies (A1)-(A3), provided $k=o(n/\log{n})$.
	\end{prop}
	
	~\cref{prop:vercond} shows that, in particular, the Euclidean $1$-NNG can be used to construct $\kmac$ which leads to a consistent estimate of $\eta_K$. This is particularly appealing because the Euclidean $1$-NNG can be computed in almost linear time.
	Even when the MST and/or the $k$-NNG are not unique, it is still possible that they satisfy (A1)-(A3); see e.g.,~\cite[Lemmas 10.3, 10.4, 10.9]{azadkia2019simple} where the authors directly show that for a (possibly non-unique) $1$-NNG assumptions like (A1)-(A3) hold. However, the non-uniqueness of the underlying graph makes the analysis overly complicated, and hence we do not consider this scenario in this paper (except in~\cref{prop:energyeqfree} where we give an assumptionless analogue of~\cref{theo:consis} that holds when $\X$ and $\Y$ are Euclidean).
	\begin{remark}[Other finite dimensional spaces]\label{rem:otherfin}
		Instead of $\R^d$, let us consider MSTs or NNGs on finite (say $d$) dimensional inner product vector spaces (say $V$) over $\R$. By an elementary basis to basis mapping, all such spaces are isometrically isomorphic to $\R^d$. This means that by an isomorphically isometric embedding, any MST/NNG on $V$ yields an MST/NNG on $\R^d$, and consequently~\cref{prop:vercond} continues to hold for such spaces. 
	\end{remark}
	In~\cref{sec:gengraphcond}, we have added a discussion on our assumptions and how they relate to NNGs and MSTs constructed on general topological spaces.
	
	With the above examples of graph functionals in mind, recall that the next step in computing $\kmac$ is to construct characteristic kernels on $\Y$ which is an area of active research in the machine learning community (see~\cite{fukumizu2009characteristic,danafar2010characteristic,christmann2010universal,SVM}). Below, we present some examples of $\Y$ for which characteristic kernels have been constructed, with relevant references: (a) all separable Hilbert spaces and finite dimensional Hyperbolic spaces (see~\cite{Lyons13,Lyons2014}); (b) certain groups and semigroups, e.g., histogram valued-data/data on periodic domains which have applications in robotics, geophysics, recognizing human activities in video sequences (see~\cite{fukumizu2009characteristic,danafar2010characteristic}); (c) compact subspaces of probability measures on $\R^d$ and compactly supported $L_2$-spaces, with applications in signal processing, text and image classifications, etc. (see~\cite{christmann2010universal}). Examples of characteristic kernels on $\R^d$ were discussed in Remarks~\ref{rem:excharker}~and~\ref{rem:Euclid}.
	
	\section{A CLT for $\kmac$ when $\mu=\mu_X\otimes \mu_Y$}\label{sec:kmacclt}
	Here we provide a CLT for our estimator when $\mu=\mu_X\otimes \mu_Y$. This is particularly useful when testing for independence between $X$ and $Y$. While popular measures of dependence, such as distance covariance and MMD have a complicated (infinite mixture of chi-squares) limiting null distribution which is difficult to simulate from, crucially $\kmac$ has a simple Gaussian limit that can be made pivotal by a suitable rescaling.

	Consider the setup of~\cref{sec:kmacestimate}. In particular, let $\mathcal{G}_n$ denote the graph $\mathcal{G}(X_1,\ldots ,X_n)$ for some graph functional $\mathcal{G}_n$. Also $(d_1,\ldots d_n)$ denotes the degree sequence associated with $\mathcal{G}_n$. Finally recall the definitions of $q_n$, $t_n$ and $r_n$ from assumptions (A2) and (A3). Note that $$\sqrt{n}\kmac=\frac{N_n}{D_n} \quad \qquad \mbox{where }\;\; D_n:=\frac{1}{2n(n-1)}\sum_{i\neq j}\lVert K(\cdot,Y_i)-K(\cdot,Y_j)\rVert_{\hk}^2$$ and
	$$N_n:=\sqrt{n}\left(\frac{1}{n}\sum_{i=1}^n d_i^{-1}\sum_{j:(i,j)\in\emgn} K(Y_i,Y_j)-\frac{1}{n(n-1)}\sum_{i\neq j} K(Y_i,Y_j)\right).$$ 
	Now, $D_n$ does not involve the $X_i$'s and converges to $\E[K(Y_1,Y_1)]-\E[K(Y_1,Y_2)]$ provided $\mu_Y\in\tmk^2(\Y)$ by standard U-statistics theory. Therefore, by Slutsky's theorem, it suffices to establish a pivotal limit distribution for $N_n$ after a suitable scaling. This is the subject of the following theorem (see~\cref{pf:theonullclt} for a proof).
	
	\begin{theorem}\label{theo:nullclt}
		Assume that (A3) holds and $\E[K^2(Y_1,Y_2)]<\infty$. Then,
		\begin{align}\label{eq:d'accord}
		\mathrm{Var}(N_n)=\mathcal{O}(1).
		\end{align}	
		Further, with $\theta:=(M,D,\gamma,\epsilon)\in (0,\infty)^3\times [0,1/3)$ consider the following subclass of graph functionals and measures on $\X \times \Y$ given by:
		\begin{align*}
		\mathcal{J}_{\theta} := & \Big\{(\tilde{\mathcal{G}},\tilde{\mu}): \E [K^4(Y_1,Y_2)]\leq M,\; \limsup_{n\to\infty}\max_{1\leq i\leq n} \frac{\td_i}{(\log{n})^{\gamma}}\leq D\ w.p.~ 1,\; r_n^{-1}(q_n+t_n)\leq D, \\ &  n^{\epsilon}\mathrm{Var}(N_n)\geq 1\ \forall\ n\geq M,\; \mbox{where} \ (\td_1,\ldots ,\td_n)\; \mbox{denotes the degree sequence of } \\ & \tilde{\mathcal{G}}_n:=\tilde{\mathcal{G}}(X_1,\ldots ,X_n), \; \mbox{with} \; (X_1,Y_1),\ldots , (X_n,Y_n)\overset{i.i.d.}{\sim}\tilde{\mu}, \ \tilde{\mu}=\tilde{\mu}_X\otimes \tilde{\mu}_Y \Big\}.
		\end{align*}
		Then the following result holds for every fixed $\theta\in (0,\infty)^3\times [0,1/3)$:
		\begin{align}\label{eq:mainres}
		\lim\limits_{n\to\infty}\sup_{(\tilde{\mathcal{G}},\tilde{\mu})\in \mathcal{J}_{\theta}}\sup_{z\in\R}\Bigg|\P \left(\frac{N_n}{\hS_n}\leq z\right)-\Phi(z)\Bigg|=0,
		\end{align}
		where $\Phi(\cdot)$ is the standard Gaussian cumulative distribution function and 
		$$\hS_n^2:=\tilde{a}\left(\tilde{g}_1+\tilde{g}_3-\frac{2}{n-1}\right)+\tilde{b}\left(\tilde{g}_2-2\tilde{g}_1-2\tilde{g}_3-1+\frac{4}{n-1}\right)+\tilde{c}\left(\tilde{g}_1+\tilde{g}_3-\tilde{g}_2+\frac{n-3}{n-1}\right),$$
		with
		\begin{align*}
		&\tilde{a}:=\frac{1}{n(n-1)}\sum_{(i,j)\ \mathrm{distinct}} K^2(Y_i,Y_j),\\ &\tilde{b}:=\frac{1}{n(n-1)(n-2)}\sum_{(i,j,l)\ \mathrm{distinct}}K(Y_i,Y_j)K(Y_i,Y_l),\\ &\tilde{c}:=\frac{1}{n(n-1)(n-2)(n-3)}\sum_{(i,j,l,m)\ \mathrm{distinct}} K(Y_i,Y_j)K(Y_l,Y_m),
		\end{align*}
		and
		\begin{align*}
		\tilde{g}_1:=\frac{1}{n}\sum_{i=1}^{n} \frac{1}{\td_i}, \qquad \tilde{g}_2:=\frac{1}{n}\sum_{i,j} \frac{\tgn(i,j)}{\td_i\td_j}, \qquad \tilde{g}_3:=\frac{1}{n}\sum_{(i,j)\in\emgn} \frac{1}{\td_i\td_j}.
		\end{align*}
		Here $\tgn(i,j):=\sum_{k} \mathbf{1}((i,k)\in \mathcal{E}(\tilde{\mathcal{G}}_n))\mathbf{1}((k,j)\in \mathcal{E}(\tilde{\mathcal{G}}_n))$ denotes the number of common neighbors of vertices $X_i$ and $X_j$. 		
	\end{theorem}

	\cref{theo:nullclt} establishes the $\sqrt{n}$-consistency and CLT under $\mu=\mu_{X}\otimes \mu_Y$ for a general class of graph functionals. Crucially, the limiting distribution of $N_n/\hS_n$ (as in~\cref{theo:nullclt}) is the pivotal standard Gaussian distribution. As a result, we can construct a test for \begin{align}\label{eq:teststate}\mathrm{H}_0: \mu=\mu_X\otimes\mu_Y\qquad \mbox{versus}\qquad \mathrm{H}_1: \mu \ne \mu_X\otimes\mu_Y
	\end{align} 
	as follows: Reject $\mathrm{H}_0$ if $$N_n/\hS_n\geq z_{\alpha}$$
	where $z_{\alpha}$ is the upper $\alpha$ quantile of the standard Gaussian distribution. This test will be asymptotically level $\alpha$, by~\cref{theo:nullclt}, and also consistent (i.e., $\P_{\mathrm{H}_1}(N_n/\hS_n\geq z_{\alpha})\to 1$ as $n\to\infty$), by~\cref{theo:consis}. Due to the absence of such a simple limiting distribution theory, other testing procedures such as distance covariance~\cite{Gabor2007} and HSIC~\cite{Gretton2005} resort to computationally expensive permutation based methods to determine the rejection thresholds.
	In the following remark, we comment on some of the assumptions required for proving~\cref{theo:nullclt}.
	
	\begin{remark}[On our assumptions]\label{rem:assum}
		~\cref{theo:nullclt} assumes that the maximum degree is bounded logarithmically in $n$. On $\R^d$, for example, this condition is true for a broad class of graphs which include popular choices such as the (a) MST and (b) $k$-NNG when $k$ is bounded logarithmically in $n$ (see the proof of~\cref{prop:vercond} for details, also see~\cite{Aldous1992,Jaffe2020}). The assumption $\liminf_{n\to\infty} n^{\epsilon} \mathrm{Var}(N_n)\geq 1$ for some $\epsilon\in [0,1/3)$ is perhaps the most difficult to parse. First note that we do not need this assumption to establish $\sqrt{n}$-consistency (see~\eqref{eq:d'accord}), but only to establish a CLT. To the best of our knowledge, a comprehensive analysis of CLTs for statistics involving graph functionals has been carried out in~\cite{bbb2019}. Even there, the author effectively assumes that the limiting variance of the test statistic is strictly positive. Loosely translated to our setting, this is equivalent to assuming $\liminf\limits_{n\to\infty}\mathrm{Var}(N_n)>0$ instead of our weaker assumption $\liminf_{n\to\infty} n^{\epsilon} \mathrm{Var}(N_n)\geq 1$ for some $\epsilon\in [0,1/3)$. We believe that such assumptions are required because CLTs on graph functionals involve several implicit quantities such as expected average degrees or 2-star counts which are difficult to obtain explicitly under a very general setting (the kind we consider here). Therefore it becomes difficult to negate the possibility of certain degeneracies in the limiting variance. As we illustrate in our simulation studies in~\cref{sec:valnull} (see in particular,~\cref{tab:rateproof}), in all the examples we consider, we have observed that $\liminf\limits_{n\to\infty}\mathrm{Var}(N_n)>0$.
		
	\end{remark}
	
	One of the crucial features of~\cref{theo:nullclt} is that it establishes a CLT uniformly over a class of graph functionals. As a result, the limiting distribution continues to hold even if the graph functional is chosen suitably in a data dependent way. The following remark formalizes this idea with an example.
	
	\begin{remark}[Uniform CLT]\label{rem:unifCLT}
		Suppose that a practitioner decides to use a $d$-dimensional Euclidean $k$-NNG for constructing $\kmac$, where $k$ is chosen from the set $\{1,2,\ldots ,(\log{n})^6\}$ using some data-dependent decision rule. Note that the choice of $k$ is itself random here. In this case, there exists a constant $C_d$ such that $\max_{1\leq i\leq n}d_i\leq C_d(\log{n})^{6}$ and $r_n^{-1}(q_n+t_n)\leq C_d$ for all $n\geq 2$, w.p.~$1$, provided $\lVert X_1-X_2\rVert_2$ has a continuous distribution (see the proof of~\cref{prop:vercond} for details). Therefore, with $M$and $\epsilon$ as in~\cref{theo:nullclt}, the associated $(\mathcal{G},\mu)$ satisfies the conditions presented in $\mathcal{J}_{\theta}$ with $D=C_d$ and $\gamma=6$, w.p.~$1$. As a result, even with this random choice of $k$, the CLT continues to hold. 
	\end{remark}
	
	\begin{remark}[Efficiency]\label{rem:efficiency}
		In~\cref{theo:nullclt}, we allow the maximum degree of the graph to grow to $\infty$ at a logarithmic rate. In the context of $k$-NNGs for instance, this translates to saying that the $k$ grows to infinity logarithmically in $n$. The reason behind allowing growing $k$ is because we believe that this may yield tests which have nontrivial asymptotic Pitman efficiency (unlike the case when $k=\mathcal{O}(1)$) under some specific alternatives (see~\cite[Proposition 4.3]{bbb2019}). In fact, in~\cref{sec:powcomp} we illustrate that choosing $k$ large can often lead to a gain in power when $\kmac$ is used to test for independence (see~\cref{fig:Powerplot}). Further, we think that allowing for growing $k$ could potentially lead to information theoretically efficient estimators (see~\cite[Theorem 2]{TB2019}). While a detailed analysis of these phenomena are beyond the scope of the current paper, we plan to pursue these questions in a future work.
	\end{remark}
	\section{Rates of convergence}\label{sec:ratescon}
	This section will be devoted to establishing the rate of convergence of $\kmac$ to $\eta_K(\mu)$. By~\cref{cor:bdiffcon} (in particular,~\eqref{eq:bdiffcon2}), it is clear that this rate of convergence is chiefly governed by the rate at which $\E[K(Y_1,Y_{N(1)})]$ converges to $\E\lVert \E[K(\cdot,Y)|X]\rVert_{\hk}^2$. As it turns out this rate of convergence is heavily dependent on the underlying graph functional $\mathcal{G}$. Therefore, in order to provide interpretable results, we will focus only on the \emph{$k\equiv k_n$-NNG as the choice for $\mathcal{G}$} in this section. 
	
	In order to establish rates of convergence, we will start with the following assumptions:
	\begin{enumerate}
		\item[(R1)] $\X$ is equipped with a metric $\rho(\cdot,\cdot)$ and the $k_n$-NNG is constructed with respect to $\rho(\cdot,\cdot)$. Also $\rho(X_1,X_2)$ has a continuous distribution.
		\item[(R2)] There exists an element $\xs\in\X$, and $\alpha,C_1,C_2>0$ such that $\P\left(\rho(X_1,\xs)\geq t\right)\leq C_1\exp(-C_2t^{\alpha}).$
		\item[(R3)] Suppose $g:\X\to\hk$ be defined as 
		\begin{equation}\label{eq:g}
		g(x):=\E[K(\cdot,Y)|X=x], \qquad \forall x\in\X.
		\end{equation}
		Then there exists $\beta_1 \ge 0$, $\beta_2 \in (0,1]$, $M>0$ such that for all $x_1,x_2,\tilde{x}_1,\tilde{x}_2\in\X$, the following holds:
		\begin{align*}
		&\big|\langle g(x_1),g(x_2)\rangle_{\hk}-\langle g(\tilde{x}_1),g(\tilde{x}_2)\rangle_{\hk}\big|\leq M\left(1+\rho(x_1,\xs)^{\beta_1}+\rho(x_2,\xs)^{\beta_1}\right.\\ &\qquad \qquad  \;\;\;\left.+\,\rho(\tilde{x}_1,\xs)^{\beta_1}+\rho(\tilde{x}_2,\xs)^{\beta_1} \right) \left(\rho(x_1,\tilde{x}_1)^{\beta_2}+\rho(x_2,\tilde{x}_2)^{\beta_2}\right).
		\end{align*}
	\end{enumerate} 
	Assumptions (R1) and (R2) are mainly for technical convenience. In particular, $\rho(X_1,X_2)$ being continuous (by (R1)) ensures that the $k_n$-NNG is defined uniquely. This assumption can be avoided if the $k_n$-NNG is defined using a tie-breaking scheme (discussed in~\cref{sec:energyfree}). We also believe that the metric $\rho(\cdot,\cdot)$ can be replaced with some measure of ``similarity" although this is not a direction we pursue here. Assumption (R2) can be viewed as a ``tail bound" on $X_1$. Assumption (R3) is potentially the most crucial one. It captures the sensitivity of $\E[K(\cdot,Y)|X=x]$ as $x$ varies. Here $\beta_2$ is like the ``Lipschitz'' exponent for the function $m(\cdot,\cdot):=\langle g(\cdot),g(\cdot)\rangle_{\hk}$. Similar assumptions were used in analyzing nearest neighbor based estimators in~\cite[Theorem 4.1]{azadkia2019simple}~and~\cite[Lemma 4]{dasgupta2014optimal}. We believe that without any assumptions on this sensitivity, the rate of convergence can be arbitrarily slow (a similar conjecture was also made in~\cite[Section 4]{azadkia2019simple}).
	
	With these assumptions in mind, the next step in obtaining rates of convergence is to understand the complexity of the support of $\mu_X$. Here we capture this geometry using the notion of {\it $\mu_X$-covering numbers}, as defined below.
	
	\begin{defn}[$\mu_X$-covering number]\label{def:covernum}
		Fix $\delta\in [0,1]$, $r>0$ and $\tX\subseteq \X$. Let $N\equiv N(\mu_X,\tX,r,\delta)$ be the smallest number of balls $\B_1,\ldots ,\B_N$, each of diameter\footnote{The diameter of a set $A$ equals $\sup_{x_1,x_2\in A} \rho(x_1,x_2)$.}$\leq r$ such that the center of $\B_i$ belongs to $\tX$ for each $i$ and  $\mu_X(\tX\setminus(\cup_{i=1}^N \B_i))\leq \delta$. Then $N(\mu_X,\tX,r,\delta)$ and the collection $\B_1,\ldots ,\B_N$ will be called the $\mu_X$-covering number and  cover respectively.
	\end{defn}
	This way of defining covering numbers is motivated from~\cite[Definition 3.3.1]{chen2018explaining}. In fact,~\cite[Definition 3.3.1]{chen2018explaining} can be recovered from~\cref{def:covernum} by choosing $\tX=\X$. With the above notion of complexity in mind, we are now in a position to present our rate of convergence result (see~\cref{pf:rateres} for a proof). The theorem below is presented in a very general and consequently, a rather abstract form. In the sequel to the theorem (see~\cref{sec:adapintrin}), we will show how this theorem can be used to obtain more interpretable results.
	\begin{theorem}\label{theo:rateres}
		Assume that $\E[K^{\theta}(Y_1,Y_1)]<\infty$ for $\theta>2$, $k_n=o(n/\log{n})$, $\max_{1\leq i\leq n} d_i\leq t_n$ for a deterministic sequence $\{t_n\}_{n\geq 1}$ and  $\mathrm{(R1)}$-$\mathrm{(R3)}$ are satisfied. For a large constant $C>0$ (free of $n$), define $\tX_n\equiv B(\xs,C(\log{n})^{1/\alpha}):=\{x \in \X:\rho(x,\xs)\leq C(\log{n})^{1/\alpha}\}$,  $\epsilon_n:=\inf\{\epsilon >0:(k_n\log{n}/n) N(\mu_X,\tX_n,\epsilon,Ck_n\log{n}/n)\leq 1\}$,
		\begin{eqnarray*}
			\nu_{1,n} & := & \frac{k_n\log{n}}{n}\int_{\epsilon_n}^{2C(\log{n})^{1/\alpha}} \epsilon^{2\beta_2-1}N(\mu_X,\tX_n,\epsilon,Ck_n\log{n}/n)\,d\epsilon, \\
			\nu_{2,n} & := & \inf_{\gamma>0}\left(\gamma^2+n\int_{\gamma}^{\infty} t\P(K(Y_1,Y_1)\geq t)\,dt\right).
		\end{eqnarray*}
		Let $\kmac^{\mathrm{num}}$ denote the numerator of $\kmac$ as in~\eqref{eq:statest}. Then the following holds:
		\begin{align}\label{eq:brate1}
		\E\Big[\kmac^{\mathrm{num}}-\E \lVert g(X)\rVert_{\hk}^2 +\E K(Y_1,Y_2)\Big]^2\lesssim \epsilon_n^{2\beta_2}+\nu_{1,n}+\frac{t_n^2\nu_{2,n}}{nk_n^2}+\frac{1}{n},
		\end{align}
		and consequently, we also have:
		\begin{align}\label{eq:brate2}
		|\kmac-\eta_K(\mu)|=\mathcal{O}_{\P}\left(\epsilon_n^{\beta_2}+\sqrt{\nu_{1,n}}+\frac{t_n}{k_n}\sqrt{\frac{\nu_{2,n}}{n}}+\frac{1}{\sqrt{n}}\right).
		\end{align}
	\end{theorem}
	Note that~\eqref{eq:brate1} provides a finite sample moment bound for the numerator of $\kmac$, appropriately centered. By standard U-statistics theory, a similar bound can be obtained for the denominator as well. However, we believe that under the assumptions of~\cref{theo:rateres}, these conclusions do not yield finite sample error bounds for $\kmac$ as we have a ratio of these quantities. Consequently,~\eqref{eq:brate2} provides an error bound in probability.
	
	\subsection{Adaptation to intrinsic dimensionality}\label{sec:adapintrin}
	Let us now use~\cref{theo:rateres} to show that the rate of convergence of $\kmac$ \emph{adapts to the intrinsic dimensionality of $X$}. For ease of exposition, consider $\X=\R^{d_1}$, $x^*=0 \in \R^{d_1}$ and $\rho(\cdot,\cdot)=\lVert \cdot-\cdot\rVert_2$. Also suppose $\mu_X$ is \emph{supported on a $d_0$-dimensional subset} of $\R^{d_1}$ (with $d_0\leq d_1$) in the sense that: for any ball $B(t):=\{x:\lVert x\rVert_2\leq t\}$, $$N(\mu_X,B(t),\epsilon,0)\leq C(t/\epsilon)^{d_0}$$ for some constant $C>0$; here $d_0$ can be fractional. Note that the above concept generalizes the notion of dimensionality when $\mu_X$ is supported on a $d_0$-dimensional hyperplane (manifold) in $\R^{d_1}$ with $d_0 \le d_1$. In this setting, simple computations reveal that $t_n/k_n\lesssim 1$ (by~\cite[Lemma 1]{Jaffe2020}), $\epsilon_n=\left(\frac{k_n\log{n}}{n}\right)^{1/d_0}(\log{n})^{1/\alpha}$, and
	$$\nu_{1,n}=\begin{cases}\frac{k_n}{n}(\log{n})^{1+2\beta_2/\alpha} & \mbox{if } d_0<2\beta_2, \\ \frac{k_n}{n}(\log{n})^{2+d_0/\alpha} & \mbox{if } d_0=2\beta_2, \\ \left(\frac{k_n}{n}\right)^{2\beta_2/d_0}(\log{n})^{1+2\beta_2/\alpha} & \mbox{if } d_0>2\beta_2.\end{cases}$$
	Next assume that $K(Y_1,Y_1)$ has sub-exponential tails, i.e., $\P(K(Y_1,Y_1)\geq t)\lesssim \exp(-C_3t)$ for some constant $C_3>0$. Under this assumption, $\nu_{2,n}\lesssim (\log{n})^2$. Combining all these observations, the following corollary is an immediate consequence of~\cref{theo:rateres}.
	\begin{corollary}\label{cor:rateuclid}
		Define 
		$$\nu_{n}:=\begin{cases}\frac{(\log{n})^2}{n}+\frac{k_n}{n}(\log{n})^{2\beta_2/d_0+2\beta_2/\alpha} & \mbox{if } d_0<2\beta_2, \\ \frac{(\log{n})^2}{n}+\frac{k_n}{n}(\log{n})^{2+d_0/\alpha} & \mbox{if } d_0=2\beta_2, \\ \frac{(\log{n})^2}{n}+\left(\frac{k_n}{n}\right)^{2\beta_2/d_0}(\log{n})^{2\beta_2/d_0+2\beta_2/\alpha} & \mbox{if } d_0>2\beta_2.\end{cases}$$
		Thus, under the same assumptions as in~\cref{theo:rateres}, we have: $\E\left[\kmac^{\mathrm{num}}-\E \lVert g(X)\rVert_{\hk}^2 +\E[K(Y_1,Y_2)]\right]^2\lesssim \nu_n$ and $|\kmac-\eta_K(\mu)|=\mathcal{O}_{\P}(\sqrt{\nu_n})$.
	\end{corollary}
	\cref{cor:rateuclid} is in spirit similar to~\cite[Theorem 4.1]{azadkia2019simple}, except that~\cref{cor:rateuclid} additionally shows that the rates adapt to the intrinsic dimensionality of $\mu_X$. Note that for~\cref{cor:rateuclid} to hold $\Y$ need not be Euclidean.

	\begin{remark}[On the choice of $k_n$]\label{rem:choosek}
		At first glance,~\cref{cor:rateuclid} seems to suggest that $k_n$ should always be chosen as small as possible. However a closer look reveals a different picture. In~\cref{cor:rateuclid}, the $(\log{n})^2/n$ term comes from the variance of $\kmac^{\mathrm{num}}$ whereas the other terms arise from its bias. Now suppose that $X$ and $Y$ are independent. In that case, the bias of $\kmac^{\mathrm{num}}$ is exactly $0$ no matter what $k_n$ is. This in turn implies that the bound from~\cref{cor:rateuclid} reduces to $\mathcal{O}((\log{n})^2/n)$ irrespective of $k_n$. Crucially, the constants involved in the $\mathcal{O}((\log{n})^2/n)$ bound tend to get better as $k_n$ increases (see~\cref{rem:choosekagn} for some computational evidence of this phenomenon). Therefore when $X$ and $Y$ are independent, increasing $k_n$ can lead to a tangible decrease in the asymptotic variance of $\kmac^{\mathrm{num}}$. Additionally, in~\cref{rem:efficiency}, we point out as to why we believe that there is a gain to be had by choosing $k_n$ larger when testing for independence. The choice of $k_n$ should therefore be informed by the application at hand.  
	\end{remark}
	
	\subsection{When does (R3) hold?}\label{sec:R3hold}
	Now we will present some simple and easily verifiable conditions under which assumption (R3) holds. The subsequent result is similar to~\cite[Proposition 4.2]{azadkia2019simple}. Therefore, we will leave the details of the proof to the reader.
	
	\begin{prop}\label{prop:suffcond}
		Suppose $\X=\R^{d_1}$ and $\Y=\R^{d_2}$. Assume that the conditional density of $Y|X=x$, say $f(\cdot|x)$ exists, is non-zero everywhere in its support, differentiable with respect to $x$ (for every $y$) and for all $1\leq i\leq d_1$, the function $|(\partial/\partial x_i)\log f(y|x)|$ is bounded above by a polynomial in $\lVert y\rVert_2$ and $\lVert x\rVert_2$. Next suppose that for any compact set $K\subset \R^{d_1}$ the function $m(y):=\max_{x\in K} f(y|x)$ is bounded in $y$ and decays faster than any negative power of $\lVert y\rVert_2$ as $\lVert y\rVert_2\to\infty$. Lastly assume that for any $k\geq 1$, $\E \left[K(Y_1,Y_2)\lVert Y_1\rVert_2^{2k}\lVert Y_2\rVert_2^{2k}|X_1=x_1,X_2=x_2 \right]$ is bounded above by a polynomial in $\lVert x_1\rVert_2$ and $\lVert x_2\rVert_2$. Then assumption (R3) is satisfied with some $\beta_1, M> 0$ and $\beta_2=1$.
	\end{prop}
	
	The crucial message from~\cref{prop:suffcond} is that (R3) holds provided $f(\cdot|x)$ is a smooth function of $\lVert x\rVert_2$ (the other conditions are just to ensure that $\mu_{Y|x}$ is sufficiently light-tailed). In fact, the existence of a conditional density is not required, and one can just as easily replace the conditional density with conditional probability mass function if $Y|X=x$ has a discrete distribution for $\mu_X$-a.e.~$x$ and the support of $Y|X=x$ does not depend on $x$.

	\section{Near linear time estimators of $\eta_K(\mu)$}\label{sec:linalt}
	The computational complexity of $\kmac$ depends on the graph functional we consider and the choice of the kernel $K(\cdot,\cdot)$. In Section~\ref{sec:compute}, we discuss the computational complexity of $\kmac$ and show that for certain kernels (and the $k$-NNG with $k = \mathcal{O}(1)$) $\kmac$ can be computed in $\mathcal{O}(n \log n)$ time. In Section~\ref{sec:linstat} we introduce $\klin$ --- an alternative to $\kmac$ --- which can indeed be computed in $\mathcal{O}(n \log n)$ time for {\it any} kernel $K(\cdot,\cdot)$ (and certain graph functionals).
	
	\subsection{Computation time for KMAc $\kmac$}\label{sec:compute}
	The expression of $\kmac$ (in~\eqref{eq:statest}) reveals that computing it involves computation of the following three terms:
	$$(a)\, \sum_{i=1}^n \frac{1}{d_i} \sum_{j:(i,j)\in\emgn} K(Y_i,Y_j), \qquad (b)\, \sum_{i\neq j} K(Y_i,Y_j),  \qquad  (c)\, \sum_{i=1}^n K(Y_i,Y_i).$$
	Clearly, (c) can be computed in $\mathcal{O}(n)$ time. For (a), we need to compute the graph $\mathcal{G}_n$ (suppose the corresponding time complexity is $\mathcal{O}(g_n)$) first and then compute a sum over $|\emgn|$ summands. For (b), note that $$\sum_{i\ne j} K(Y_i,Y_j) = \Big\| \sum_{i=1}^n K(Y_i,\cdot) \Big\|^2_{\h_K} -  \sum_{i=1}^n K(Y_i,Y_i).$$ The second term in the above display can be computed in $\mathcal{O}(n)$ time. By denoting the time complexity for the first term above by $v_n$, we immediately get the following proposition.
	\begin{prop}\label{prop:timecomp}
		$\kmac$, as defined in~\eqref{eq:statest}, can be computed in $\mathcal{O}(n+v_n+g_n+|\emgn|)$ time complexity.
	\end{prop} 
	When $\X$ is ``Euclidean-like" (see~\cref{rem:otherfin}), $g_n$ and $|\emgn|$ are both $\mathcal{O}(n\log{n})$ when $k$-NNGs are used with $k=\mathcal{O}(1)$ (see~\cref{sec:exkergraph}). For finite dimensional RKHSs $\| \sum_{i=1}^n K(Y_i,\cdot) \|^2_{\h_K}$ can be computed in time $\mathcal{O}(n)$. Thus, in that case, the time complexity of $\kmac$ is $\mathcal{O}(n\log{n})$.

	However, in general, it may not be possible to compute $\| \sum_{i=1}^n K(Y_i,\cdot) \|^2_{\h_K}$ easily, and especially in linear time, unless the underlying kernel is assumed to have special structure; see~\cref{rem:morecomp} for a detailed discussion. Note that if the term (c) is computed by a naive averaging, the computational complexity of $\kmac$ becomes (at least) $\mathcal{O}(n^2)$. 
	
	\subsection{$\klin$ --- a near linear time analogue of $\kmac$}\label{sec:linstat}
	We alleviate the computational complexity in computing $\kmac$, discussed in~\cref{sec:compute}, by replacing the term $(n(n-1))^{-1}\sum_{i\neq j} K(Y_i,Y_j)$ in~\eqref{eq:statest} with $n^{-1}\sum_{i=1}^n K(Y_i,Y_{i+1})$ where $Y_{n+1}\equiv Y_1$. The new, easily computable version of $\kmac$ is then defined as 
	\begin{align}\label{eq:fastkmac}
	\klin:=\frac{\frac{1}{n}\sum_{i=1}^n d_i^{-1}\sum_{j:(i,j)\in\emgn} K(Y_i,Y_j)- \frac{1}{n}\sum_{i=1}^n K(Y_i,Y_{i+1})}{\frac{1}{n}\sum_{i=1}^n K(Y_i,Y_i)-\frac{1}{n}\sum_{i=1}^n K(Y_i,Y_{i+1})}.
	\end{align}
	By~\cref{prop:timecomp}, $\klin$ can now be \emph{computed in $\mathcal{O}(n\log{n})$ time complexity} for all kernels and for certain graph functionals, which include the $k$-NNGs. Further, the next result shows that $\klin$ satisfies all the nice properties of $\kmac$; see~\cref{rem:Compare} for another motivation of the estimator $\klin$. 
	
	\begin{prop}[Consistency and rates]\label{prop:ratepconsis}
		$\klin$ satisfies~\cref{theo:consis},~\cref{theo:rateres}~and~\cref{cor:rateuclid} under the same exact assumptions as $\kmac$.
	\end{prop}
	The proof of~\cref{prop:ratepconsis} is trivial and follows the proofs of~\cref{theo:consis},~\cref{theo:rateres}~and~\cref{cor:rateuclid}; we leave the details to the reader.
	
	The following proposition (see~\cref{pf:nullcltlesscomp} for a proof) shows that, like $\kmac$, $\klin$ too is $\sqrt{n}$-consistent when $X$ and $Y$ are independent, and moreover satisfies a uniform CLT. As in~\cref{theo:nullclt}, we will only focus on establishing a CLT for the (scaled) numerator of $\klin$. Define 
	$$\linn:=\sqrt{n}\left(\frac{1}{n}\sum_{i=1}^n d_i^{-1}\sum_{j:(i,j)\in\emgn} K(Y_i,Y_j)-\frac{1}{n}\sum_{i=1}^n K(Y_i,Y_{i+1})\right).$$
	\begin{prop}[Uniform CLT for $\linn$]\label{prop:nullcltesscomp}
		Suppose that Assumption (A3) holds and $\E K^2(Y_1,Y_2)<\infty$. Then, $\limsup_{n\to\infty}\mathrm{Var}(\linn)<\infty.$ Set $\theta:=(M,D,\gamma,\epsilon)\in (0,\infty)^3\times [0,1/3)$ and consider the same subclass of graph functionals and measures on $\X \times \Y$ as in~\cref{theo:nullclt}, i.e., $\mathcal{J}_{\theta}$ with $N_n$ replaced by $\linn$. Next, define:
		\begin{align*}
		\tilde{a}^{\mathrm{lin}}:=\frac{1}{n}\sum_{i=1}^n & K^2(Y_i,Y_{i+1}), \qquad \tilde{b}^{\mathrm{lin}}:=\frac{1}{n}\sum_{i=1}^n K(Y_i,Y_{i+1})K(Y_{i+1},Y_{i+2})\\ & \tilde{c}^{\mathrm{lin}}:=\frac{1}{n}\sum_{i=1}^n K(Y_i,Y_{i+1})K(Y_{i+2},Y_{i+3}).
		\end{align*}
		In the above display, all indices are taken modulo $n$; so $Y_1\equiv Y_{n+1}$, $Y_{2}\equiv Y_{n+2}$, $Y_3\equiv Y_{n+3}$, etc. Recall the definitions of $\tilde{g}_1$, $\tilde{g}_2$ and $\tilde{g}_3$ from~\cref{theo:nullclt} and define:
		$$\hS_{n,\mathrm{lin}}^2:=\tilde{a}^{\mathrm{lin}}(\tilde{g}_1+\tilde{g}_3+1)+\tilde{b}^{\mathrm{lin}}(\tilde{g}_2-2\tilde{g}_1-2\tilde{g}_3-3)+\tilde{c}^{\mathrm{lin}}(2+\tilde{g}_1+\tilde{g}_3-\tilde{g}_2).$$
		Then the following result holds for every fixed $\theta$:
		\begin{align}\label{eq:mainreslesscomp}
		\lim\limits_{n\to\infty}\sup_{(\tilde{\mathcal{G}},\tilde{\mu})\in \mathcal{J}_{\theta}}\sup_{z\in\R}\Bigg|\P \left(\frac{\linn}{\hS_{n,\mathrm{lin}}}\leq z\right)-\Phi(z)\Bigg|=0.
		\end{align}
	\end{prop}
	
	Therefore, $\linn$ has the same rate of convergence as $N_n$ (from~\cref{theo:nullclt}) when $X$ and $Y$ are independent, despite being much faster to compute. Moreover by~\cref{prop:ratepconsis}, $\klin$ shares the same statistical properties of $\eta_K(\mu)$ --- consistency and rate adaptivity. This is in sharp contrast to the computationally faster versions of the popular HSIC or distance covariance measures that lose out on certain theoretical aspects when compared to their original versions. For instance, the proposal in~\cite[Section 3.2, Equation 18]{Zhang2018} has a slower rate of convergence when $X$ and $Y$ are independent, whereas in~\cite{jitkrittum2017adaptive}, the standard notion of consistency is replaced by a.s.~consistency.
	
	Certainly there is a price to pay if $\klin$ is used instead of $\kmac$. Firstly, unlike $\kmac$, $\klin$ is no longer permutation invariant. Moreover the asymptotic variance (when $X$ and $Y$ are independent) for $\linn$ is in general larger than that of $N_n$. 
	
	A few other remarks on $\klin$ are now in order.
	\begin{remark}[Comparison with other correlation coefficients]\label{rem:compare}
		Combining~\cref{prop:timecomp} with the discussions in~\cref{sec:exkergraph}, we observe that $\klin$ can be computed in $\mathcal{O}(n\log{n})$ time using $k$-NNGs when $k=\mathcal{O}(1)$. This is the same time complexity as required for computing classical correlations such as Spearman's correlation and the Pearson's correlation coefficient (up to an additional logarithmic factor) and also the measures of association  in~\cite{chatterjee2019new,azadkia2019simple}. None of these other measures satisfy (II) from the Introduction on Euclidean spaces with general dimensions $d_1,d_2>1$ and cannot be computed for more general spaces. Computing $\klin$ with NNGs is also faster than distance covariance (see~\cite{Gabor2007}) or HSIC (see~\cite{gretton2008kernel}), which proceed with $\mathcal{O}(n^2)$ time complexity, and also do not guarantee (II). In fact, computing $\klin$ with Euclidean MST is also strictly faster than $\mathcal{O}(n^2)$. Finally, we would like to point out that Euclidean MSTs and NNGs can be constructed readily in \texttt{R} using packages like \texttt{emstreeR} and \texttt{RANN} respectively, which in turn implies that $\klin$ can be computed very easily in \texttt{R}.
	\end{remark}
	\begin{remark}[Linear time scaling]\label{rem:linscale}
		While~\cref{theo:nullclt} gives a pivotal distribution for $\kmac$, the computation of $\hS_n$, or more specifically, the terms $\tilde{a}$, $\tilde{b}$ and $\tilde{c}$, takes $\mathcal{O}(n^2)$ time (the other terms can be computed in near linear time if the $k$-NNG is used). In contrast,~\cref{prop:nullcltesscomp} yields a uniform CLT where the scaling, i.e., $\hS_{n,\mathrm{lin}}$, can also be computed in near linear time. 
	\end{remark}
	\begin{remark}[A permutation test]\label{rem:perm}
		When testing for independence between $X$ and $Y$, one could do a permutation test based on $\klin$. In this case, if $B$ permutations are used for determining rejection thresholds, then the computational complexity of the procedure is $\mathcal{O}(Bn\log{n})$. The corresponding test will be exactly level $\alpha$ and also consistent against fixed alternatives.
	\end{remark}
	We would like to conclude this section by noting that an alternate approach to estimating $\eta_K(\mu)$ would be to use cross-covariance operators and conditional mean embeddings directly to estimate the alternative form of $\eta_K(\mu)$ stated in~\eqref{eq:MMD-kernel}; this is presently being pursued in~\cite{Zhen2020}.

	
	
	
	\section{Measure of association on $\R^{d_1+d_2}$}\label{sec:massoceuclid}
	So far in the paper, we have dealt with the case where $\X$ and $\Y$ are quite general topological spaces. In this section, we will take a more streamlined approach and focus on the case where $\X\subseteq \R^{d_1}$ and $\Y\subseteq\R^{d_2}$, which is of great practical importance. The goal of this section is to leverage the structure of $\R^d$ and go beyond the properties (P1)-(P3) from~\cref{sec:popkmac}. 

	\subsection{Connection with Pearson's correlation}\label{sec:conspear}
	For this subsection, we will restrict ourselves to the following class of measures: 
	\begin{equation}\label{eq:propmeas}
	T_{\alpha}(\mu)\coloneqq 1-\frac{\mme\lVert Y'-\tilde{Y'}\rVert_2^{\alpha}}{\mme\lVert Y_1-Y_2\rVert_2^{\alpha}}
	\end{equation}
	where $\alpha\in (0,2]$, $\mu_Y$ has finite $\alpha$-th moments to ensure that the right hand side of~\eqref{eq:propmeas} is well-defined and $Y',\tilde{Y'},Y_1,Y_2$ are defined as in~\eqref{eq:assocRd}. Clearly, $T$ from~\eqref{eq:assocRd} is a special case of $T_{\alpha}$ (with $\alpha=1$). Easy simplifications show that $T_{\alpha}$ is the same as $\eta_K$ with $K(\cdot,\cdot)$ defined as in~\cref{rem:Euclid}. The following result (see~\cref{pf:spGauss} for a proof) demonstrates the connection between $T_{\alpha}$ and Pearson's correlation coefficient for different choices of $\alpha$.
	\begin{lemma}\label{prop:spGauss}
		Suppose $(X,Y)\sim\mu_{\rho}$ where $\mu_{\rho}$, $\rho\in [-1,1]$, is the bivariate normal distribution with mean vector $(\theta_1,\theta_2)$ and covariance matrix $\Sigma$ with $\Sigma_{11}=\sigma_1^2$, $\Sigma_{22}=\sigma_2^2$ and $\Sigma_{12}=\Sigma_{21}=\rho\sigma_1\sigma_2$. Then we have:
		\begin{itemize}
			\item[(a)] $T_{\alpha}(\mu_{\rho})$ is a continuous and strictly increasing function in $|\rho|$, where $T_{\alpha}(\mu_{\rho})=0$ if and only if $\rho=0$ and $T_{\alpha}(\mu_{\rho})=1$ if and only if $|\rho|=1$, for all $\alpha\in (0,2]$.
			\item[(b)] For $\alpha=1$, we have the following:
			\begin{enumerate}    			
				\item $T_1(\mu_{\rho})=1-\sqrt{1-\rho^2}$.
				\item $\sqrt{T_1(\mu_{\rho})}\leq|\rho|$.
				\item $\inf_{|\rho|\neq 0}\sqrt{T_1(\mu_{\rho})}/|\rho|=\lim_{|\rho|\to 0}\sqrt{T_1(\mu_{\rho})}/|\rho|=2^{-1/2}\approx 0.71$.
			\end{enumerate}
			\item[(c)] For $\alpha=2$, $T_2(\mu_{\rho})=\rho^2$.
		\end{itemize}
	\end{lemma}
	The above proposition shows that $T_{\alpha}(\cdot)$ varies continuously and monotonically between $0$ and $1$ as $\rho$ varies in the bivariate normal setting. This is different from our discussions in the previous sections, where we only focused on conditions for which $T_{\alpha}(\cdot)$ equals $0$ or $1$, but not on the intermediate values between $0$ and $1$. It must be pointed out that unlike $T_{\alpha}(\cdot)$ for $0<\alpha<2$, $T_2(\cdot)$ does not satisfy (P2) from~\cref{sec:popkmac} in general (more on this in~\cref{rem:notalpha2} in the Appendix). We would also like to highlight part (b) of~\cref{prop:spGauss} which provides a complete characterization of $T_1(\mu_{\rho})$. In fact, it shows that $\sqrt{T_1(\mu_{\rho})}$ is ``close" to the absolute correlation, i.e., $|\rho|$, in the sense that $2^{-1/2}\leq \sqrt{T_1(\mu_{\rho})}/|\rho|\leq 1$ for all $\rho\in [-1,1]$. Part (b) of~\cref{prop:spGauss} can be compared with~\cite[Theorem 7]{Gabor2007} where a similar set of properties were established for the well-known distance correlation.
	
	\subsection{Invariance, continuity and equitability}\label{sec:ice}
	In addition to the properties already described in (P1)-(P3), two other properties have been advocated for dependence measures in~\cite{Mori2019} --- namely \emph{invariance} and \emph{continuity}. Informally, \emph{invariance} means that the measure should be unaffected under a ``suitable'' class of transformations and \emph{continuity} means that whenever a sequence of measures $\mu_n$ converges to $\mu$ (in an ``appropriate" sense), the sequence of values of the dependence measure for $\mu_n$'s should converge to that of $\mu$. We show that $\eta_K$ satisfies these properties for a class of kernels.
	
	We will restrict to kernels having the following form: $K(y,\tilde{y})=h_1(y)+h_2(\tilde{y})+h_3(y-\tilde{y})$, for $y,\tilde{y}\in\R^{d_2}$, and $h_i:\R^{d_2}\to\R$ for $i=1,2,3$ are continuous functions (the Gaussian kernel from~\cref{rem:excharker} and the class of kernels from~\cref{rem:Euclid} are both of this form). For such a $K(\cdot,\cdot)$, it is easy to check that:
	\begin{equation}\label{eq:revpopkmac}
	\eta_K(\mu)=\frac{\E [h_3(Y'-\tilde{Y'})]-\E [h_3(Y_1-Y_2)]}{h_3(0)- \E [h_3(Y_1-Y_2)]}.
	\end{equation}
	The following proposition demonstrates the invariance and continuity of $\eta_K$. The proof is simple and we leave the details to the reader.
	\begin{prop}\label{prop:otprop}
		Consider $\eta_K$ as in~\eqref{eq:revpopkmac}.
		\begin{enumerate} 
			\item \emph{Invariance}. Suppose $h_3(\cdot)=\tilde{h}(\lVert \cdot\rVert_2)$ for some measurable function $\tilde{h}:[0,\infty) \to\R$. Let $A_2$ be a $d_2\times d_2$ orthogonal matrix, $b_2$ a vector of size $d_2$, and $f_1:\R^{d_1}\to\R^{d_1}$ an invertible function. Then $\eta_K$ is invariant under the transformation $(X,Y)\mapsto (f_1(X),A_2Y+b_2)$. 
			\item \emph{Continuity}. Suppose $(X_n,Y_n)\sim \mu_n$ and $(X,Y)\sim\mu$. Generate $(X_n',Y_n',\tilde{Y}_n')$ and $(X',Y',\tilde{Y}')$ as in~\eqref{eq:assocRd} from $\mu_n$ and $\mu$ respectively. Assume $(X_n',Y_n',\tilde{Y}_n')$ converges weakly to $(X,Y',\tilde{Y}')$\footnote{The condition ``$(X_n',Y_n',\tilde{Y}_n')$ converges weakly to $(X',Y',\tilde{Y}')$" most often follows from the assumption that $\mu_n\overset{w}{\longrightarrow}\mu$. However this is not true in general (see~\cref{rem:conconv}).}, and $\limsup\limits_{n\to\infty} \E [h_3^{1+\epsilon}(Y_n'-\tilde{Y}_n')]<\infty$ for some $\epsilon>0$. Then $\eta_K(\mu_n)\to \eta_K(\mu)$.
		\end{enumerate}
	\end{prop}
	
	We now move on to the notion of \emph{equitability} as proposed in~\cite{Reshef2016,Reshef11}. While \emph{equitability} lacks a rigorous mathematical definition, intuitively, it means that a equitable measure should be similar for equally noisy relationships. Let us illustrate this with a simple example based on a regression setting:
	\begin{align*}
	&\mu_1: Y_1=f_1(X_1)+\epsilon_1;  f_1:\R^{d_1}\to\R^{d_2}, X_1\sim \mu_{1,X}, \epsilon_1\sim\mathcal{N}\left(0,\sigma_1^2I_{d_2}\right), X_1\indep\epsilon_1,\\ & \mu_2: Y_2=f_2(X_2)+\epsilon_2;  f_2:\R^{d_1}\to\R^{d_2}, X_2\sim \mu_{2,X}, \epsilon_2\sim\mathcal{N}\left(0,\sigma_2^2I_{d_2}\right), X_2\indep\epsilon_2.
	\end{align*}
	Here $\mu_1$ and $\mu_2$ denote the joint distributions of $(X_1,Y_1)$ and $(X_2,Y_2)$ respectively. Let the marginals for $Y_1$ and $Y_2$ be $\mu_{1,Y}$ and $\mu_{2,Y}$ respectively. In the above setting, provided $\sigma_1=\sigma_2$, an equitable measure should give the same value for $\mu_1$ and $\mu_2$ irrespective of how different $f_1(\cdot)$ and $f_2(\cdot)$ are.
	
	Let us investigate the equitability of $\eta_K$. To adjust for the denominator of $\eta_K$, we will further assume that $\mu_{1,Y}=\mu_{2,Y}$. Based on~\eqref{eq:revpopkmac}, we only need to deal with the term $\E h_3(Y'-\tilde{Y'})$. It is easy to check that:
	$$\frac{\eta_K(\mu_1)-1}{\eta_K(\mu_2)-1}=\frac{\E [h_3(Z_1)]-h_3(0)}{\E [h_3(Z_2)]-h_3(0)}$$
	where $Z_1\sim\mathcal{N}(0,2\sigma_1^2I_{d_2})$ and $Z_2\sim\mathcal{N}(0,2\sigma_2^2I_{d_2})$. Therefore, from the above display, it is clear that whenever $\sigma_1=\sigma_2$, we have $\eta_K(\mu_1)=\eta_K(\mu_2)$. This implies that $\eta_K$ is equitable in this very simple setting. 
	
	Next suppose that $\sigma_1<\sigma_2$. Also assume that $\E [h_3(Z)]$, where $Z\sim\mathcal{N}(0,\sigma^2)$, is a strictly decreasing function of $\sigma^2$ (this is the case when $K(\cdot,\cdot)$ is the Gaussian or the Laplacian kernel from~\cref{rem:excharker}, and the class of distance kernels from~\cref{rem:Euclid}). Then, the above display implies that $\eta_K(\mu_1)>\eta_K(\mu_2)$. In other words, $\eta_K(\cdot)$ tells us that $Y_1$ and $X_1$ are more strongly associated that $Y_2$ and $X_2$, which agrees with our intuition as the relationship between $Y_1$ and $X_1$ is less noisy.

	\subsection{An assumptionless analogue of~\cref{prop:energyeq}}\label{sec:energyfree}
	The goal of this section is to provide an analogue of~\cref{theo:consis} under no assumptions on $\mu$ (supported on a subset of $\R^{d_1+d_2}$). Under this convention, recall the definition of $\kmac$ from~\eqref{eq:statest}. We will first specify a particular choice of the graph functional --- the $1$-NNG, i.e., for any vertex $X_i$, an edge is drawn between $X_i$ and its nearest neighbor, with ties broken at random. This is the same choice as used in~\cite{azadkia2019simple}. Next we need to specify a kernel. For this, we will resort to bounded kernels on $\R^{d_2}$ (see~\cref{rem:excharker} for examples). We are now in a position to present an assumption-free analogue of~\cref{theo:consis} (see~\cref{pf:energyeqfree} for a proof).
	
	\begin{prop}\label{prop:energyeqfree}
		Suppose $(X_1,Y_1),\ldots ,(X_n,Y_n)\overset{i.i.d.}{\sim} \mu$ with $\mu$, $\mu_Y$ and $\mu_X$ bearing their usual meanings. Also let $Y'$, $\tilde{Y'}$ be drawn as in~\cref{prop:energyeq}. Assume that $K(\cdot,\cdot)$ is a bounded, continuous, characteristic kernel and the graph functional used in constructing $\kmac$ is as described in the preceding paragraph. Then the following holds:
		$$\kmac\overset{a.s.}{\longrightarrow} \frac{\E[ K(Y_1,Y_2)]-\E [K(Y',\tilde{Y'})]}{\E [K(Y_1,Y_1)]-\E[ K(Y_1,Y_2)]}.$$
		The right hand side of the above display also satisfies $\mathrm{(P1)}$-$\mathrm{(P3)}$. 
	\end{prop}
	
	\section{A multivariate rank based measure of association}\label{sec:R-Q-Intro}
	In the previous sections, we presented a measure of association, i.e., $\kmac$, which provides a nonparametric analogue of the classical correlation coefficients. However, there is one interesting property of Spearman's correlation (or even Kendall's $\tau$) that we are yet to emulate --- the property of \emph{distribution-freeness} when $\mu=\mu_X\otimes \mu_Y$ (provided the distributions are continuous). In this section we develop a new class of measures of association, building on our work in Sections~\ref{sec:popkmac}--\ref{sec:massoceuclid}, which have this distribution-free property. In this section we will restrict our attention to the case when $\X=\R^{d_1}$ and $\Y=\R^{d_2}$, for $d_1,d_2 \ge 1$.

	The main tool in this development will be the theory of optimal transport (see~\cite{Hallin17, del2018center, Cher17, ghosal2019multivariate}). In particular, following the line of recent works in~\cite{Deb19,Shi19}, we will use a notion of {\it multivariate ranks} defined via optimal transport, that will aid us in constructing measures of association having the desired distribution-free property. The fact that we use ``ranks'' is motivated by the fact that when $d_1=d_2=1$ it is indeed the univariate ranks that lead to the distribution-free property of Spearman's correlation and Kendall's $\tau$.
	
	\subsection{A brief overview of optimal transport}\label{sec:overviewOT}
	
	Let $\mathcal{P}(\R^d)$ denote the space of probability measures on $\R^d$ ($d \ge 1$) and $\mathcal{P}_{ac}(\R^d)$ be the space of absolutely continuous probability measures on $\R^d$. For a function $F:\R^d\to\R^d$, we will use $F\#\mu$ to denote the push forward measure of $\mu$ under $F$, i.e., the distribution of $F(Z)$ when $Z\sim\mu$.
	
	Below we present perhaps the simplest version of the optimal transport problem (courtesy the works of Gaspard Monge in 1781, see~\cite{monge1781memoire}):
	\begin{align}\label{eq:Mongeproblem}
	\inf_{F} \int  \lVert z-F(z)\rVert^2 \,d\mu(z)\qquad \mbox{subject to}\quad F\#\mu=\nu.
	\end{align}
	A minimizer of~\eqref{eq:Mongeproblem}, if it exists, is referred to as an {\it optimal transport map}. An important result in this field, known as the {\it Brenier-McCann theorem}, takes a ``geometric" approach to the problem of optimal transport (as opposed to the analytical approach presented in~\eqref{eq:Mongeproblem}) and will be very useful to us in the sequel; see e.g.,~\cite[Theorem 2.12 and Corollary 2.30]{Villani2003}.
	
	\begin{prop}[Brenier-McCann  theorem~\cite{Mccann1995}]\label{prop:Mccan}
		Suppose that $\mu, \nu\in\mathcal{P}_{ac}(\R^d)$. Then there exists functions $R(\cdot)$ and $Q(\cdot)$ (usually referred to as ``optimal transport maps"), both of which are gradients of (extended) real-valued $d$-variate convex functions, such that: (i) $R\# \mu=\nu$, $Q\# \nu=\mu$; (ii) $R$ and $Q$ are unique ($\mu$ and $\nu$-a.e.~respectively); (iii) $R\circ Q (u)=u$ ($\nu$-a.e.~$u$) and $Q\circ R(z)=z$ ($\mu$-a.e.~$z$). Moreover, if $\mu$ and $\nu$ have finite second moments, $R(\cdot)$ is also the solution to Monge's problem in~\eqref{eq:Mongeproblem}.
	\end{prop}
	In~\cref{prop:Mccan}, by ``gradient of a convex function" we essentially mean a function from $\mathbb{R}^d$ to $\mathbb{R}^d$ which is $\mu$ (or $\nu$) a.e.~equal to the gradient of some convex function. It is instructive to note that, when $d=1$, the standard $1$-dimensional distribution function $F$ associated with a distribution $\mu$ is nondecreasing and hence the gradient of a convex function. Therefore when $d=1$, $F$ is the optimal transport map from $\mu$ to $\mathcal{U}[0,1]$ (i.e., the Uniform([0,1]) distribution) by~\cref{prop:Mccan}.
	
	\subsection{Multivariate ranks defined via optimal transport}\label{sec:defn}
	\begin{defn}[Population multivariate ranks and quantiles]\label{def:popquanrank}
		Set $\nu:=\mathcal{U}[0,1]^d$ --- the uniform distribution on $[0,1]^d$. Given $\mu\in\mathcal{P}_{ac}(\mathbb{R}^d)$, the corresponding population rank and quantile maps are defined as the optimal transport maps $R(\cdot)$ and $Q(\cdot)$ respectively as in~\cref{prop:Mccan}. These maps are unique a.e.~with respect to $\mu$ and $\nu$ respectively. 
	\end{defn}
	
	In standard statistical applications, the population rank map is not available to the practitioner. In fact, the only accessible information about $\mu$ comes in the form of empirical observations $Z_1,Z_2,\ldots ,Z_n\overset{i.i.d.}{\sim}\mu\in\mathcal{P}_{ac}(\mathbb{R}^d)$. In order to estimate the population rank map from these observations, let us denote
	\begin{equation}\label{eq:H_n^d}
	\mathbf{\mathcal{H}}_n^d:=\{h_1^d,\ldots ,h_n^d\}
	\end{equation}
	to be a set of $n$ vectors in $[0,1]^d$. We would like the points in $\mathbf{\mathcal{H}}_n^d$ to be ``uniform-like", i.e., their empirical distribution, $n^{-1}\sum_{i=1}^n \delta_{h_i^d}$ should converge weakly to $\nu$. In practice, for $d=1$, we may take $\mathbf{\mathcal{H}}_n^d$ to be the usual $\{i/n\}_{1\leq i\leq n}$ sequence and for $d\geq 2$, we may take it as a quasi-Monte Carlo sequence (such as the $d$-dimensional Halton sequence) of size $n$ (see~\cite{Hofer2009,Hofer2010} for details) or a random draw of $n$ i.i.d. random variables from $\nu$. The empirical distribution on $\mathbf{\mathcal{H}}_n^d$ will serve as a discrete approximation of $\nu$. We are now in a position to define the empirical multivariate rank which will proceed via a discrete analogue of problem~\eqref{eq:Mongeproblem}.
	
	\begin{defn}[Empirical rank map]\label{def:empquanrank}
		Let $S_n$ denote the set of all $n!$ permutations of $\{1,2,\ldots ,n\}$. Consider the following optimization problem:  
		\begin{align}\label{eq:empopt}
		\mathbf{\hat{\sigma}}_n\coloneqq \argmin_{\mathbf{\sigma}\in S_n} \sum_{i=1}^n \lVert X_{i}-h_{{\sigma(i)}}^d\rVert^2.
		\end{align}  
		Note that $\mathbf{\hat{\sigma}}_n$ is a.s.~uniquely defined (for each $n$) as $\mu\in\mathcal{P}_{ac}(\R^d)$. The empirical ranks are then defined as:
		\begin{equation}\label{eq:R_n_h}
		\hat{R}_n(X_i)=h_{\hat{\sigma}_n(i)}^d, \qquad \mathrm{for \;\;} i = 1,\ldots, n.
		\end{equation}
	\end{defn}
	The optimization problem in~\eqref{eq:empopt} is combinatorial in nature, but it can be solved exactly in polynomial time (with worst case complexity $\mathcal{O}(n^3)$) using the Hungarian algorithm (see~\cite{munkres1957,bertsekas1988} for details). For a comprehensive list of faster (approximate) algorithms, see~\cite[Section 5]{Shi19}. Moreover, when $d=1$, if we choose $\mathbf{\mathcal{H}}_n^1= \{i/n\}_{i=1}^{n}$, then the empirical ranks, i.e., $\hat{R}_n(X_i)$'s are exactly equal to the usual one-dimensional ranks.
	
	A crucial reason behind defining empirical ranks as in~\cref{def:empquanrank} is that due to the exchangeability of the $X_i$'s, the vector of ranks, i.e., $(\hat{R}_n(X_1),\ldots ,\hat{R}_n(X_n))$ is uniformly distributed over the following set:
	$$\{(h^d_{\sigma(1)},\ldots ,h^d_{\sigma(n)}):\sigma\in S_n\}.$$
	This lends a distribution-free property to the empirical ranks (see~\cite[Proposition 2.2]{Deb19} for a formal statement and proof). In other words, the distribution of $(\hat{R}_n(X_1),\ldots ,\hat{R}_n(X_n))$ is free of $\mu\in\mathcal{P}_{ac}(\R^d)$. Moreover, the empirical ranks are maximal ancillary (see e.g.,~\cite{del2018center}).
	
	
	\subsection{A distribution-free measure of association and its properties}
	
	We will construct a measure of association with the distribution-free property based on a simple and classical analogy between Pearson's and Spearman's correlation. Note that when $d_1=d_2=1$,  Spearman's correlation is equivalent to the classical Pearson's correlation coefficient computed between the one-dimensional ranks of the $X_i$'s and the $Y_i$'s, instead of the actual observations. It is this usage of one-dimensional ranks that lends Spearman's correlation the distribution-free property when $\mu=\mu_X\otimes\mu_Y$. We will mimic the same approach in this section, i.e., instead of computing $\kmac$ using the $X_i$'s and $Y_i$'s themselves, we will instead use their empirical multivariate ranks. 
	
	Let us briefly recall the setting. We have $(X_1,Y_1),\ldots ,(X_n,Y_n)\overset{i.i.d.}{\sim}\mu\in \mathcal{P}_{ac}(\R^{d_1+d_2})$ with marginals $\mu_X$ and $\mu_Y$, such that $\mu_X\in\mathcal{P}_{ac}(\R^{d_1})$ and $\mu_Y\in \mathcal{P}_{ac}(\R^{d_2})$. Let $\mathbf{\mathcal{H}}_n^{d_1}$ and $\mathbf{\mathcal{H}}_n^{d_2}$ be two sets of $n$ ``uniform-like" points in dimension $d_1$ and $d_2$ respectively. We can then use~\cref{def:empquanrank} to define the empirical multivariate rank vectors $(\hat{R}_n^X(X_1),\ldots ,\hat{R}_n^X(X_n))$ and $(\hat{R}_n^Y(Y_1),\ldots ,\hat{R}_n^Y(Y_n))$ based on $\mathbf{\mathcal{H}}_n^{d_1}$ and $\mathbf{\mathcal{H}}_n^{d_2}$ respectively. Next, given a graph functional $\mathcal{G}$, let $\rgn:=\mathcal{G}(\hat{R}_n^X(X_1),\ldots ,\hat{R}_n^X(X_n))$ and $\rmgn$ be the set of edges of $\rgn$. Also, with a slight notational abuse, we will still use $(d_1,\ldots ,d_n)$ to be the degree sequence of the vertices in $\rgn$. The rank version of $\kmac$ is then defined as follows:
	\begin{equation}\label{eq:rankkmac}
	\rkmac:=\frac{n^{-1}\sum_{i}d_i^{-1}\sum_{j:(i,j)\in\rmgn} K(\hat{R}_n^Y(Y_i),\hat{R}_n^Y(Y_j))-F_n}{(2n(n-1))^{-1}\sum_{i,j,i\neq j} \lVert K(\cdot,\hat{R}_n^Y(Y_i))-K(\cdot,\hat{R}_n^Y(Y_j))\rVert_{\hk}^2},
	\end{equation}
	where $F_n :=(n(n-1))^{-1}\sum_{i\neq j} K(\hat{R}_n^Y(Y_i),\hat{R}_n^Y(Y_j))$. 
	Note that $\rkmac$ is the same as $\kmac$ with the observations replaced with their empirical ranks. In the following theorem (see~\cref{pf:Consistency} for a proof), we show that $\rkmac$ is a measure of association which additionally has a pivotal distribution when $\mu=\mu_X\otimes\mu_Y$.
	\begin{theorem}\label{thm:Consistency}
		\emph{(a)} When $\mu=\mu_X\otimes\mu_Y$, $\rkmac$ has a pivotal distribution, i.e., the distribution does not depend on $\mu_X$ and $\mu_Y$.
		
		\emph{(b)} Assume that $K(\cdot,\cdot):\R^{d_2}\times\R^{d_2}\to\R$ is continuous and let $R^X(\cdot)$ and $R^Y(\cdot)$ denote the population rank maps for $\mu_X$ and $\mu_Y$ respectively and suppose that $r(x_1,x_2):=\mathbb{E}\big[K(R^Y(Y_1), R^Y(Y_2))|R^X(X_1)=x_1, R^X(X_2)=x_2\big]$ is uniformly $\beta$-H{\"o}lder continuous in $x_1,x_2\in [0,1]^{d_1}$ for some $\beta\in (0,1]$, i.e., given any $x_1,x_2,\tilde{x}_1,\tilde{x}_2\in [0,1]^{d_1}$, there exists a constant $C$ (free of $x_1,x_2,\tilde{x}_1,\tilde{x}_2$) such that:
		\begin{align}\label{eq:artifact}
		\big|r(x_1,x_2)-r(\tilde{x}_1,\tilde{x}_2)\big|\leq C\left(\lVert x_1-\tilde{x}_1\rVert_2^{\beta}+\lVert x_2-\tilde{x}_2\rVert_2^{\beta}\right).
		\end{align}
		Additionally we assume that $\rgn$ satisfies Assumption (A3) with some $r_n$, $t_n$ and $\frac{1}{nr_n}\sum_{e\in \rmgn}|e|^{\beta} \to 0$ as $n\to \infty$ where $|e|$ denotes the edge length of $e\in\rmgn$ and the empirical distributions on $\mathbf{\mathcal{H}}_n^{d_1}$ and $\mathbf{\mathcal{H}}_n^{d_2}$ converge weakly to $\mathcal{U}[0,1]^{d_1}$ and $\mathcal{U}[0,1]^{d_2}$ respectively. Under these assumptions, 
		$$\rkmac\overset{\mathbb{P}}{\longrightarrow} 1-\frac{\E\lVert K(\cdot,R^Y(Y'))-K(\cdot,R^Y(\tilde{Y'}))\rVert_{\hk}^2}{\E\lVert K(\cdot,R^Y(Y_1))-K(\cdot,R^Y(Y_2))\rVert_{\hk}^2}:=\rpk,$$
		where $(Y_1,Y_2,Y',\tilde{Y'})$ are defined as in~\eqref{eq:assocRd}. Further, if $K(\cdot,\cdot)$ is characteristic, then $\rpk$ satisfies $\mathrm{(P1)}$-$\mathrm{(P3)}$ from~\cref{sec:popkmac}.
	\end{theorem}
	\begin{remark}\label{rem:detgraph}
		Given the set $\mathbf{\mathcal{H}}_n^{d_1}$, the set of empirical multivariate ranks, i.e., $(\hat{R}_n^X(X_1),\ldots ,\hat{R}_n^X(X_n))$ is some permutation of $\mathbf{\mathcal{H}}_n^{d_1}$. As $\mathcal{G}$ is a geometric graph functional, note that $\max_{1\leq i\leq n} d_i$ and $\min_{1\leq i\leq n} d_i$ are both deterministic quantities. So Assumption (A3) (a requirement for~\cref{thm:Consistency}) is a deterministic condition, i.e., there is no need to view it as a probabilistic constraint. For the same reason, the condition $\frac{1}{nr_n}\sum_{e\in \rmgn}|e|^{\beta} \to 0$ is also a deterministic one instead of the probabilistic condition imposed in Assumption (A1) in~\cref{sec:kmacestimate}. 
	\end{remark}
	\begin{remark}\label{rem:rateassume}
		There are two important differences between our assumptions in~\cref{theo:consis}~and~\cref{thm:Consistency}. Firstly, the assumption $\frac{1}{nr_n}\sum_{e\in \rmgn}|e|^{\beta} \to 0$ is essentially a slightly stronger reformulation of (A1) used in~\cref{theo:consis}. Secondly, the assumption~\eqref{eq:artifact} is an artifact of our proof technique and we expect~\cref{thm:Consistency} to hold in more generality.
	\end{remark}

	~\cref{thm:Consistency} shows that $\rkmac$ can be used to construct a test for independence (in addition to being a measure of association) which will be consistent and exactly distribution-free under the null hypothesis of independence.
	\subsubsection{Connection with the correlation coefficient proposed in~\cite{dette2013copula},~\cite{chatterjee2019new} and~\cite{azadkia2019simple}}\label{sec:concha}
	Here we show that $\rpk$ (which arises naturally as the limit of $\rkmac$ in~\cref{thm:Consistency}) is exactly the same as the following population correlation coefficient proposed in~\cite[Equation 1.2]{chatterjee2019new} and~\cite[Equation 2.1]{azadkia2019simple} where the authors only consider the case $d_2=1$ (also see~\cite{dette2013copula}):
	$$\xi(\mu):=\frac{\int \mathrm{Var}(\mathbb{P}(Y\geq t|X))\,d\mu_Y(t)}{\int \mathrm{Var}(\mathbf{1}(Y\geq t))\,d\mu_Y(t)}$$ for a particular choice of a characteristic kernel. The choice turns out to be the kernel from~\cref{rem:Euclid} with $\alpha=1$. Incidentally this is also the kernel used in the construction of distance covariance (see~\cite{Gabor2007}). We state the result formally in the proposition below (see~\cref{pf:Chacon} for a proof).
	\begin{prop}\label{prop:Chacon}
		When $d_2=1$ and $K(y_1,y_2)=|y_1|+|y_2|-|y_1-y_2|$, then $\rpk=\xi(\mu)$.
	\end{prop}
	The above proposition shows that the family of measures of association proposed in this paper includes that of~\cite{chatterjee2019new} and extends it significantly to general $d_1,d_2\geq 1$ and also a large class of kernel functions. 
	
	\subsubsection{A rank CLT for $\rkmac$}\label{sec:rankclt}
	We present a CLT for $\rkmac$ when $\mu=\mu_X\otimes\mu_Y$. Earlier in~\cref{theo:nullclt} we presented a CLT for $\kmac$. While $\rkmac$ and $\kmac$ have a similar form, there is a crucial difference in their construction. Note that $\kmac$ is based on the $Y_i$'s (respectively $X_i$'s) which are independent among themselves, whereas $\rkmac$ is a function of $\hat{R}_n(Y_i)$'s (respectively $\hat{R}_n(X_i)$'s) which are no longer independent among themselves. As a result, a different technique is required to prove the corresponding CLT. Informally speaking, the technique used in this paper is that of a H\'ajek representation (as in~\cite[Theorem 5.1]{shi2020rate}), where we show that the empirical multivariate ranks can be replaced by their population counterparts in $\rkmac$ at a $o_{\mathbb{P}}(1/\sqrt{n})$ cost. As $\rkmac$ does not have a standard $U$-statistic representation, we do not use the explicit form of H\'ajek projections as in~\cite[Theorem 5.1]{shi2020rate}, but opt for a more hands-on method of moments based approach. The details of this result have been relegated to the Appendix (see~\cref{lem:hajekrep}) to avoid notational clutter and we only present the statement of the resulting uniform CLT below (see~\cref{pf:ranknullclt} for a proof). Using a similar argument as in~\cref{sec:kmacclt} and recalling the definition of $F_n$ after~\eqref{eq:rankkmac}, we will only study
	\begin{equation}\label{eq:recallnrk}
	\nrk:=\sqrt{n}\left(\frac{1}{n}\sum_{i=1}^n\frac{1}{d_i}^{-1}\sum_{j:(i,j)\in\rmgn} K\big(\hat{R}_n^Y(Y_i),\hat{R}_n^Y(Y_j)\big)-F_n\right).\end{equation}
	\begin{theorem}\label{theo:ranknullclt}
		If (A3) holds, $K(\cdot,\cdot)$ is continuous, $\mu=\mu_X\otimes\mu_Y$, $\mu_X\in\mathcal{P}_{ac}(\R^{d_1})$ and $\mu_Y\in\mathcal{P}_{ac}(\R^{d_2})$, then
		$\mathrm{Var}(\nrk)=\mathcal{O}(1)$. Further, with $\theta:=(D,\gamma,\epsilon)\in (0,\infty)^3$ consider the following subclass of graph functionals and measures on $\X \times \Y$ given by:
		\begin{align*}
		&\mathcal{J}_{\theta} :=\Big\{\tilde{\mathcal{G}}:  \limsup_{n\to\infty}\max_{1\leq i\leq n} \frac{\td_i}{(\log{n})^{\gamma}}\leq D,\; r_n^{-1}t_n\leq D, \; \mathrm{Var}(\nrk)\geq\epsilon\ \forall n\geq D , \;\mathrm{ where} \\ 
		&  \qquad (\td_1,\ldots ,\td_n)\; \mathrm{denotes \; the \; degree \; sequence \; of \;}  \trgn:=\tilde{\mathcal{G}}(\hat{R}_n^X(X_1),\ldots ,\hat{R}_n^X(X_n))\Big\}.
		\end{align*}
		Then the following result holds for every fixed $\theta\in (0,\infty)^3$:
		\begin{align}\label{eq:rankmainres}
		\lim\limits_{n\to\infty}\sup_{\tilde{\mathcal{G}}\in \mathcal{J}_{\theta}}\sup_{z\in\R}\Bigg|\P \left(\frac{\nrk}{\hS_n}\leq z\right)-\Phi(z)\Bigg|=0,
		\end{align}
		where $\tilde{S}_n$ is defined as in~\cref{theo:nullclt} with $\tilde{a}:=\E \left[K^2\big(R^Y(Y_1),R^Y(Y_2)\big)\right]$, $\tilde{b}:=\E \left[K\big(R^Y(Y_1),R^Y(Y_2)\big) \, K \big(R^Y(Y_1),R^Y(Y_3)\big) \right]$ and $\tilde{c}:=\left(\E \big[K\big(R^Y(Y_1), \right.$ $\left.R^Y(Y_2)\big)\big]\right)^2$, and $\tilde{g}_1$, $\tilde{g}_2$, $\tilde{g}_3$ are as in~\cref{theo:nullclt} with $\tilde{\mathcal{G}}_n$ replaced by $\trgn$.
		
	\end{theorem}
	It is worth noting that $\tilde{S}_n^2$ can be computed explicitly as $R^Y(Y_1),R^Y(Y_2), $ $R^Y(Y_3)\overset{i.i.d.}{\sim}\mathcal{U}[0,1]^{d_2}$.
	
	\section*{Acknowledgements} The authors would like to thank Huang Zhen for his careful reading of the paper and a number of insightful suggestions; and also Bhaswar Bhattacharya, Sourav Chatterjee, Holger Dette, Arthur Gretton, Marc Hallin, and Fang Han for helpful comments.
	
	\bibliographystyle{chicago}
	\bibliography{OT,template,References}
	\newpage
	\appendix
	\section{Some general discussions}\label{sec:gendis}
	In this section, we will elaborate on some parts of the main text which were initially deferred so as not to impede the flow of the paper. 
	\subsection{Intuition behind the construction of $T_n$}\label{sec:kmacintuit}
	We would like to provide some intuition as to why $T_n$ is a natural candidate that satisfies (II). Towards this direction, let us consider $k=\mathcal{O}(1)$ and focus on the term $$\underbrace{n^{-1}\sum_{i=1}^n d_i^{-1}\sum_{j\in \mathcal{N}_i} \lVert Y_i-Y_j\rVert_2}_{S_n}$$ which is the only term in $T_n$ that involves both $X_i$'s and $Y_i$'s. It can be shown (as in the proof of~\cref{theo:consis}) that $S_n$ concentrates around its expectation. 
	
	If $X$ and $Y$ are independent, it is easy to see that for any $i,j$, $j\in \mathcal{N}_i$, $\E \lVert Y_i-Y_j\rVert_2=\E \lVert Y_1-Y_2\rVert_2 >0$; and consequently (as $\min_{1\leq i\leq n} d_i\geq k$), 
	\begin{equation}\label{eq:E_S_n}
	\E S_n \geq \frac{n}{k} \, \E \lVert Y_1-Y_2\rVert_2.
	\end{equation}
	Next, suppose $Y=g(X)$ for some measurable function $g:\R^{d_1}\to\R^d_2$. To fix ideas, further assume that $g(\cdot)$ is continuous (this can be justified using Lusin's Theorem~\cite{lusin1912proprietes}; alternatively see~\cref{prop:lusin}). Now, as $X_i$ and $X_j$, $j\in \mathcal{N}_i$, are two neighboring vertices in the $k$-NNG, it seems reasonable to expect that $\lVert X_i-X_j\rVert_2$ would be stochastically ``small". By continuity of $g(\cdot)$, the same can be said for $\lVert Y_i-Y_j\rVert_2=\lVert g(X_i)-g(X_j)\rVert_2$. Further, by~\cite[Lemma 1]{Jaffe2020}, $\max_{1\leq i\leq n}|\N_i|\leq C_{d_1}$ for some constant $C_{d_1}$. As a result, $\E\left[\sum_{j\in \mathcal{N}_i}d_i^{-1}\lVert Y_i-Y_j\rVert_2\right]=o(1)$. By summing over $i$ from $1$ to $n$, we get $\E S_n=o(n)$. Therefore, comparing with~\eqref{eq:E_S_n}, there seems to be a clear distinction in the stochastic ``size" of $S_n$ depending on whether $X$ and $Y$ are independent, or $Y$ is a noiseless function of $X$. Through a suitable scaling, it is the above difference in asymptotic behaviors that we intend to capture in the definition of $T_n$.

	\subsection{RKHS: Some preliminaries}\label{sec:rkhsprelim}
	In this subsection we formally define some concepts from the theory of RKHS that is be used repeatedly in the paper. We start with the basic definition of a RKHS.
	
	\begin{defn}[Reproducing kernel Hilbert space (RKHS)]\label{def:rkhs}
		Let $\mathcal{H}$ be a Hilbert space of real-valued functions defined on a topological space $\Y$ with inner product $\langle \cdot,\cdot\rangle_{\mathcal{H}}: \mathcal{H} \times \mathcal{H} \to \R$. A function $K:\Y\times\Y\to\R$ is called a \emph{reproducing kernel} if the following two conditions hold:
		\begin{enumerate}
			\item For all $y\in\Y$, $K(\cdot,y)\in\mathcal{H}$.
			\item For all $y\in\Y$ and $f\in\mathcal{H}$, $\langle f,K(\cdot,y)\rangle_{\mathcal{H}}=f(y)$.
		\end{enumerate} 
		If $\mathcal{H}$ admits a reproducing kernel, then it is termed as a RKHS.
	\end{defn}
	By the Moore-Aronszajn Theorem (see e.g.~\cite[Theorem 3]{Berlinet2004}) a symmetric, nonnegative definite kernel function $K(\cdot,\cdot)$ on $\Y\times \Y$ can be identified uniquely with a unique RKHS of real-valued functions on $\Y$ for which $K(\cdot,\cdot)$ is the reproducing kernel. Let us denote this RKHS by $\mathcal{H}_K$.
	
	The map $y\mapsto K(\cdot,y)$ from $\Y$ to $\hk$ is often called the \emph{feature map}. Further, the reproducing property as stated in~\cref{def:rkhs} implies that 
	\begin{equation}\label{eq:hilnorm}
	\langle K(\cdot,y),K(\cdot,\tilde{y})\rangle_{\hk}=K(y,\tilde{y}), \qquad \mbox{for all}\;y,\tilde{y} \in \Y.
	\end{equation}
	In the following we define three concepts that will be crucial in defining $\eta_K$ satisfying properties (P1)-(P3). Suppose that $Y \sim \mu_Y$ has a probability distribution on $\Y$ such that $\E[\sqrt{K(Y,Y)}] < \infty$. Let us also assume that $\hk$ is separable (this can be ensured under mild conditions\footnote{For example, if $\Y$ is a separable space and $K(\cdot,\cdot)$ is continuous.}, see e.g.,~\cite[Lemma 4.33]{SVM}).
	
	\begin{defn}[Mean embedding]\label{def:meanembed}
		Define the following class of probability measures on $\Y$: $$\tmk^{\theta}(\Y)\coloneqq \left\{\nu\in\mathcal{M}(\Y):\int_\Y K^{\theta}(y,y)\,d\nu(y)<\infty \right\}, \qquad \mbox{for } \, \theta>0.$$ Let $\mu_Y\in\tmk^{1/2}(\Y)$. Then the (kernel) \emph{mean embedding} of $\mu_Y$ into $\hk$ is given by $m_K(\mu_Y)\in \hk$ such that 
		\begin{equation}\label{eq:kmeanembed}
		\langle f,m_K(\mu_Y)\rangle_{\hk} = \int_\Y f(y)\,d\mu_Y(y), \qquad \mbox{for all } \, f\in\hk.
		\end{equation}
	\end{defn}
	In fact, one can write $m_K(\mu_Y)=\int K(\cdot,y)\,d\mu_Y(y)=\E_{\mu_Y} [K(\cdot,Y)]$. It is well-defined as a consequence of the Riesz representation theorem,~see~\cite{smola2007hilbert,akhiezer1993} (equivalently also by Bochner's theorem, see~\cite{Diestel1974,Dinculeanu2011}). The map $\mu_Y\mapsto m_K(\mu_Y)$ with domain $\tmk^{1/2}(\Y)$ can be viewed as a natural extension of the map $y\mapsto K(\cdot,y)$ with domain $\Y$.
	
	In a similar vein we can also define the (kernel) \emph{mean embeddings} of conditional distributions. Towards this direction, consider another topological space $\X$ and suppose $(X,Y)\sim \mu\in \mathcal{M}(\X\times\Y)$ where $\mu$ admits a regular conditional distribution $\mu_{Y|x}$ --- the conditional distribution of $Y$ given $X = x$; existence of regular conditional distributions can be guaranteed under mild conditions (see~\cite{Faden1985} for a survey). For $X=x$, the (kernel) \emph{conditional mean embedding} of $\mu_{Y|x}$ is defined as an element of $\hk$ in the same way as~\eqref{eq:kmeanembed} (also see~\cite[Section 4.1.1]{muandet2017kernel}). In other words, we can also write $m_K(\mu_{Y|x})=\int K(\cdot,y)d\mu_{Y|x}(y)=\E_{\mu_{Y|x}}[K(\cdot,Y)]$.

	\begin{defn}[Maximum mean discrepancy]\label{defn:MMD}
		The difference between two probability distributions $Q_1$ and $Q_2$ in $\tmk^{1}(\Y)$ can then be conveniently measured by $${\rm MMD}_K(Q_1,Q_2) := \|m_K(Q_1) - m_K(Q_2)\|_{\hk}$$ (here $m_K(Q_i)$ is the mean element of $Q_i$, for $i=1,2$) which is called the {\it maximum mean discrepancy} (MMD) between $Q_1$ and $Q_2$ (see~\cite[Definition 10]{gretton2008}). The following alternative representation of the squared MMD is also known (see e.g.,~\cite[Lemma 6]{Gretton12}, or simply use~\eqref{eq:hilnorm}):
		\begin{equation}\label{eq:MMD}
		{\rm MMD}_K^2(Q_1,Q_2) = \E[K(S,S')] + \E[K(W,W')] - 2 \E[K(S,W)],
		\end{equation}
		where $S, S',W,W'$ are independent, $S, S' \stackrel{}{\sim} Q_1$ and $W, W' \stackrel{}{\sim} Q_2$.
	\end{defn}
	
	\begin{defn}[Characteristic kernel]\label{defn:Char}
		The kernel $K(\cdot,\cdot)$ is said to be {\it characteristic} if and only if the map $Q\mapsto m_K(Q)$ is one-to-one on the domain $\tmk^1(\Y)$, i.e., $$m_K(Q_1) = m_K(Q_2) \quad  \implies \quad Q_1=Q_2, \quad \mbox{for all }\;Q_1,Q_2\in \tmk^1(\Y).$$
	\end{defn}
	Note that the last condition is equivalent to $\langle m_K(Q_1),f\rangle_{\hk} = \langle m_K(Q_2),f\rangle_{\hk}$ for all $f \in \hk$, i.e., $\E_{S \sim Q_1}[f(S)] = \E_{W \sim Q_2}[f(W)]$, for all $f \in \hk$. A characteristic kernel implicitly implies that the associated RKHS is rich enough.

	\subsection{NNGs and MSTs beyond Euclidean spaces}\label{sec:gengraphcond}
	
	The proof of~\cref{prop:vercond} reveals that (A1) continues to hold for general metric spaces, with metric $\rho(\cdot,\cdot)$, provided $\rho(X_1,X_2)$ has a continuous distribution. In fact, it can also be shown that (A2) continues to hold for NNGs whenever (A3) holds. Therefore, the crucial step is to verify (A3). In a similar vein, it is known that the MST also satisfies (A1) under some technical assumptions on the underlying metric space (for related results, see~\cite[Theorem 1.1]{gottlieb2020non},~\cite[Lemma 6]{talwar2004bypassing},~\cite[Proposition 12]{Sanjeev1998}). Once again, (A2) can be verified if (A3) holds and consequently, establishing (A3) is of prime importance. Recall that (A3) intuitively assumes that the maximum degree over the minimum degree of the associated MST/NNG is bounded (in $n$). Some results on upper bounds on the maximum degree of an MST, for points on normed spaces, can be found in~\cite{robins1994maximum,Swanepoel2018}.
	
	
	\noindent From a methodological standpoint, note that calculating $\kmac$ is not confined to metric spaces only. The computation of (approximate) NNGs on non-metric spaces has attracted a lot of attention over the years, with metrics being replaced by certain semimetrics, ``similarity" functions or ``divergence" measures (see~\cite{boytsov2013learning,athitsos2004boostmap,miranda2013very,jacobs2000classification,Gottlieb2017}). The example of semimetric spaces seems to be of particular interest in the machine learning literature, where a standard approach towards analyzing data taking values in some abstract topological space, is by defining ``interesting" kernels on the space and studying the corresponding RKHS (see~\cref{def:rkhs}) instead. Under certain technical assumptions, these kernels can be used to construct semimetrics (see~\cite[Corollary 16]{Sejdinovic2013}) which can then be used to obtain NNGs.
	
	\subsection{Some general remarks}\label{sec:genrem}
	\begin{remark}[Assumptions in~\cref{prop:otprop}]\label{rem:conconv}
		Suppose $Z,Z_1,Z_2\overset{i.i.d.}{\sim}\mathcal{N}(0,1)$. Let $(X_n,Y_n):=(Z/n,Z)$. Then $(X_n,Y_n)$ converges weakly to $(X,Y)\overset{d}{=}(0,Z)$, whereas $(X_n',Y_n',\tilde{Y}_n')$ (as stated in~\cref{prop:otprop}-(2)) converges weakly to $(0,Z,Z)$ which is not the same as the distribution of $(X',Y',\tilde{Y}')\overset{d}{=}(0,Z_1,Z_2)$. Therefore $(X_n,Y_n)\overset{w}{\longrightarrow}(X,Y)$ is not enough to guarantee the assumption made in~\cref{prop:otprop}-(2). Additionally, the assumption $\limsup_{n\to\infty} \E h_3^{1+\epsilon}(Y_n'-\tilde{Y}_n')<\infty$ for some $\epsilon>0$ is easily verifiable in a number of cases. For example, when the underlying kernel is Gaussian or Laplacian (see~\cref{rem:excharker}), $h_3(\cdot)$ is uniformly bounded and there is nothing to check. If the underlying kernel is the one described in~\cref{rem:Euclid}, then the condition holds provided $\limsup_{n\to\infty}\E\lVert Y_n\rVert_2^{\alpha+\gamma}<\infty$ for some $\gamma>0$.
	\end{remark}
	
	\begin{remark}[More on the computation of $\kmac$]\label{rem:morecomp}
		When we are dealing with a finite dimensional RKHS, e.g., $K(y,\tilde y) := \varphi(y)^\top \varphi(\tilde y)$, for some `feature map' $\varphi(\cdot) \in \R^m$, then it is indeed possible to make $v_n = \mathcal{O}(n)$ as $\left\| \sum_{i=1}^n K(Y_i,\cdot) \right\|^2_{\h_K} = \| \varphi(\cdot)^\top (\sum_{i=1}^n  \varphi(Y_i)) \|^2_{\h_K} = \|\sum_{i=1}^n  \varphi(Y_i)\|^2_{\R^m}$. 
		However, for infinite dimensional RKHSs, we do not know of such generally applicable linear time algorithms. For certain kernels, there might be special tricks that can be used to facilitate fast computation; see e.g.,~\cite{gray2001} (and \url{http://www.cs.cmu.edu/~agray/nbody.html}). 
		
		Alternatively, under certain conditions (i.e., Mercer's theorem and extensions), we
		can write 
		\begin{equation}\label{eq:Mercers}
		K(y,\tilde y) := \sum_{k=1}^\infty \lambda_k e_k(y) e_k(\tilde y), \qquad \forall \; y,\tilde y \in \Y
		\end{equation}
		where the series converges absolutely for each $(y, \tilde y) \in \Y \times \Y$ (and uniformly on compact subsets of $\Y \times \Y$), $\lambda_1 \ge \lambda_2 \ge \cdots \ge 0$ are the eigenvalues and $\{e_k\}_{k\ge 1}$ are the corresponding $L_\tau^2$-normalized eigenfunctions of $K(\cdot,\cdot)$ with respect to a probability measure $\tau$ on $\Y$ (i.e., $\int_\Y e_i(y) e_j(y) d \tau(y) = \delta_{i,j}$, where $\delta_{i,j}$ is the Kronecker delta function). Then, $\{\sqrt{\lambda_k} e_k\}_{k\ge 1}$ forms an orthonormal basis of $\h_K$ and for any $f \in \h_K$ such that $$f = \sum_{k=1, \lambda_k >0}^\infty a_k e_k$$ we have 
		\begin{equation}\label{eq:Expand}
		\|f\|_{\h_K} = \sum_{k=1, \lambda_k >0}^\infty \frac{a_k^2}{\lambda_k} < \infty.
		\end{equation}
		
		Now consider $$g(\cdot) := \sum_{i=1}^n K(Y_i,\cdot) \in \h_K.$$ Then, by~\eqref{eq:Mercers}, $g(\cdot) = \sum_{i=1}^n \sum_{k=1}^\infty \lambda_k e_k(Y_i) e_k(\cdot) = \sum_{k=1}^\infty a_k  e_k(\cdot)$, where
		\begin{equation}\label{eq:f_k}
		a_k = \sum_{i=1}^n \lambda_k e_k(Y_i), \quad \mbox{for } k \ge 1.
		\end{equation}
		Thus, using~\eqref{eq:Expand} we have 
		\begin{equation}\label{eq:g_H_K}
		\Big\| \sum_{i=1}^n K(Y_i,\cdot) \Big\|^2_{\h_K} = \sum_{i,j} K(Y_i,Y_j) = \|g\|_{\h_K}^2 = \sum_{k=1, \lambda_k >0}^\infty \frac{a_k^2}{\lambda_k}.
		\end{equation}
		
		For example, for the Gaussian kernel on $\R$ we have $K(y,\tilde y) := \exp \left(-\frac{1}{2 \sigma^2}(y -\tilde y)^2\right)$. Letting $d \tau(y) = \frac{1}{\sqrt{2 \pi}} e^{-y^2/2} dy$, we can show that $$\lambda_k \propto b^k \;\; (\mbox{for } b <1) \; \qquad \mbox{and} \qquad e_k(y) \propto \exp(-(c-a)y^2) H_k(y \sqrt{2c}),$$ where $a, b, c$ are functions of $\sigma$, and $H_k$ is the $k$-th order Hermite polynomial (i.e., $H_k(y) = (-1)^k \exp(y^2) \frac{d^k}{d y^k} \exp(-y^2)$); see e.g.,~\cite[Section 4.3]{RW-2006}). Thus, we can compute $a_k$ using~\eqref{eq:f_k} now, and consequently compute $\Big\| \sum_{i=1}^n K(Y_i,\cdot) \Big\|^2_{\h_K}$ using~\eqref{eq:g_H_K}.
		
	\end{remark}
	
	\begin{remark}[Motivation for $\klin$]\label{rem:Compare} 
		At first glance, the construction of $\klin$ may seem artificial. However, note that in the definition of $\kmac$ in~\eqref{eq:statest}, the term $(n(n-1))^{-1}\sum_{i\neq j} K(Y_i,Y_j)$ was used because it is a ``natural" estimator for $\E [K(Y_1,Y_2)]$, a quantity which features in the numerator of the population version $\eta_K(\mu)$ (see~\cref{prop:wdf}). The replacement of $(n(n-1))^{-1}\sum_{i\neq j} K(Y_i,Y_j)$ in $\klin$ also has similar ``good" properties.	 In particular, ${n}^{-1}\sum_{i=1}^n K(Y_i,Y_{i+1})$ is also unbiased for $\E K(Y_1,Y_2)$ and
		\begin{align}\label{eq:goodprop}
		n\mathrm{Var}\left(\frac{1}{n}\sum_{i=1}^n K(Y_i,Y_{i+1})\right)\overset{n\to\infty}{\longrightarrow} a+2b-3c
		\end{align}
		where $a:=\E K^2(Y_1,Y_2)$, $b:=\E[K(Y_1,Y_2)K(Y_1,Y_3)]$ and $c:=\E^2 [K(Y_1,Y_2)]$. Therefore, $n^{-1}\sum_{i=1}^n K(Y_i,Y_{i+1})$  is also a $\sqrt{n}$-consistent estimator for $\E[K(Y_1,Y_2)]$ (under appropriate moment assumptions).
		
		We would also like to point out that $\E[K(Y_1,Y_2)]$ is only a function of $\mu_Y$ (and has no dependence on $X$) which makes it much easier to estimate than $\E \lVert \E[K(\cdot,Y)|X]\rVert_{\hk}^2$ --- the other term appearing in the numerator of $\eta_K(\mu)$ (see~\cref{prop:wdf}) --- which we have already found a  near  linear time estimator for, by using~\eqref{eq:analogtarget}. This leads us to conjecture that replacing $(n(n-1))^{-1}\sum_{i\neq j} K(Y_i,Y_j)$ with $n^{-1}\sum_{i=1}^n K(Y_i,Y_{i+1})$ to get $\klin$ from $\kmac$, so as to get faster computability, is a fair alternative. 
	\end{remark}

	\section{Simulation studies}\label{sec:sim}
	In this section, we will illustrate through simulations, the different properties of $\kmac$ and $\klin$ under different choices of kernels $K(\cdot,\cdot)$ and graph functionals $\mathcal{G}$. Our two primary examples of kernels are the distance kernel and the Gaussian kernel, which we recall below:
	\begin{enumerate}
		\item[(A)] $\kd(y_1,y_2):=(1/2)(\lVert y_1\rVert_2+\lVert y_2\rVert_2-\lVert y_1-y_2\rVert_2)$, and
		\item[(B)] $\kg(y_1,y_2):=\exp(-\lVert y_1-y_2\rVert_2^2)$.
	\end{enumerate}
	For graph functionals, we will either use the Euclidean MST (EMST) or $k$-NNG as described in~\cref{sec:exkergraph}. We will also use a shorthand $(*,*,*)$ to describe the choice of the measure of association, the kernel and gthe raph functional respectively. So, for instance, $(\kmac,\kg,\mst)$ will imply that $\kmac$ has been constructed using the Gaussian kernel and the EMST; similarly $(\klin,\kd,\onn)$ will imply that $\klin$ has been constructed using the distance kernel and the $1$-NNG.
	
	\subsection{A general measure of association}\label{sec:genassoc}
	In this sub-section, we will present simulation evidence to demonstrate that $\kmac$ and $\klin$ are very general measures that capture the strength of dependence between $X$ and $Y$. In other words, they are generally capable of distinguishing between an exact functional relationship between $X$ and $Y$, as opposed to noisier relationships between the same variables; thereby making them very powerful measures of association. No other dependence measure that we know of has this property. Our illustration will feature the distance correlation (dCor) as a benchmark as it is standardized between $0$ and $1$ and is potentially the most popular measure of dependence in the statistics community.
	
	Let us begin with two simple simulation settings. As a general rule for this section, $n$ will denote the sample size and we will stick to $d_1=d_2=2$. Also the $4$-dimensional vector $(X,Y)$ will be generated by first drawing $(X^{(1)},Y^{(1)}),(X^{(2)},Y^{(2)})$ from some bivariate distribution $\mu$ and then setting $X:=(X^{(1)},X^{(2)})$ and $Y:=(Y^{(1)},Y^{(2)})$. Therefore we will only specify $\mu$, i.e., the distribution of $(X^{(1)},Y^{(1)})$ in the sequel.
	\begin{enumerate}
		\item \emph{Sinusoidal}: Let $X^{(1)}\sim \mathcal{U}[-1,1]$ and $Y^{(1)}:=\cos(8\pi X^{(1)})+\lambda\epsilon$, where $\epsilon\sim\mathcal{N}(0,1)$ is independent of $X^{(1)}$ and $\lambda$ varies in $[0,2.5]$. Set $n=2000$. This setting has been taken from~\cite[Setting 4]{chatterjee2019new}.
		\item \emph{Linear}: Let $(X^{(1)},Y^{(1)})$ be a bivariate Gaussian random vector with correlation $\rho$ varying in $[0,1]$ and both marginals having mean $0$ and variance $1$. Set $n=2000$. Note that in this setting, the most intuitive measure of dependence is the correlation parameter $\rho$.
	\end{enumerate}
	\begin{figure}[h] 
		\centering
		\begin{subfigure}[b]{0.48\textwidth}
			\includegraphics[width=\textwidth,height=6.5cm]{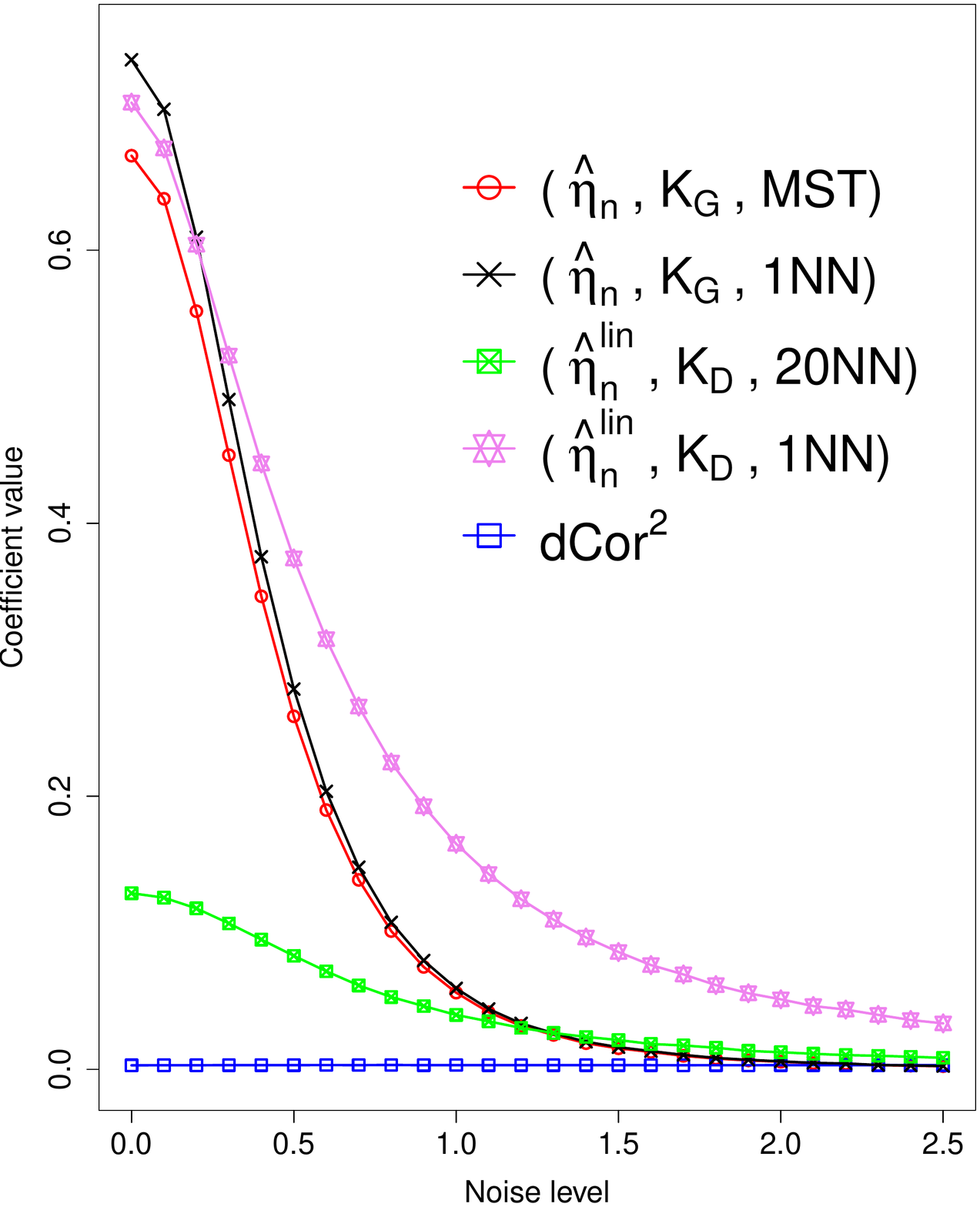}
			\caption{Sinusoidal setting}
			\label{fig:Coeffplot1}
		\end{subfigure}
		\begin{subfigure}[b]{0.48\textwidth}
			\includegraphics[width=\textwidth,height=6.5cm]{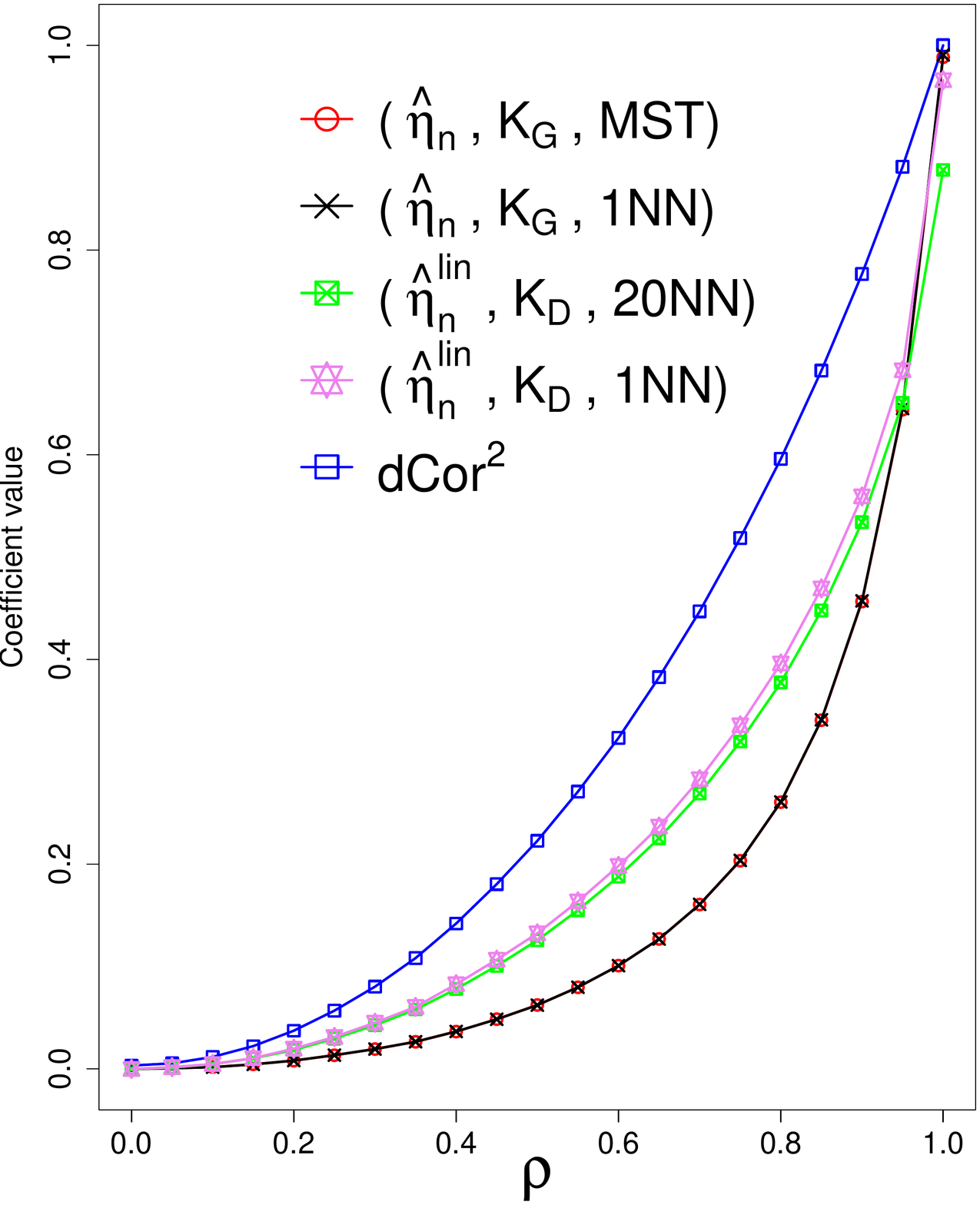}
			\caption{Linear setting}
			\label{fig:Coeffplot2}
		\end{subfigure}
		\caption{Each subplot depicts the value of $\mathrm{dCor}^2$ and $\kmac$, $\klin$ for different kernels and graph functionals.}
		\label{fig:Coeffplot}
	\end{figure}
	In~\cref{fig:Coeffplot1}, we observe that $\mathrm{dCor}^2$ is unable to distinguish between noise variance $0$ (perfect functional relationship) and noise variance $6.25$. In fact, if we zoom in on the values of $\mathrm{dCor}^2$, they don't seem to be monotonically decreasing with the noise level. On the other hand, both $\kmac$ and $\klin$ decrease sharply as we go from noise variance $0$ to $6.25$ which clearly shows a decline in the degree of dependence between the two variables. Note that $(\klin,\kd,\tnn)$ does not show as rapid a decline as our other proposed measures. This is in accordance with~\cref{theo:rateres} which shows that $\kmac$ has a larger bias if we increase the number of nearest neighbors.
	
	\noindent For~\cref{fig:Coeffplot2}, the ideal measure of dependence should have been $\rho^2$. We can clearly see that the curve of $\dC^2$ is closest to $\rho^2$. In fact, distance correlation is well suited to detecting such linear relationships as was argued in~\cite{Gabor2007}. Even in this case though, all our proposed measures are clearly able to distinguish between varying levels of correlation, and vary smoothly and monotonically between $0$ and $1$. 
	\subsection{Validity of asymptotic theory when $X$ and $Y$ are independent}\label{sec:valnull}
	In this sub-section, we will provide numerical evidence in support of~\cref{theo:nullclt}~and~\cref{prop:nullcltesscomp}. We will use the following two simulation settings, where $X$ and $Y$ are independent. In one setting, $\mu_X$ is absolutely continuous whereas in the other case, it is not. This has been chosen to highlight the generality of our results.
	\begin{enumerate}
		\item[(i)] $\tilde{X}=(\tilde{X}_1,\tilde{X}_2,\tilde{X}_3,\tilde{X}_4)\sim \mathcal{U}[0,1]^4$ and $X\overset{d}{=}(\tilde{X},\tilde{X}_1+\tilde{X}_2)$ is a $5$-dimensional vector, $Y\overset{d}{=} (Y_1,Y_2,Y_3,Y_4)$ with $Y_i\overset{i.i.d.}{\sim}\mathrm{Exp}(1)$ for $i=1,2,3,4$ is a $4$-dimensional vector. The sample size ($n$) is $2000$. Clearly the marginal distribution of $X$ in this case, does not admit a Lebesgue density.
		\item[(ii)] The same as above with $X\overset{d}{=}\tilde{X}$. The sample size ($n$) is $2000$. The marginal density of $X$ in this case does admit a Lebesgue density.
	\end{enumerate}

	\begin{figure}[h] 
		\centering
		\begin{subfigure}[b]{0.48\textwidth}
			\includegraphics[width=\textwidth,height=6.5cm]{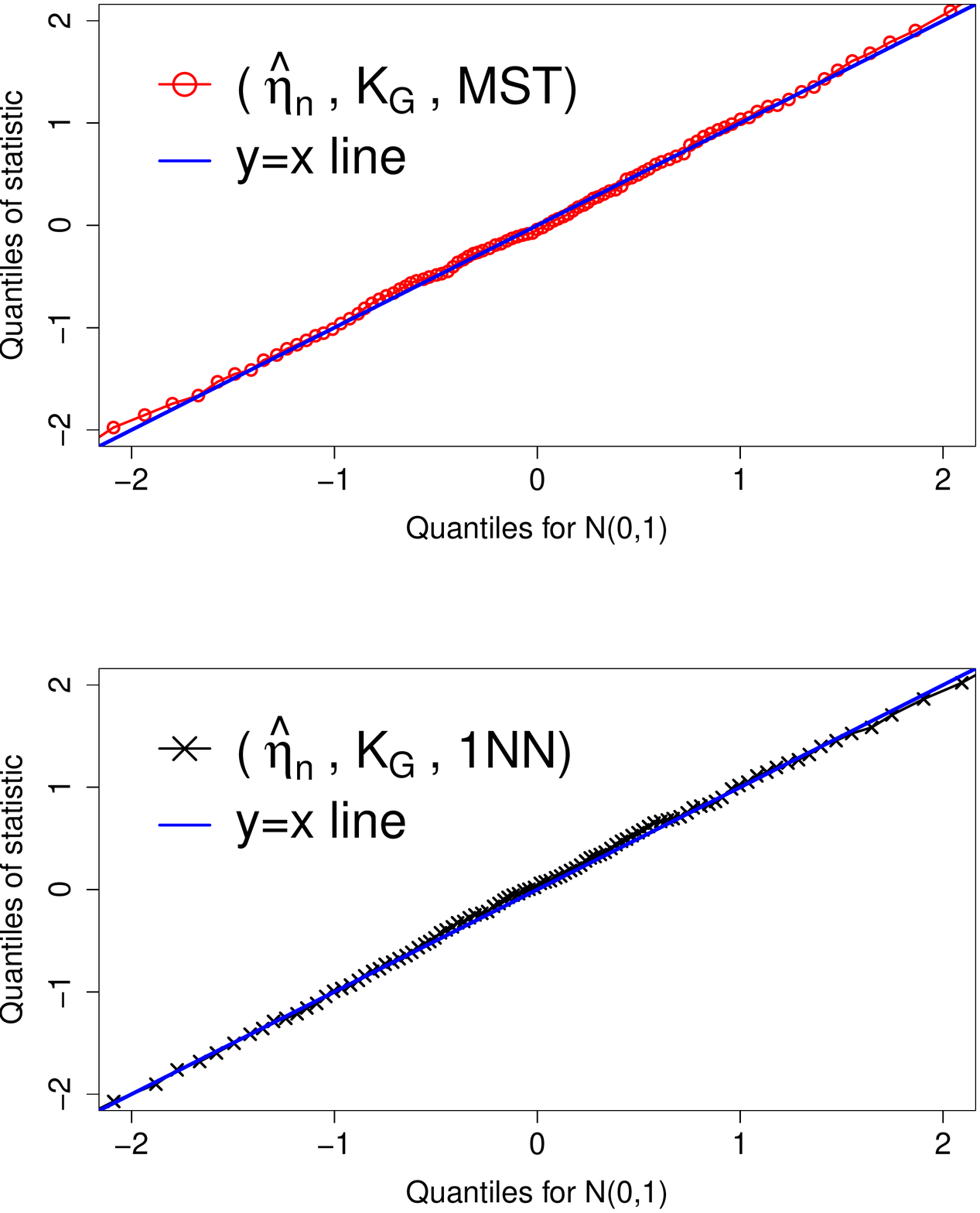}
			\caption{Setting (i)}
			\label{fig:CLTqqplot1}
		\end{subfigure}
		\begin{subfigure}[b]{0.48\textwidth}
			\includegraphics[width=\textwidth,height=6.5cm]{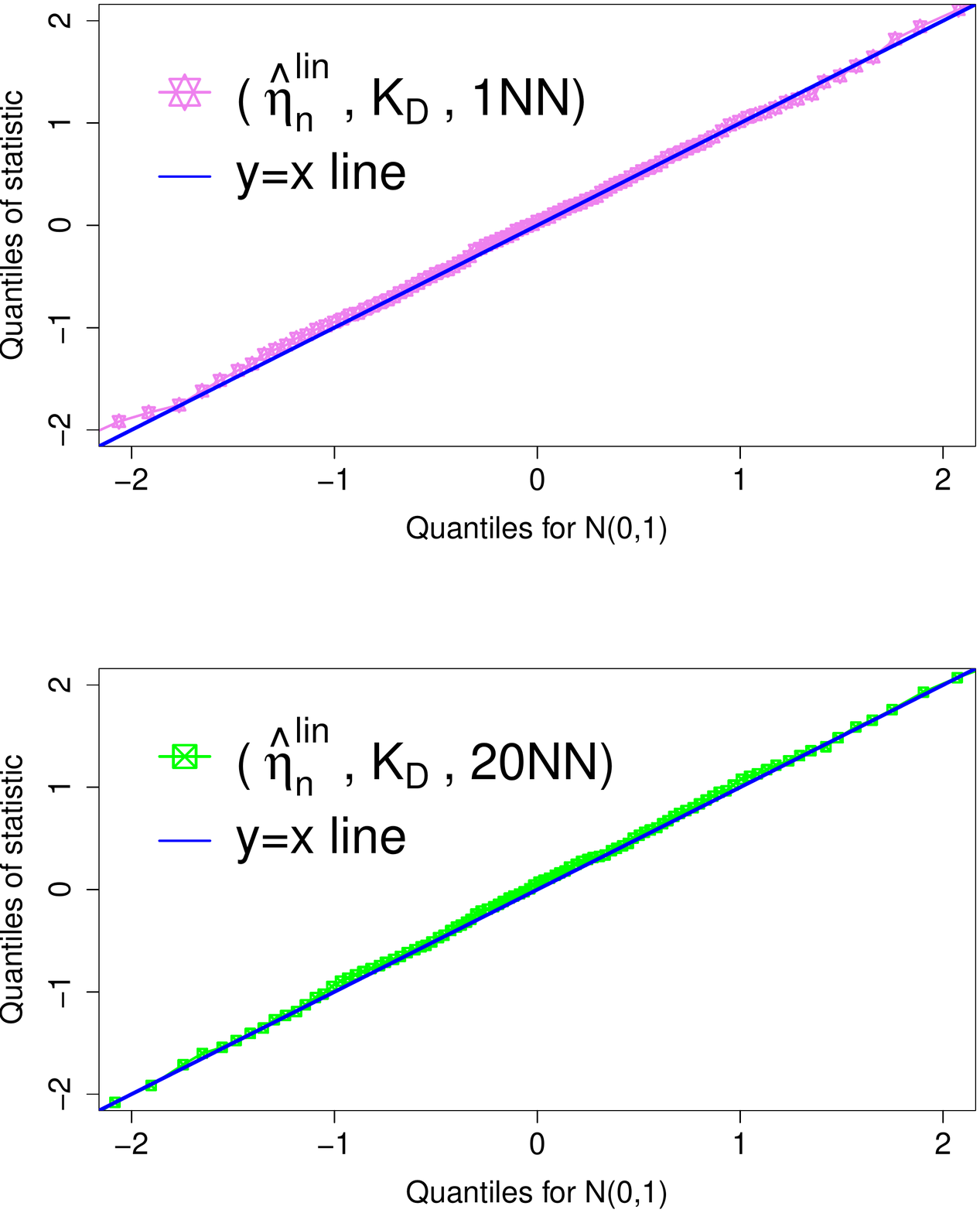}
			\caption{Setting (ii)}
			\label{fig:CLTqqplot2}
		\end{subfigure}
		\caption{In the left panel, we plot the quantiles for $N_n/\tilde{S}_{n}$ (see~\cref{theo:nullclt}) versus those of a standard Gaussian random variable where the underlying graph is the MST (top) and $1$-NNG (bottom), with the kernel being $\kg(\cdot,\cdot)$ for both plots. In the right panel, we plot the quantiles for $\linn/\tilde{S}_{n,\mathrm{lin}}$ (see~\cref{prop:nullcltesscomp}) versus those of a standard Gaussian random variable where the underlying graph is the $1$-NNG (top) and $20$-NNG (bottom), with the kernel being $\kd(\cdot,\cdot)$. In all the plots we take $n=2000$.}
		\label{fig:CLTqqplot}
	\end{figure}
	In~\cref{fig:CLTqqplot}, we present quantile-quantile plots between $N_n/\tilde{S}_n$ and the standard Gaussian distribution (see~\cref{theo:nullclt}); also $\linn/\tilde{S}_{n,\mathrm{lin}}$ and the standard Gaussian distribution (see~\cref{prop:nullcltesscomp}). All the quantile-quantile plots are in close agreement with the $y=x$ line, thereby providing a strong evidence in support of the claimed asymptotic normality. We also carried out a Kolmogorov-Smirnov test for normality in all the above cases. The $p$-values as we move counter-clockwise from the top left were $0.792,0.336,0.291$ and $0.511$ respectively.
	
	\begin{remark}[Choice of $k_n$]\label{rem:choosekagn}
		An interesting observation that came out of this computational study is the fact that the limiting asymptotic variance for $(\klin,\kd,\onn)$ in setting (ii) was approximately $1.367$ whereas the same for $(\klin,\kd,\tnn)$ was $0.528$. In fact, we observed this reduction in variance across other simulation settings and kernel choices as well. This supports~\cref{rem:choosek} where we claimed that when $X$ and $Y$ are independent, increasing the number of nearest neighbors can lead to greater efficiency.
	\end{remark}
	
	\noindent Another aspect of~\cref{theo:nullclt}~and~\cref{prop:nullcltesscomp} which should be verified is the tacit assumption that $\liminf_{n\to\infty}\mathrm{Var}(N_n)>0$ and $\liminf_{n\to\infty}\mathrm{Var}(\linn)>0$. We will verify this using the classical ``log-log" plot which we describe below.
	
	Note $N_n=\sqrt{n}\kmac^{\mathrm{num}}$. We chose a grid of sample sizes, i.e., $n$ varying between $2^8$ and $2^{11}$. For each sample size, we approximated the standard deviation of $\kmac^{\mathrm{num}}$. Then we obtained the slope of the least squares regression line between the logarithms of the standard deviations versus logarithms of sample sizes. The same was repeated with the numerator of $\klin$. If our conjecture is correct, then this slope should be close to $-0.5$. We present our findings in~\cref{tab:rateproof}. The table provides strong evidence in favor of our assumption.
	\begin{table}[ht]
		\caption{Slopes from log-log plots with their $95\%$ confidence intervals}
		\begin{center}
			\begin{tabular}{|c|c|c|c|}
				\hline
				Setting & Configuration & Slope & Confidence Interval\\
				\hline
				\multirow{2}{*}{(i)}& $(\kmac,\kg,\mst)$ & $-0.496$ & $(-0.513,-0.479)$\\
				\cline{2-4}
				& $(\kmac,\kg,\onn)$ & $-0.497$ & $(-0.522,-0.473)$ \\
				\hline
				\multirow{2}{*}{(ii)}& $(\klin,\kd ,\onn)$ & $-0.498$ & $(-0.522,-0.473)$ \\
				\cline{2-4}
				& $(\klin,\kd,\tnn)$ & $-0.503$ & $(-0.523,-0.483)$ \\
				\hline
			\end{tabular}
		\end{center}
		\label{tab:rateproof}
	\end{table}
	
	\subsection{Power comparisons}\label{sec:powcomp}
	In this section, we consider the null hypothesis of independence between $X$ and $Y$ (see~\eqref{eq:teststate}) and compare the power functions of the statistical tests obtained using $\kmac$, $\klin$ implemented via the $k$-NNG, with two popular tests namely $\mathrm{dCor}$ (implemented in the \texttt{energy} package in \texttt{R}) and $\mathrm{HSIC}$ (implemented in the \texttt{dHSIC} package in \texttt{R} with the standard Gaussian kernel, i.e., $\kg(\cdot,\cdot)$). All the tests have been calibrated using a permutation procedure with $1000$ random permutations. A sample size of $n=300$ is used and the power of each test is estimated with $1000$ independent replicates. Moreover, we work with $d_1=d_2=2$ and revert back to the general rule discussed in~\cref{sec:genassoc}. Recall that the $4$-dimensional vector $(X,Y)$ was generated by first drawing $(X^{(1)},Y^{(1)}),(X^{(2)},Y^{(2)})\overset{i.i.d.}{\sim}\mu$ where $\mu$ is a bivariate distribution and then setting $X:=(X^{(1)},X^{(2)})$, $Y:=(Y^{(1)},Y^{(2)})$.  Throughout the sequel, $X^{(1)}\sim\mathcal{U}[-1,1]$ (unless specified otherwise), $\epsilon\sim\mathcal{N}(0,1)$ is drawn independent of $X^{(1)}$ and $\lambda$ is to be interpreted as a parameter varying between $0$ and $1$ which controls the noise level in the relationship between $X$ and $Y$. All our simulation settings are motivated from those in~\cite{chatterjee2019new}.
	\begin{enumerate}
		\item[(a)] \emph{Linear}: $Y^{(1)}=0.5X^{(1)}+3\lambda\epsilon$.
		\item[(b)] \emph{Sinusoidal}: $Y^{(1)}=\cos{(8\pi X^{(1)})}+3\lambda\epsilon$.
		\item[(c)] \emph{W-shaped}: $Y^{(1)}=|X^{(1)}+0.5|\mathbf{1}(X^{(1)}\leq 0)+|X^{(1)}-0.5|\mathbf{1}(X^{(1)}> 0)+0.75\lambda\epsilon$.
		\item[(d)] \emph{Step function}: $Y^{(1)}=f(X^{(1)})+10\lambda\epsilon$ where $f(\cdot)$ is a step function taking values $-3,2,4$ and $3$ in the intervals $[-1,-.05)$, $[-0.05,0)$, $[0,0.05)$ and $[0.05,1]$.
		\item[(e)] \emph{Semicircular}: As an exception, here we choose $X^{(1)}\sim\mathcal{U}(0,1)$, $Y^{(1)}=Z\sqrt{1-(X^{(1)})^2}+0.9\lambda\epsilon$ where $Z$ takes values $\pm 1$ with equal probability and is independent of both $X^{(1)}$ and $\epsilon$.
		\item[(f)] \emph{Heterogeneous}: $Y^{(1)}=3\left(\sigma(X^{(1)})(1-\lambda)+\lambda\right)\epsilon$ where $\sigma(x):=1$ if $|x|\leq 0.5$ and $0$ otherwise.
	\end{enumerate}
	\begin{figure}[h] 
		\centering
		\begin{subfigure}[b]{0.3\textwidth}
			\includegraphics[width=\textwidth,height=5.5cm]{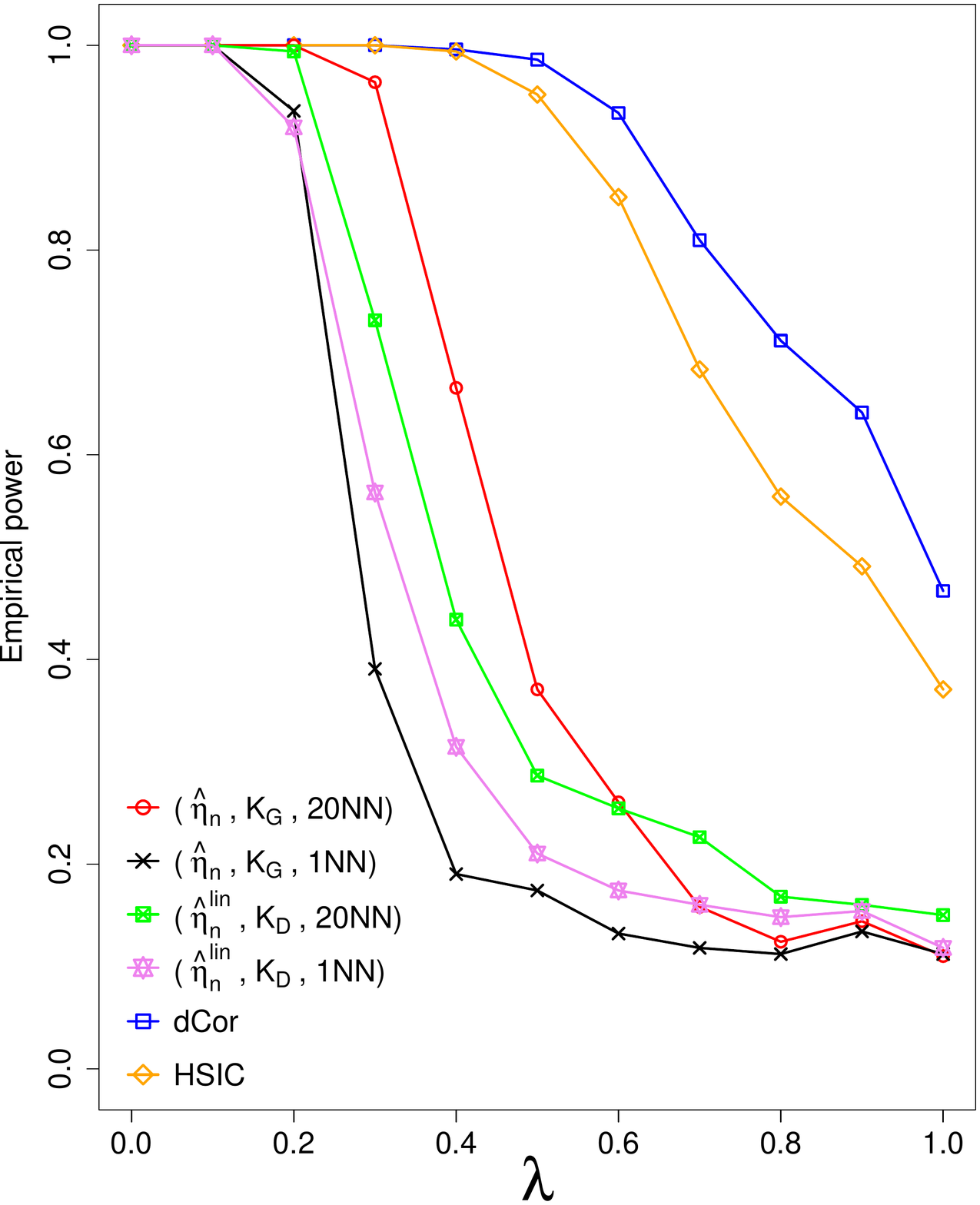}
			\caption{Linear}
			\label{fig:Powerplot1}
		\end{subfigure}
		\begin{subfigure}[b]{0.3\textwidth}
			\includegraphics[width=\textwidth,height=5.5cm]{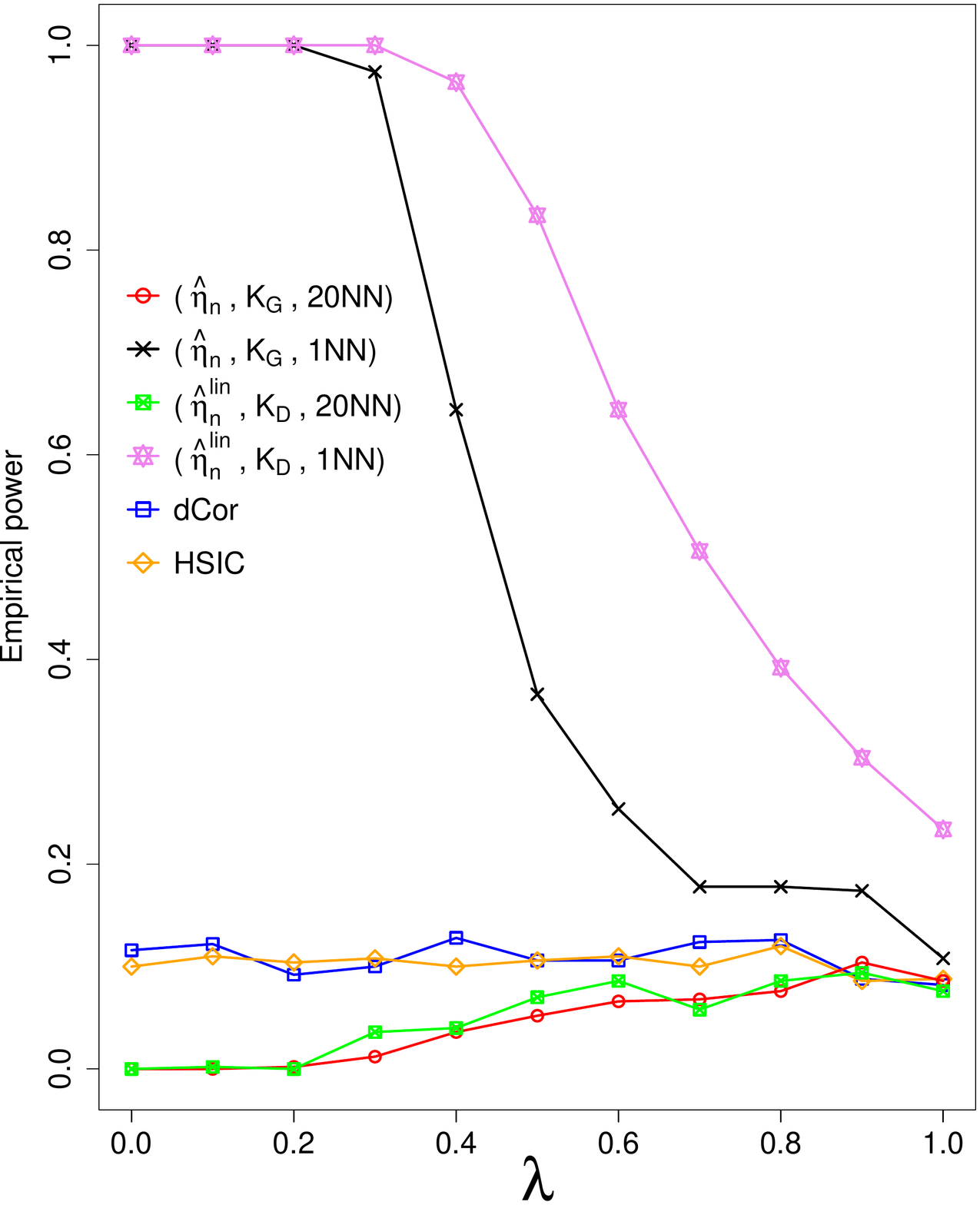}
			\caption{Sinusoidal}
			\label{fig:Powerplot2}
		\end{subfigure}
		\begin{subfigure}[b]{0.3\textwidth}
			\includegraphics[width=\textwidth,height=5.5cm]{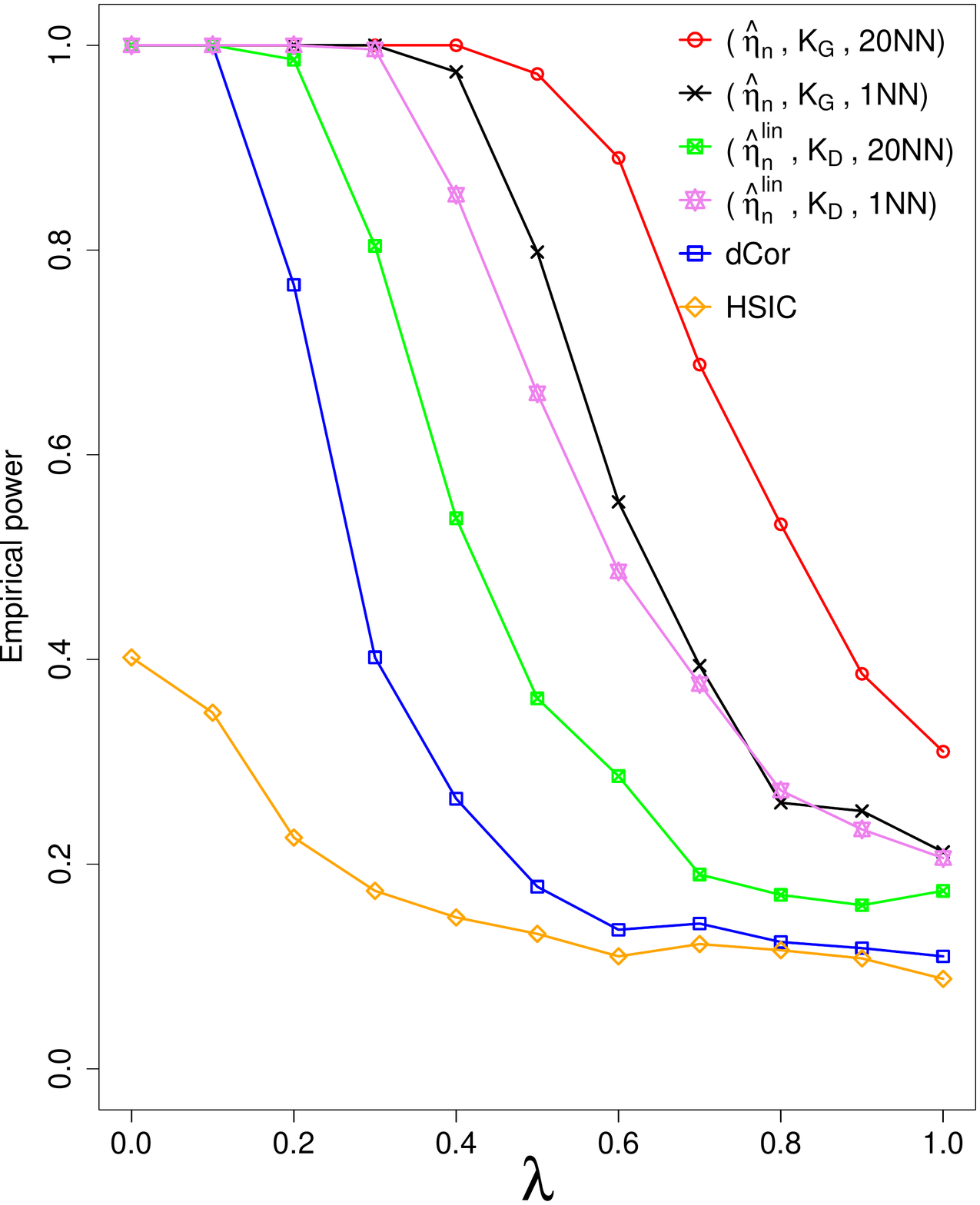}
			\caption{W-shaped}
			\label{fig:Powerplot3}
		\end{subfigure}
		\begin{subfigure}[b]{0.3\textwidth}
			\includegraphics[width=\textwidth,height=5.5cm]{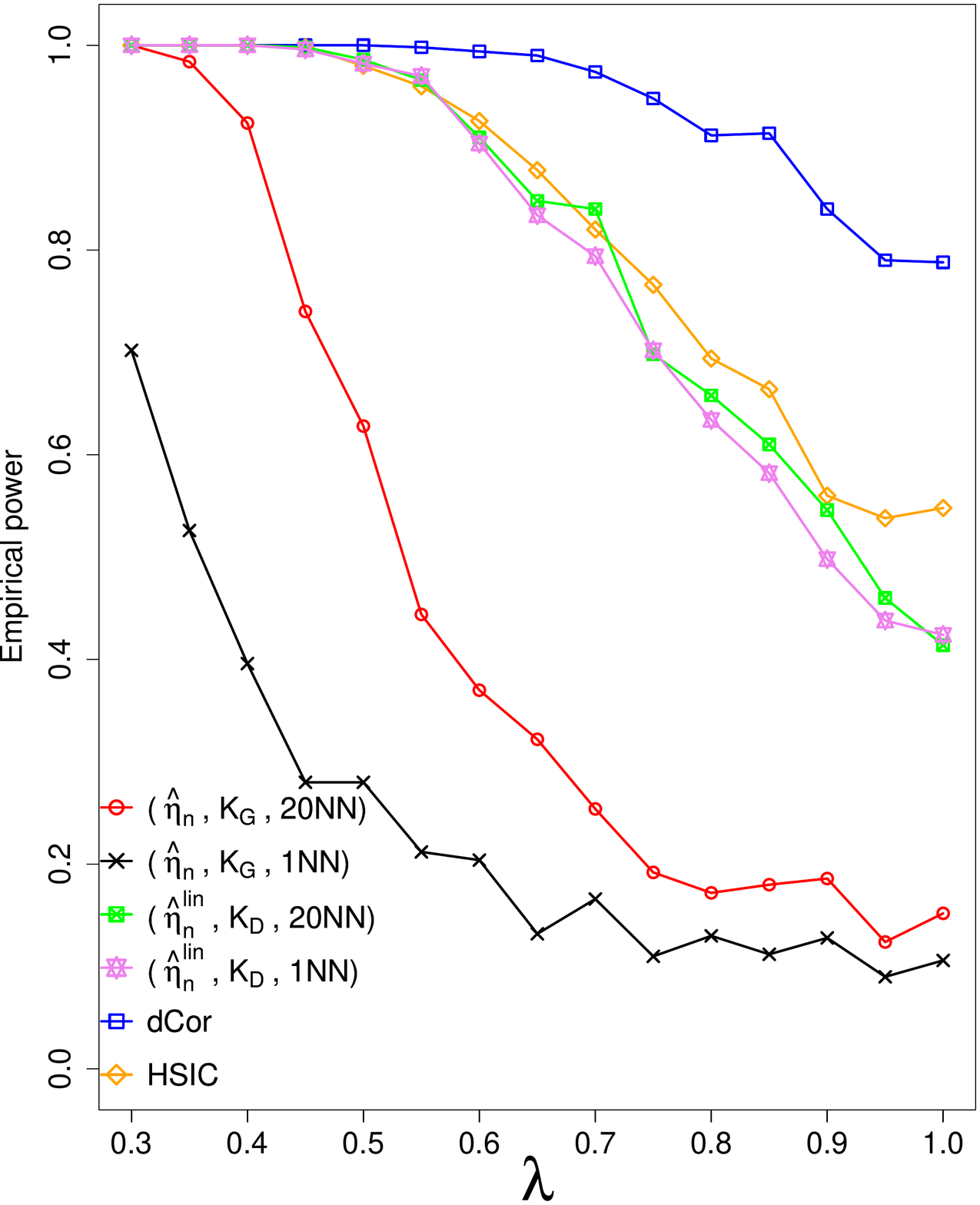}
			\caption{Step function}
			\label{fig:Powerplot4}
		\end{subfigure}
		\begin{subfigure}[b]{0.3\textwidth}
			\includegraphics[width=\textwidth,height=5.5cm]{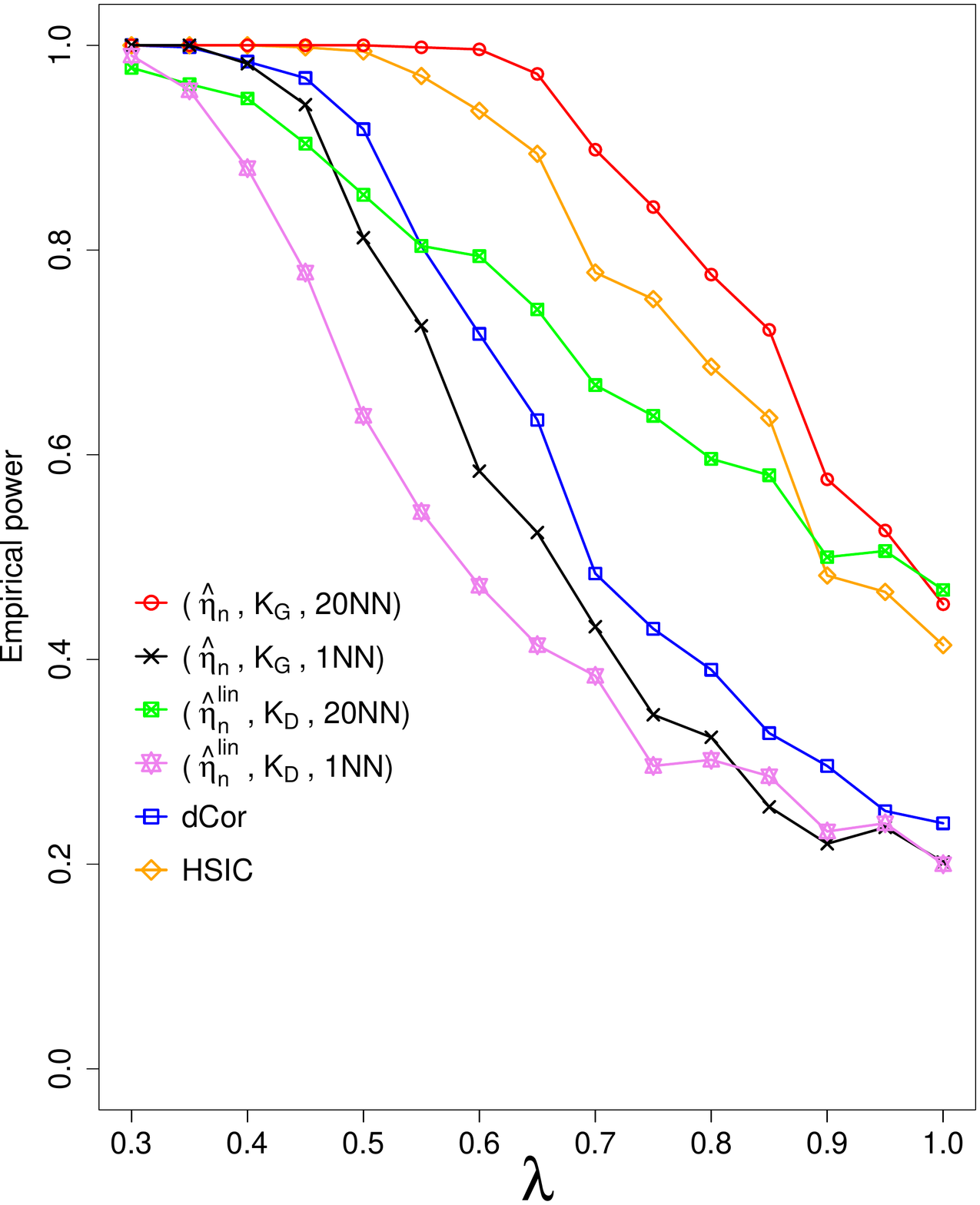}
			\caption{Semicircular}
			\label{fig:Powerplot5}
		\end{subfigure}
		\begin{subfigure}[b]{0.3\textwidth}
			\includegraphics[width=\textwidth,height=5.5cm]{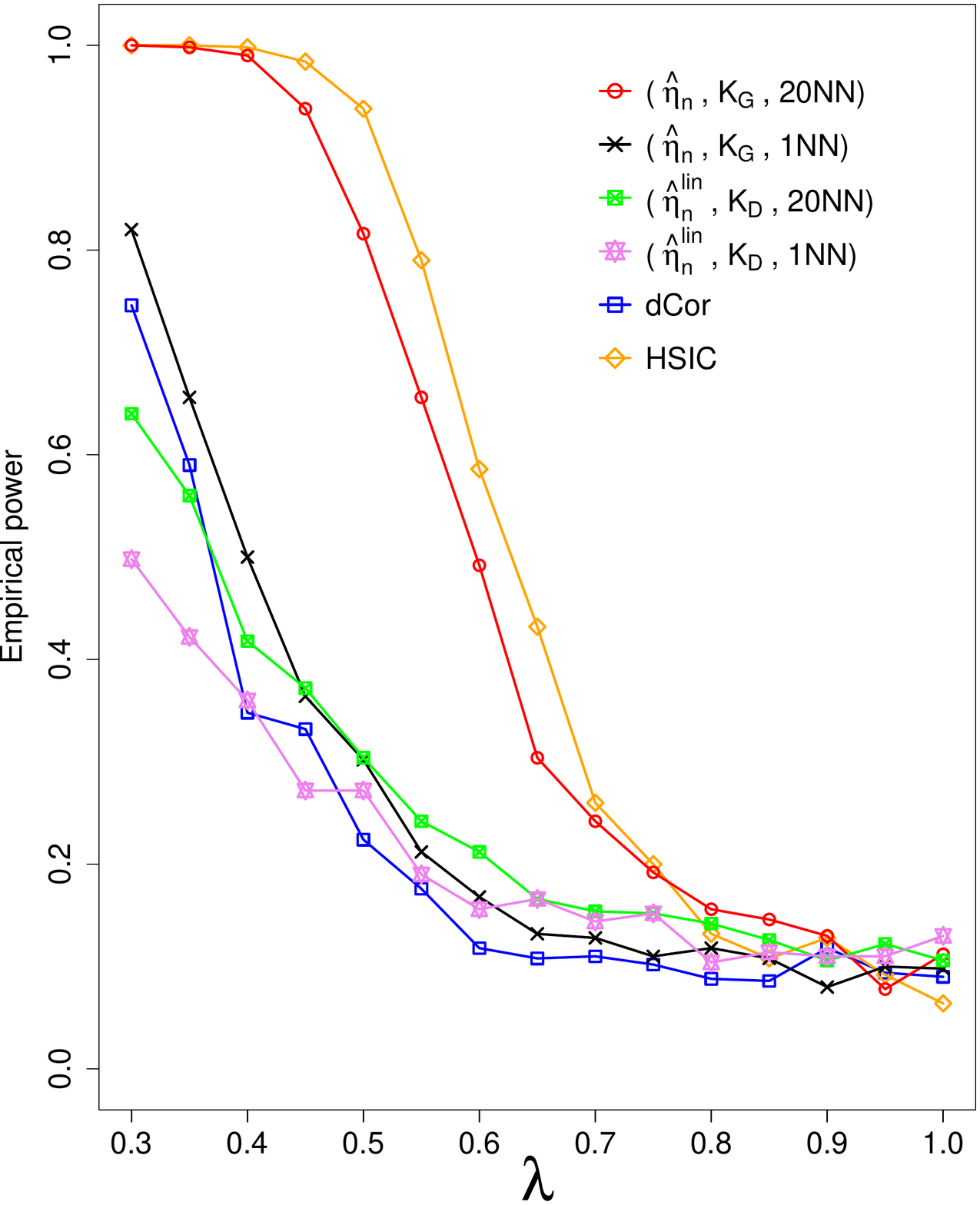}
			\caption{Heterogeneous}
			\label{fig:Powerplot6}
		\end{subfigure}
		\caption{A comparison of power functions of different tests of mutual independence between $X$ and $Y$. The title of each plot captures the shape of the underlying data cloud. Here $\lambda$ increases from left to right. In each plot, a sample size of $300$ was used and the power was estimated from a $1000$ random replications.}
		\label{fig:Powerplot}
	\end{figure}
	~\cref{fig:Powerplot} shows that tests based on both $\kmac$ and $\klin$ are competitive when compared to $\mathrm{dCor}$ and $\mathrm{HSIC}$. The general message seems to be that in the presence of a wiggly relationship (small changes in $X$ leading to large changes in $Y$, as in the sinusoidal, W-shaped or semi-circular setting) our methods tend to outperform $\mathrm{dCor}$ and $\mathrm{dHSIC}$, whereas in the presence of smoother relationships (small changes in $X$ leading to small changes in $Y$, as in the linear, step function or heterogeneous setting), our methods are less powerful than $\mathrm{dCor}$ and $\mathrm{dHSIC}$. This is in line with the observations made in~\cite{chatterjee2019new}. However, rather crucially, our simulations provide a more subtle insight into the sensitivity of our tests on the number of nearest neighbors used. 
	
	Consider the linear setting $(a)$. While our methods get outperformed by $\mathrm{dCor}$ and $\mathrm{HSIC}$, note that both $\kmac$ and $\klin$ with $20$ nearest neighbors perform better than their $1$ nearest neighbor counterparts. We believe that whenever the underlying relationship between $X$ and $Y$ is smooth, choosing $k$ large does not introduce too much bias but significantly reduces the variance leading to more powerful tests. On the flip side, whenever the relationship between $X$ and $Y$ is wiggly as in the sinusoidal one, choosing $k$ large introduces a lot of bias and consequently the $20$-NNG based tests have less power. In the sinusoidal setting $(b)$, both $\mathrm{dCor}$ and $\mathrm{dHSIC}$ are entirely powerless while our $1$-NNG based tests are quite powerful. This shows that there is an underlying trade-off while choosing the value of $k$, and a universal choice of $k=1$ as has been used in~\cite{azadkia2019simple}, is not necessarily recommended. In fact, the gain incurred by choosing $k$ large is not confined to just the linear setting. In settings $(c)$-$(f)$, the power curves for $\kmac$ with $20$-NNG are better than their $1$-NNG counterparts. In setting $(f)$ in particular, the $(\kmac,\kd,\tnn)$ power curve almost matches the $\mathrm{HSIC}$ power curve which performs the best of the lot in this setting. For $\klin$ too, the $20$-NNG power curves for the settings $(a)$, $(d)$-$(f)$ are better than the $1$-NNG ones. We believe that this observation highlights the importance of having such a flexible and general framework as we have considered in this paper.
	
	\begin{remark}[On the computational complexity of $\klin$] \label{rem:moreklin}
		~\cref{fig:Powerplot} shows that $\klin$ performs reasonably well compared to $\mathrm{dCor}$,  $\mathrm{dHSIC}$ and $\kmac$. Further, it crucially has $\mathcal{O}(n\log{n})$ time complexity compared to $\mathcal{O}(n^2)$ for the other methods under consideration. In fact, while implementing $\klin$, we observed that it is at least $50$ times faster to compute than the other methods for large sample sizes (in excess of $15000$). We believe that this makes $\klin$ a very useful general measure of the strength of dependence between two variables $X$ and $Y$.
	\end{remark}

	\section{Proofs of main results}\label{sec:pfmain}

	\subsection{Proof of~\cref{prop:wdf}}\label{pf:wdf}
	
	By using~\eqref{eq:hilnorm}, we get:
	\begin{align}\label{eq:altform1}
	\eta_{K}(\mu)&=1-\frac{\mme\lVert K(\cdot,Y')\rVert_{\hk}^2+\mme\lVert K(\cdot,\tilde{Y'})\rVert_{\hk}^2-2\mme\langle K(\cdot,Y'),K(\cdot,\tilde{Y'})\rangle_{\hk}}{\mme\lVert K(\cdot,Y_1)\rVert_{\hk}^2+\mme\lVert K(\cdot,Y_2)\rVert_{\hk}^2-2\mme\langle K(\cdot,Y_1),K(\cdot,Y_2)\rangle_{\hk}}\nonumber \\ &=\frac{\mme\langle K(\cdot,Y'),K(\cdot,\tilde{Y'})\rangle_{\hk}-\mme\langle K(\cdot,Y_1),K(\cdot,Y_2)\rangle_{\hk}}{\mme\lVert K(\cdot,Y_1)\rVert_{\hk}^2-\mme\langle K(\cdot,Y_1),K(\cdot,Y_2)\rangle_{\hk}}.
	\end{align} 
	Next, note that by the tower property, we have:
	$$\mme\langle K(\cdot,Y'),K(\cdot,\tilde{Y'})\rangle_{\hk}=\mme\left[\mme\left(\langle K(\cdot,Y'),K(\cdot,\tilde{Y'})\rangle_{\hk}|X'\right)\right]=\mme\lVert \mme\left[K(\cdot,Y)|X\right]\rVert_{\hk}^2.$$
	Similarly we have $\mme\langle K(\cdot,Y_1),K(\cdot,Y_2)\rangle_{\hk}=\lVert \mme K(\cdot,Y)\rVert_{\hk}^2$. Plugging these observations in~\eqref{eq:altform1}, we have:
	\begin{align*}
	\eta_{K}(\mu)&=\frac{\mme\lVert \mme\left[K(\cdot,Y)|X\right]\rVert_{\hk}^2-\lVert \mme K(\cdot,Y)\rVert_{\hk}^2}{\mme\lVert K(\cdot,Y)\rVert_{\hk}^2-\lVert \mme K(\cdot,Y)\rVert_{\hk}^2}\\ &=\frac{\mme\lVert \mme\left[K(\cdot,Y)|X\right]-\mme K(\cdot,Y)\rVert_{\hk}^2}{\mme\lVert K(\cdot,Y)-\mme K(\cdot,Y)\rVert_{\hk}^2}.
	\end{align*}
	In order to establish the second equality, note that by~\cref{def:meanembed}, we have $\mme[K(\cdot,Y)|X]=m_{K}(\mu_{Y|X})$ and $\mme[K(\cdot,Y)]=m_{K}(\mu_Y)$. Therefore, the numerator in the above display equals $\mme\lVert m_{K}(\mu_{Y|X})-m_{K}(\mu_Y)\rVert_{\hk}^2=\mme_{\mu_X}[\mathrm{MMD}^2_{K}(\mu_{Y|X},\mu_Y)]$. This completes the proof.
	\qed
	
	\subsection{Proof of~\cref{theo:kacmas}}\label{pf:kacmas}
	By~\eqref{eq:kac}, $\eta_{K}(\mu)$ is bounded above by $1$. The alternate representation of $\eta_{K}(\mu)$, as in~\cref{prop:wdf}, shows that it is nonnegative. This proves (P1).
	
	Recall that as $K(\cdot,\cdot)$ is characteristic and $\mu_Y$, $\mu_{Y|x}$ (for $\mu_X$-a.e.~$x$) are both elements of $\tmk^1(\mathcal{X})$. Therefore, by~\cref{prop:wdf}, $\mme[\mathrm{MMD}_{K}^2(\mu_{Y|X},\mu_Y)]=0$ if and only if $\mu_{Y|x}=\mu_Y$ for $\mu_X$-a.e.~$x$. Now, assume that $\mu=\mu_X\otimes \mu_Y$. This implies $\mu_{Y|x}=\mu_Y$ for $\mu_X$-a.e.~$x$, which in turn implies $\mme[\mathrm{MMD}_{K}^2(\mu_{Y|X},\mu_Y)]=0$, and consequently $\eta_{K}(\mu)=0$. Next, assume that $\eta_{K}(\mu)=0$ which implies $\mme[\mathrm{MMD}_{K}^2(\mu_{Y|X},\mu_Y)]=0$. As a result, $\mu_{Y|x}=\mu_Y$ for $\mu_X$-a.e.~$x$ and consequently $\mu=\mu_X\otimes \mu_Y$. This proves (P2).
	
	Next, suppose that $Y=g(X)$ for some measurable function $g:\mathcal{X}\to\Y$. This implies that both $Y'$ and $\tilde{Y'}$ are almost surely (a.s.) equal to $h(X')$. Plugging this in~\eqref{eq:kac} immediately yields that $\eta_{K}(\mu)=1$. 
	
	For the other direction, suppose $\eta_{K}(\mu)=1$. This implies $K(\cdot,Y')=K(\cdot,\tilde{Y'})$ a.s., and consequently $Y'=\tilde{Y'}$ a.e.~with respect to the joint distribution of $(X',Y',\tilde{Y'})$ (see~\cite[Proposition 14]{Sejdinovic2013}). This implies that there exists a subset $\Omega'\subseteq \mbox{supp}(\mu_X)$ such that for all $x'\in\Omega'$, the conditional distribution $(Y',\tilde{Y'})|X'=x'$ is degenerate, and $\mu_X(\Omega')=1$. We will show that for all $x'\in\Omega'$, the support of $\mu_{Y|x'}$ is a singleton. Let us proceed by contradiction. If the support of $\mu_{Y|x'}$ contains at least two points, then by the Hausdorff nature of $\Y$, we can find two disjoint Borel sets $\mathcal{A}_1^{x'}$ and $\mathcal{A}_2^{x'}$ such that $\mu_{Y|x'}(\mathcal{A}_1^{x'})>0$ and $\mu_{Y|x'}(\mathcal{A}_2^{x'})>0$. This would imply,
	$$0=\mu_{Y|x'}(\mathcal{A}_1^{x'}\cap \mathcal{A}_2^{x'})=\mu_{Y|x'}(\mathcal{A}_1^{x'})\times \mu_{Y|x'}(\mathcal{A}_2^{x'})>0$$
	which in turn gives a contradiction. Therefore, the support of $\mu_{Y|x'}$ is a singleton for every $x'\in \Omega'$; let us call the singleton element in the support $r_{x'}$. It remains to show that the map $g:x'\mapsto r_{x'}$ is measurable. Given any Borel set $\mathcal{A}\subset \Y$, we will first show the following:
	\begin{equation}\label{eq:perasfin}
	g^{-1}(\mathcal{A})=\{\tilde{x}\in\Omega':\mu_{Y|\tilde{x}}(\mathcal{A})=1\}.
	\end{equation}
	Towards this direction, assume that $v\in g^{-1}(\mathcal{A})$. This implies $g(v)=r_v\in\mathcal{A}$. Therefore the support of $\mu_{Y|v}$ is contained in $\mathcal{A}$ and so $\mu_{Y|v}(\mathcal{A})=1$. This shows that the left hand side of~\eqref{eq:perasfin} is contained in the right hand side. The other direction follows similarly. Finally, by definition of regular conditional distribution, $\mu_{Y|v}(\mathcal{A})$ is a measurable function in $v$ for every $\mathcal{A}$. Therefore, the right hand side of~\eqref{eq:perasfin} is a measurable set and so is the left. This completes the proof.
	\qed
	
	\subsection{Proof of~\cref{theo:consis}}\label{pf:theoconsis}
	Firstly by the strong law of large numbers for U-statistics, we immediately have the following consequences: 
	\begin{align*}
	&\frac{1}{n(n-1)}\sum_{i\neq j}K(Y_i,Y_j)\overset{a.s}{\longrightarrow}\mme[K(Y_1,Y_2)],\quad\mbox{and}\\
	&\frac{1}{2n(n-1)}\sum_{i\neq j} \lVert K(\cdot,Y_i)-K(\cdot,Y_j)\rVert_{\hk}^2\overset{a.s}{\longrightarrow}\mme[K(Y_1,Y_1)]-\mme[K(Y_1,Y_2)],
	\end{align*}
	provided $\mu_Y\in\tmk^2(\Y)$. Therefore it suffices to prove the convergence in probability or a.s. convergence of $\bar{T}_n$ defined below:
	\begin{align}\label{eq:mainterm}
	\bar{T}_n\equiv \bar{T}_n((X_1,Y_1),\ldots ,(X_n,Y_n))\coloneqq \frac{1}{n}\sum_{i=1}^n \sum_{j:(i,j)\in\emgn} \frac{K(Y_i,Y_j)}{d_i}.  
	\end{align}
	In particular, observe that it is sufficient to prove the following:
	\begin{enumerate}
		\item[(i)]  $\bar{T}_n-\mme[\bar{T}_n]\overset{\P}{\longrightarrow} 0$ (or $\bar{T}_n-\mme[\bar{T}_n]\overset{a.s.}{\longrightarrow} 0$).
		\item[(ii)] $\E[\bar{T}_n]\to \E\lVert g(X)\rVert_{\hk}^2$ where $g(X):=\E[K(\cdot,Y)|X].$
	\end{enumerate}
	To prove (i), let $(\tilde{X}_1,\tilde{Y}_1), \ldots , (\tilde{X}_n,\tilde{Y}_n)$ be $n$ i.i.d. samples from $\mu$ which are also drawn independently to $(X_1,Y_1),\ldots ,(X_n,Y_n)$. Set $\mathbf{X}^{i}\coloneqq (X_1,\ldots ,X_{i-1},\tilde{X}_i,X_{i+1},\ldots ,X_n)$, $\mathbf{Y}^{i}\coloneqq (Y_1,\ldots ,Y_{i-1},\tilde{Y}_i,Y_{i+1},\ldots ,Y_n)$ and $\bar{T}_N^i\coloneqq \bar{T}_n(\mathbf{X}^{i},\mathbf{Y}^{i})$. We will use $\mathcal{G}_n^i$ to denote the graph functional associated with $\mathbf{X}^{i}$. Recall the definitions of $r_n$, $q_n$, $t_n$ from (A2) and in the same spirit, define $\mathcal{C}_n^i:=(\emgn\cap\mathcal{E}(\mathcal{G}_n^i))\setminus \{(j,k):j=i\;\mbox{or}\;k=i\}$. Therefore $\mathcal{C}_{n}^i$ denotes the edges common to $\mathcal{G}_n$ and $\mathcal{G}_n^i$ which do not have $X_i$ or $\tilde{X}_i$ as one of their vertices. Let $V_n^i$ denote the set of vertices in $\mathcal{G}_n$ which have at least one edge outside $\emgn\cap\mathcal{C}_{n}^i$. Define $\tilde{V}_n^i$ similarly with $\emgn$ replaced by $\mathcal{E}(\mathcal{G}_n^i)$. In other words, $V_n^i$ denotes the set of vertices which have at least one edge in $(\mathcal{C}_{n}^i)^c$ (same holds for $\tilde{V}_{n}^i$). By our assumption, there are at most $q_n+t_n$ edges in $\mathcal{C}_{n}^i$. Therefore, note that $\max_{1\leq i\leq n}\max\{|V_n^i|,|\tilde{V}_n^i|\}\leq 2s_n$ a.s., where $s_n:=q_n+t_n$. Let $\tilde{d}_1^i,\ldots ,\tilde{d}_n^i$ be the degree sequence of $X_1,\ldots ,X_{i-1},\tilde{X}_i,X_{i+1},\ldots ,X_n$ in $\mathcal{E}(\mathcal{G}_n^i)$. Also suppose $Y^i_j=Y_j$ for $j\neq i$ and $Y^i_j=\tilde{Y}_j$ for $i=j$. Then we have:
	\begin{align*}
	\bar{T}_n-\bar{T}_n^1&=\frac{1}{n}\sum_{i\in V_n^1} \sum_{j:(i,j)\in \emgn\setminus\mathcal{C}_n^1} \frac{K(Y_i,Y_j)}{d_i}+\frac{1}{n}\sum_{i\in (V_n^1)^c} \sum_{j:(i,j)\in \emgn\setminus\mathcal{C}_n^1} \frac{K(Y_i,Y_j)}{d_i}\\ &-\frac{1}{n}\sum_{i\in \tilde{V}_n^1} \sum_{j:(i,j)\in \mathcal{E}(\mathcal{G}_n^1)\setminus\mathcal{C}_n^1} \frac{K(Y_i^1,Y_j^1)}{\tilde{d}_i^1}-\frac{1}{n}\sum_{i\in (\tilde{V}_n^1)^c} \sum_{j:(i,j)\in \mathcal{E}(\mathcal{G}_n^1)\setminus\mathcal{C}_n^1} \frac{K(Y_i^1,Y_j^1)}{\tilde{d}_i^1}
	\end{align*}
	Now, by construction, for all $i\in (V_n^1)^c$, the vertices $X_i$ have the same neighbors in $\mathcal{G}_n$ and $\mathcal{G}_n^i$. The same holds for all $i\in (\tilde{V}_n^1)^c$. As a result the summands above with $i\in (V_n^1)^c$ and $i\in (\tilde{V}_n^1)^c$ cancel out. We are thus left with:
	\begin{align}\label{eq:perturb}
	\big|\bar{T}_n-\bar{T}_n^1\big|&=\Bigg|\frac{1}{n}\sum_{i\in V_n^1} \sum_{j:(i,j)\in \emgn\setminus\mathcal{C}_n^1} \frac{K(Y_i,Y_j)}{d_i}\Bigg|+\Bigg|\frac{1}{n}\sum_{i\in \tilde{V}_n^1} \sum_{j:(i,j)\in \mathcal{E}(\mathcal{G}_n^1)\setminus\mathcal{C}_n^1} \frac{K(Y_i^1,Y_j^1)}{\tilde{d}_i^1}\Bigg|\nonumber \\ &\leq \frac{2s_n}{nr_n}\left(\max_{1\leq i\leq n}\lVert K(\cdot,Y_i)\rVert_{\hk}^2+\max_{1\leq i\leq n}\lVert K(\cdot,Y_i^1)\rVert_{\hk}^2\right).
	\end{align}
	By the Efron-Stein inequality (see~\cite{Boucheron2005}; also see~\cref{prop:Genef}) with $q=2$, we have:
	\begin{align*}
	\E\big(\bar{T}_n-\E[\bar{T}_n]\big)^2&\leq \sum_{i=1}^n \E\big(\bar{T}_n-\bar{T}_n^i\big)^2\nonumber \\& \leq \frac{16s_n^2}{nr_n^2}\E\left(\max_{1\leq i\leq n}\lVert K(\cdot,Y_i)\rVert_{\hk}^4\right)\overset{(a)}{\leq} \frac{16s_n^2}{nr_n^2}\cdot o\left(n^{\frac{2}{2+\epsilon/2}}\right)\overset{n\to\infty}{\longrightarrow}0
	\end{align*}
	where (a) follows from~\cref{prop:maxmom} and Assumption (A1). By Markov's inequality, we immediately have:
	\begin{equation}\label{eq:convprob}
	\bar{T}_n-\E[\bar{T}_n]\overset{\mathbb{P}}{\longrightarrow}0.
	\end{equation}
	Next, we will use~\cref{prop:Genef} (see~\cite[Theorem 2]{Loeb2004}) with $q=4$ coupled with Jensen's inequality to get:
	\begin{align*}
	\sum_{n=1}^{\infty}\E\big(\bar{T}_n-\E[\bar{T}_n]\big)^4&\leq \sum_{n=1}^{\infty} \left[\sum_{i=1}^n \E\big(\bar{T}_n-\bar{T}_n^i\big)^2\right]^2\nonumber \\& \leq \sum_{n=1}^{\infty}\frac{256s_n^4}{n^2r_n^4}\E\left(\max_{1\leq i\leq n}\lVert K(\cdot,Y_i)\rVert_{\hk}^8\right)\overset{(b)}{\leq} \sum_{n=1}^{\infty}\frac{256s_n^4}{n^2r_n^4}\cdot o\left(n^{\frac{4}{4+\epsilon/2}}\right)<\infty
	\end{align*}
	where (b) once again follows from~\cref{prop:maxmom} and assumption (A1). By combining Markov's inequality with the Borel-Cantelli lemma, we get:
	\begin{equation}\label{eq:convas}
	\bar{T}_n-\E[\bar{T}_n]\overset{a.s.}{\longrightarrow}0.
	\end{equation}
	It therefore remains to show (ii), i.e., $\E[\bar{T}_n]$ converges to $\E\lVert \E[K(\cdot,Y)|X]\rVert_{\hk}^2$. Recall the definition of $g(X)$ from (ii). Observe that, by exchangeability,
	\begin{align*}
	& \bigg|\E[\bar{T}_n]-\E\lVert g(X_1)\rVert_{\hk}^2\bigg| \\ &=\bigg|\E\Big[{d_1}^{-1}\sum_{j=1}^n \langle K(\cdot,Y_1),K(\cdot,Y_j)\rangle_{\hk}\mathbbm{1}((1,j)\in\emgn)\Big]-\E\lVert g(X_1)\rVert_{\hk}^2\bigg|\\&\leq \E\Big[{d_1}^{-1}\sum_{j=1}^n \big|\langle g(X_1),g(X_j)\rangle_{\hk}-\lVert g(X_1)\rVert_{\hk}^2\big|\mathbbm{1}((1,j)\in\emgn)\Big]\\ &\leq \E\left[\lVert g(X_1)\rVert_{\hk}\lVert g(X_{N(1)})-g(X_1)\rVert_{\hk}\right]\\ &\leq \sqrt{\E\lVert g(X_1)\rVert_{\hk}^2}\sqrt{\E \lVert g(X_{N(1)})-g(X_1)\rVert_{\hk}^2}.
	\end{align*}
	Note that $\E\lVert g(X_1)\rVert_{\hk}^2<\infty$. By~\cref{lem:techlemmconv}, $\lVert g(X_{N(1)})-g(X_1)\rVert_{\hk}^2\overset{\mathbb{P}}{\longrightarrow}0$. By~\cref{lem:techlemnbd}, $\E\lVert g(X_{N(1)})-g(X_1)\rVert_{\hk}^4<\infty$. Therefore, $\lVert g(X_{N(1)})-g(X_1)\rVert_{\hk}^2$ is a uniformly integrable sequence of random variables and consequently $\E \lVert g(X_{N(1)})-g(X_1)\rVert_{\hk}^2\overset{n\to\infty}{\longrightarrow}0$. This completes the proof. \qed
	
	\subsection{Proof of~\cref{cor:bdiffcon}}\label{pf:bdiffcon}
	
	Recall the definition of $\bar{T}_n$ from~\eqref{eq:mainterm}. Fix an arbitrary $t>0$. By combining~\eqref{eq:perturb}, the boundedness assumption on $K(\cdot,\cdot)$ with McDiarmid's bounded differences inequality (see~\cite[Theorem 6.5]{Boucheron2013}), we get:
	$$\P\big[\big|\bar{T}_n-\E\bar{T}_n\big|\geq t\big]\leq 2\exp(-C^*nt^2)$$
	where $C^*:=16M^2\left(\limsup_{n\to\infty} s_n/r_n\right)^2$, $s_n=q_n+t_n$ as defined before. This establishes~\eqref{eq:bdiffcon1}.
	
	By Lyapunov's inequality (see~\cite{Brown2000}), it suffices to establish~\eqref{eq:bdiffcon2} for $k\in \mathbb{N}$. By standard concentration inequalities for U-statistics, see for example,~\cite[Theorem 2]{Arcones1995}, we have:
	$$\P\left[\bigg|\frac{1}{n(n-1)}\sum_{i\neq j} K(Y_i,Y_j)-\E K(Y_1,Y_2)\bigg|\geq t\right]\leq C_1\exp(-C_2nt^2)$$ and 
	\begin{align*}
	& \P\left[\bigg|\frac{1}{2n(n-1)}\sum_{i\neq j} \lVert K(\cdot,Y_i)-K(\cdot,Y_j)\rVert_{\hk}^2-(\E K(Y_1,Y_1)-\E K(Y_1,Y_2))\bigg|\geq t\right] & \\
	& \leq C_1\exp(-C_2nt^2)
	\end{align*}
	where $C_1$ and $C_2$ are positive constants (not depending on $t$). The proof can then be completed by invoking~\cref{prop:basicmom} with $A_n:=\bar{T}_n-(n(n-1))^{-1}\sum_{i\neq j} K(Y_i,Y_j)$, $a_n:=\E K(Y_1,Y_{N(1)})-\lVert \E K(\cdot,Y)\rVert_{\hk}^2$, $B_n:=(2n(n-1))^{-1}\sum_{i\neq j} \lVert K(\cdot,Y_i)-K(\cdot,Y_j)\rVert_{\hk}^2$ and $b:=(1/2)\E\lVert K(\cdot,Y_1)-K(\cdot,Y_2)\rVert_{\hk}^2$.
	\qed
	
	\subsection{Proof of~\cref{prop:vercond}}\label{pf:vercond}
	\emph{Part (i).} The proof is a consequence of combining different existing results from the theory of \emph{stabilizing graphs} and \emph{stabilizing graph functionals} (see~\cite{Penrose2003} for details). Let $\mathcal{G}_n$ denote the minimum spanning tree of $X_1,\ldots ,X_n$. As $\mu_X$ is absolutely continuous, $\mathcal{G}_n$ is stabilizing (see~\cite[Lemma 2.1]{Penrose2003}). For any $e\in \emgn$, let $|e|$ denote the edge weight of $e$. Note that $\# \emgn = n-1$. For any $M>0$, by the exchangeability of $X_i$'s, the following holds:
	$$\P(n^{1/d}\lVert X_1-X_{N(1)}\rVert_2\geq M)=\frac{1}{n-1}\sum_{e\in\emgn} \E\left[\mathbf{1}(n^{1/d}|e|\geq M)\right].$$
	As $\mathbf{1}(n^{1/d}|e|\geq M)$ is uniformly bounded by $1$, we can use~\cite[Theorem 2.3, part (i)]{Penrose2003} (also see~\cite{Penrose1996}) along with the above display to get $n^{1/d}\lVert X_1-X_{N(1)}\rVert_2=\mathcal{O}_p(1)$. This implies $\lVert X_1-X_{N(1)}\rVert_2\overset{\P}{\to}0$ and establishes (A1). Further, by~\cite[Lemma 4]{Aldous1992}, there exists a constant $B(d)$ such that $\max_{1\leq i\leq n} d_i\leq B(d)$ with probability $1$. By choosing $r_n=1$ and $t_n=B(d)$ establishes (A3). Finally, note that changing one point with another only alters the minimum spanning tree in the neighborhood of the two points (see e.g.,~\cite[Lemma 2.1]{Steele1987}), as is the case with many other stabilizing graphs as presented in~\cite{Penrose2003}. Therefore, by choosing $q_n=2B(d)$ establishes (A2).
	
	\emph{Part (ii).} The proof is the same for the directed/undirected cases. When $k=\mathcal{O}(1)$, the result follows once again from the theory of stabilizing graphs (see~\cite[Theorem 2.4]{Penrose2003}), as we proved in part (i). Unfortunately, when $k$ grows with $n$, the corresponding nearest neighbor graph is no longer stabilizing, as has been pointed out in~\cite[Section 4.3]{bbb2019}. So we provide a different proof. Also we will write $k\equiv k_n$ in the proof, to make the dependence on $n$ explicit. Recall that we have assumed $k_n\leq Cn^{1-\delta}$ for some $C\geq 1$ and some $\delta\in (0,1]$.
	
	As every vertex has at least $k_n$ neighbors, we can choose $r_n=k_n$. By~\cite[Lemma 1]{Jaffe2020}, there exists a constant $C(d)$ such that $\max_{1\leq i\leq n} d_i\leq C(d)k_n$ w.p.~1. Therefore, (A3) holds with $t_n=C(d)k_n$. It is easy to check once again that altering between two points only changes the NNG in the neighborhood of the points, which implies (A2) holds by choosing $q_n=2C(d)k_n$. For establishing (A1), we will use the same argument as has been used in the proof of~\cref{theo:rateres}. Towards that direction, pick an arbitrary $M>0$. By repeating the same argument as that of~\eqref{eq:callearly}, we get the following set of inequalities: 
	\begin{align*}
	&\;\;\;\P(\lVert X_1-X_{N(1)}\rVert_2\geq \epsilon)\\ &\lesssim \P(\lVert X_1-X_{N(1)}\rVert_2\geq \epsilon, \lVert X_1\rVert_2\leq M,\lVert X_{N(1)}\rVert_2\leq M)+\P(\lVert X_1\rVert_2\geq M)\nonumber \\ &\lesssim n^{-1}+\frac{k_n\log{n}}{n}\cdot \frac{M^p}{\epsilon^p}+\P(\lVert X_1\rVert_2\geq M).
	\end{align*}
	In the above sequence of displays, all the hidden (by $\lesssim$) constants are further free of $M$.
	Finally, by taking $n\to\infty$ followed by taking $M\to\infty$ (the order of taking limits is important here) establishes (A1).
	\subsection{Proof of~\cref{theo:nullclt}}\label{pf:theonullclt}
	
	Throughout this proof, we will use the $\lesssim$ symbol to hide constants which are uniform over $\mathcal{J}_{\theta}$ for $\theta$ fixed. Also we will use $\mathcal{F}_n$ to denote the $\sigma$-field generated by $(X_1,\ldots ,X_n)$. To begin, note that,
	\begin{equation}\label{eq:conmeanzero}
	\E [N_n|\mathcal{F}_n]=\E[K(Y_1,Y_2)]\cdot \left(\left(\frac{1}{n}\sum_{i=1}^n \frac{1}{\td_i}\sum_{j:(i,j)\in\emgn}1\right)-1\right)=0.
	\end{equation}
	Set $a:=\E K^2(Y_1,Y_2)$, $b:=\E[K(Y_1,Y_2)K(Y_1,Y_3)]$,~$c:=\E [K(Y_1,Y_2) K(Y_3,Y_4)]$ and write, 
	{\small    	\begin{align}\label{eq:vareq}
		&\mathrm{Var} (N_n)=\E [N_n^2]=\underbrace{\frac{1}{n}\E\left(\sum_{i=1}^n \frac{1}{\td_i} \sum_{j:(i,j)\in \emgn} K(Y_i,Y_j)\right)^2}_{\Gamma_1}+\underbrace{\frac{1}{n(n-1)^2}\E\left(\sum_{i\neq j} K(Y_i,Y_j)\right)^2}_{\Gamma_3}\nonumber \\&-2\cdot\underbrace{\frac{1}{n(n-1)}\E\left[\left(\sum_{i=1}^n \td_i^{-1} \sum_{j:(i,j)\in \emgn} K(Y_i,Y_j)\right)\left(\sum_{i\neq j} K(Y_i,Y_j)\right)\right]}_{\Gamma_2},
		\end{align}}
	where $\Gamma_1,\Gamma_2,\Gamma_3$ have been simplified below:
	{\small    	\begin{align}\label{eq:maint1}
		\Gamma_1&=\frac{1}{n}\E\Bigg[a\left(\sum_{i=1}^n \frac{1}{\td_i}+\sum_{i,j:(i,j)\in\emgn} \frac{1}{\td_i\td_j}\right)+b\Bigg(\sum_{\substack{(i,j,t)\ \mathrm{ distinct}:\\ (i,j),(i,t)\in\emgn}}\frac{1}{\td_i^2}+2\sum_{\substack{(i,j,s)\ \mathrm{ distinct}:\\ (i,j),(i,s)\in\emgn}}\frac{1}{\td_i\td_s}\nonumber \\ 
		&\qquad \qquad +\sum_{\substack{(i,j,s)\ \mathrm{ distinct}:\\ (i,j),(i,s)\in\emgn}}\frac{1}{\td_j\td_s}\Bigg)+c\left(\sum_{\substack{(i,j,s,t)\ \mathrm{ distinct}:\\ (i,j),(s,t)\in\emgn}}\frac{1}{\td_i\td_s}\right)\Bigg]\nonumber \\ &=a\E\left[\frac{1}{\td_1}+\frac{1}{\td_1}\sum_{j:(1,j)\in\emgn}\frac{1}{\td_j}\right]+\frac{b}{n}\E\Bigg[\sum_{i,t}\frac{1}{\td_i^2}\cdot \mathbf{1}((i,t)\in\emgn)(\td_i-\mathbf{1}((i,t)\in\emgn))\nonumber\\
		&\qquad +2\sum_{i,t}\frac{1}{\td_i\td_t}\mathbf{1}((i,t)\in\emgn)(\td_i-\mathbf{1}((i,t)\in\emgn))+\sum_{i,j,s} \frac{1}{\td_i\td_s}\cdot\mathbf{1}((i,j),(j,s)\in\emgn)\nonumber \\ 
		& \qquad -\sum_{i=1}^n \frac{1}{\tilde{d}_i}\Bigg]+\frac{c}{n}\E\left[\sum_{i,s}\frac{1}{\td_i\td_s}\sum_{\substack{j,t:j\neq i,s\\ t\neq i,s,\ j\neq t}}\mathbf{1}((i,j),(s,t)\in\emgn)-\sum_{\substack{(i,j,t)\ \mathrm{ distinct}:\\ (i,j),(i,t)\in\emgn}}\frac{1}{\td_i^2}\right]\nonumber \\ 
		&=a\E\left[\frac{1}{\td_1}+\frac{1}{\td_1}\sum_{j:(1,j)\in\emgn}\frac{1}{\td_j}\right]+b\Bigg[3+\frac{1}{\td_1}\sum_{j}\frac{\tgn(1,j)}{\td_j}-\frac{2}{\td_1}-\frac{2}{\td_1}\sum_{j:(1,j)\in\emgn}\frac{1}{\td_j}\Bigg]\nonumber \\ & \qquad \qquad +c\E\left[n-3+\frac{1}{\td_1}+\frac{1}{\td_1}\sum_{j:(1,j)\in\emgn}\frac{1}{\td_j}-\frac{1}{\td_1}\sum_{j}\frac{\tgn(1,j)}{\td_j}\right].
		\end{align}}
	Similar calculations show that $\Gamma_2$ and $\Gamma_3$ simplify as follows:
	\begin{align}\label{eq:maint2}
	\Gamma_2=\frac{2a}{n-1}+4b\cdot\frac{n-2}{n-1}+c\cdot \frac{(n-2)(n-3)}{n-1},
	\end{align}
	\begin{align}\label{eq:maint3}
	\Gamma_3=\frac{2a}{n-1}+4b\cdot\frac{n-2}{n-1}+c\cdot \frac{(n-2)(n-3)}{n-1}.
	\end{align}
	Plugging~\eqref{eq:maint1},~\eqref{eq:maint2}~and~\eqref{eq:maint3} in~\eqref{eq:vareq}, we get:
	{\small \begin{align}\label{eq:mainvareq}
		\mathrm{Var}(N_n)&=\E\left[\frac{1}{\td_1}+\frac{1}{\td_1}\sum_{j:(1,j)\in\emgn} \frac{1}{\td_j}\right](a-2b+c)\nonumber \\& \quad \quad +\left[\E\left(\frac{1}{\td_1}\sum_{j}\frac{\tgn(1,j)}{\td_j}\right)-1\right](b-c) -\frac{2(a-2b+c)}{n-1}.
		\end{align}}
	Note that by using Assumption (A3), the following bounds follow:
	\begin{equation}\label{eq:boundcount}
	\frac{1}{\tilde{d}_1}\sum_{l\neq 1}\frac{\tgn(1,l)}{\tilde{d}_l}\leq \frac{t_n^2}{r_n^2}\lesssim 1\; , \; \frac{1}{\tilde{d}_1}\sum_{(1,l)\in\emgn}\frac{1}{\tilde{d}_l}\lesssim \frac{t_n}{r_n^2}\lesssim 1\; ,\; \frac{1}{\tilde{d}_1}\leq \frac{1}{r_n}\leq 1,
	\end{equation}
	which implies $\limsup\limits_{n\to\infty}\mathrm{Var}(N_n)<\infty$ and establishes~\eqref{eq:d'accord}.
	
	\noindent In order to prove~\eqref{eq:mainres}, we introduce some notation first. Towards this direction, define: 
	$$\mathcal{V}_n:=\frac{1}{\sqrt{n}(n-1)}\sum_{i\neq j} K(Y_i,Y_j),\qquad  g(Y_i):=2\E[K(Y,Y_i)|Y_i]-\E [K(Y_1,Y_2)]$$
	where $Y\sim \mu_{Y}$ and is independent of $Y_1,\ldots ,Y_n$. By some standard U-statistics projection theory (see for example,~\cite[Theorem 12.3]{van1998}), the following holds:
	{\small \begin{equation}\label{eq:project}
		\limsup\limits_{n\to\infty} \sup\limits_{(\mathcal{G},\mu)\in \mathcal{J}_{\theta}}n\E\left(\frac{1}{\sqrt{n}(n-1)}\sum_{i\neq j} K(Y_i,Y_j)-\frac{1}{\sqrt{n}}\sum_{i=1}^n g(Y_i)\right)^2\lesssim 1.
		\end{equation}}
	In fact, the bound in~\cite[Theorem 12.3]{van1998} is not explicitly stated in the form as above. We have therefore added a formal proof in~\cref{lem:projerrorctrl} for completion. Next define, $$V_i:=\frac{1}{\sqrt{n}}\left(\frac{1}{\tilde{d}_i}\sum_{j:(i,j)\in\emgn} K(Y_i,Y_j)-g(Y_i)\right), \qquad \tilde{N}_n:=\sum_{i=1}^n V_i,$$
	and note that,
	\begin{equation}\label{eq:rewrite}
	N_n=\tilde{N}_n-\underbrace{\frac{1}{\sqrt{n}(n-1)}\sum_{i\neq j} K(Y_i,Y_j)+\frac{1}{\sqrt{n}}\sum_{i=1}^n g(Y_i)}_{\tilde{\mathcal{N}}_n}.
	\end{equation}
	Observe that,
	\begin{align}\label{eq:meanzero}
	\E[V_i|\mathcal{F}_n]=\E [K(Y_1,Y_2)]\left[\frac{1}{\td_i}\sum_{j:(i,j)\in\emgn} 1-1\right]=0. 
	\end{align}
	for all $i=1,2,\ldots ,n$. The next step in obtaining normality is to observe that although the $V_i$'s are not independent, their dependence is, in a way local. This notion can be formalized by the construction of a dependency graph (see~\cite{Baldi1989,Rinot2003} for details) which we illustrate below.
	
	Let us construct a graph, say $\mathcal{D}(\mathcal{G}_n)$ depending on $\mathcal{G}_n$ as follows: given $i\neq j$, we say that there is an edge between $V_i$ and $V_j$ in $\mathcal{D}(\mathcal{G}_n)$ iff there is a path of length $\leq 2$ joining $X_i$ and $X_j$ in $\mathcal{G}_n$. Note that whenever $V_i$ and $V_j$ are not connected in $\mathcal{D}(\mathcal{G}_n)$, there are no common neighbors between the corresponding $X_i$ and $X_j$ in $\mathcal{G}_n$ which further implies that $V_i$ and $V_j$ are independent conditioned on $\mathcal{F}_n$. As a result, $\mathcal{D}(\mathcal{G}_n)$ is a dependency graph  with maximum degree $\lesssim (\log{n})^{2\gamma}$. We are now in a position to apply~\cref{prop:cltdep} (see~\cite[Theorem 2.7]{Chen2004}), which yields:
	\begin{align}\label{eq:berbd}
	\sup_{z\in\R} \Bigg|\P \left(\frac{\tilde{N}_n}{\sqrt{\mathrm{Var}(\tilde{N}_n|\mathcal{F}_n)}}\leq z\big| \mathcal{F}_n\right)-\Phi(z)\Bigg|\lesssim \min\left\{\frac{(\log{n})^{20\gamma}\sqrt{n}\E[\sum_{i=1}^n |V_i|^3|\mathcal{F}_n]}{(n^{1/3}\mathrm{Var}(\tilde{N}_n|\mathcal{F}_n))^{3/2}},2\right\}
	\end{align}
	almost surely. By an application of the standard power mean inequality, we get:
	\begin{align}\label{eq:prelbd}
	\sqrt{n}\E\left[\sum_{i=1}^n |V_i|^3|\mathcal{F}_n\right]\lesssim \frac{\sqrt{n}}{n\sqrt{n}}\E|K(Y_1,Y_2)|^3\sum_{i=1}^n \frac{1}{\td_i}\sum_{j:(i,j)\in\emgn} 1\leq \E|K(Y_1,Y_2)|^3.
	\end{align}
	By combining~\eqref{eq:berbd}~and~\eqref{eq:prelbd} with the tower property, we get:
	{\small     	\begin{align}\label{eq:bound1}
		&\sup\limits_{(\mathcal{G},\mu)\in \mathcal{J}_{\theta}}\sup_{z\in\R} \Bigg|\P \left(\frac{\tilde{N}_n}{\sqrt{\mathrm{Var}(\tilde{N}_n|\mathcal{F}_n)}}\leq z\right)-\Phi(z)\Bigg|\lesssim \sup\limits_{(\mathcal{G},\mu)\in \mathcal{J}_{\theta}} \E\left[\min\left\{\frac{(\log{n})^{20\gamma}}{(n^{1/3}\mathrm{Var}(\tilde{N}_n|\mathcal{F}_n))^{3/2}},2\right\}\right]\nonumber \\&\lesssim (\log{n})^{20\gamma}n^{\frac{3\epsilon-1}{2}}+\sup\limits_{(\mathcal{G},\mu)\in \mathcal{J}_{\theta}}\P\left(n^{\epsilon}\mathrm{Var}(\tilde{N}_n|\mathcal{F}_n)\leq \frac{1}{2}\right).
		\end{align}}
	In order to show that the right hand side of~\eqref{eq:bound1} converges to $0$, it suffices to show (by Markov's inequality) the following:
	\begin{equation}\label{eq:claimtoprove}
	\sup\limits_{(\mathcal{G},\mu)\in \mathcal{J}_{\theta}} n^{\epsilon}\E\Big|\mathrm{Var}(\tilde{N}_n|\mathcal{F}_n)-\mbox{Var}(N_n)\Big|\to 0.
	\end{equation} 
	Towards this direction, let $\mathcal{H}_n$ denote the $\sigma$-algebra generated by the unordered set $(Y_1,\ldots ,Y_n)$. It is easy to check that $\E[N_n|\mathcal{F}_n]=\E[\tilde{N}_n|\mathcal{F}_n]=\E[N_n|\mathcal{F}_n,\mathcal{H}_n]=0$. As $\tilde{\mathcal{N}}_n$ is measurable with respect to $\mathcal{H}_n$ (recall~\eqref{eq:rewrite}), we get:
	\begin{eqnarray}
	\mbox{Var}(\tilde{N}_n|\mathcal{F}_n) & =& \mbox{Var}(N_n|\mathcal{F}_n)+\mbox{Var}(\tilde{\mathcal{N}}_n)+2\E\left[\E[N_n\tilde{\mathcal{N}}_n|\mathcal{F}_n,\mathcal{H}_n]|\mathcal{F}_n\right] \nonumber \\ & = & \mbox{Var}(N_n|\mathcal{F}_n)+\mbox{Var}(\tilde{\mathcal{N}}_n) \label{eq:anova1}
	\end{eqnarray}
	and similarly,
	\begin{equation}\label{eq:anova2}
	\mbox{Var}(\tilde{N}_n)=\mbox{Var}(N_n)+\mbox{Var}(\tilde{\mathcal{N}}_n).
	\end{equation}
	Also, in a similar vein as~\eqref{eq:mainvareq}, we have the following:
	\begin{equation}\label{eq:convareq}
	\mbox{Var}(N_n|\mathcal{F}_n)=(\tilde{g}_1+\tilde{g}_3)(a-2b+c)+(\tilde{g}_2-1)(b-c)-\frac{2(a-2b+c)}{n-1}.
	\end{equation}
	By using~\eqref{eq:project},~\eqref{eq:anova1},~\eqref{eq:anova2},~\eqref{eq:convareq}~and~\eqref{eq:mainvareq}, we get:
	\begin{align}\label{eq:proveclaim}
	&\sup\limits_{(\mathcal{G},\mu)\in \mathcal{J}_{\theta}} n^{\epsilon}\E\Big|\mbox{Var}(N_n|\mathcal{F}_n)-\mbox{Var}(N_n)\Big|\nonumber \\ &\lesssim o(1)+\left(\sup\limits_{(\mathcal{G},\mu)\in \mathcal{J}_{\theta}} n^{2\epsilon}\E\left(\mbox{Var}(N_n|\mathcal{F}_n)-\mbox{Var}(N_n)\right)^2\right)^{1/2}\nonumber \\ &\lesssim o(1)+\frac{1}{\sqrt{n}}+\Bigg[n^{2\epsilon}\sup\limits_{(\mathcal{G},\mu)\in \mathcal{J}_{\theta}}\Bigg\{\E\left(\tilde{g}_1-\E\left[\frac{1}{\td_1}\right]\right)^2 +\E\left(\tilde{g}_2-\E\left[\frac{1}{\td_1}\sum_{l\neq 1}\frac{\tgn(1,l)}{\td_l}\right]\right)^2\nonumber \\ & \qquad \qquad + \E\left(\tilde{g}_3-\E\left[\frac{1}{\td_1}\sum_{(1,l)\in\emgn}\frac{1}{\tilde{d}_l}\right]\right)^2\Bigg\}\Bigg]^{1/2}.
	\end{align}
	Note that each term within the braces of~\eqref{eq:proveclaim} is $\lesssim n^{-1}(\log{n})^{4\gamma}$ uniformly in $\mathcal{J}_{\theta}$ by using~\cref{prop:Genef} coupled with a similar calculation as in~\eqref{eq:perturb}. We omit the details for brevity. However, note that this observation implies that the right hand side of~\eqref{eq:proveclaim} converges to $0$ uniformly in $\mathcal{J}_{\theta}$ and consequently proves~\eqref{eq:claimtoprove}. By~\eqref{eq:bound1}, this further implies that:
	\begin{align}\label{eq:supconv}
	\limsup\limits_{n\to\infty}\sup\limits_{(\mathcal{G},\mu)\in \mathcal{J}_{\theta}}\sup_{z\in\R} \Bigg|\P \left(\frac{\tilde{N}_n}{\sqrt{\mbox{Var}(\tilde{N}_n|\mathcal{F}_n)}}\leq z\right)-\Phi(z)\Bigg| = 0.
	\end{align}
	Next we show that the $\mbox{Var}(\tilde{N}_n|\mathcal{F}_n)$ term in~\eqref{eq:supconv} can be replaced with $\hS_n^2$. In particular, it suffices to show that there exists a sequence $\epsilon_n\to 0$ as $n\to\infty$ such that:
	\begin{equation}\label{eq:needtoprove}
	\sup\limits_{(\mathcal{G},\mu)\in \mathcal{J}_{\theta}}\P\left(\Bigg|\frac{\hS_n^2}{\mbox{Var}(\tilde{N}_n|\mathcal{F}_n)}-1\Bigg|\geq \epsilon_n\right)\overset{n\to\infty}{\longrightarrow}0.
	\end{equation}
	We will prove the above with $\epsilon_n=n^{-1/6}$. Towards this direction, note that:
	\begin{align}\label{eq:slutskystep}
	&\;\;\;\sup\limits_{(\mathcal{G},\mu)\in \mathcal{J}_{\theta}}\P\left(\Bigg|\frac{\hS_n^2}{\mbox{Var}(\tilde{N}_n|\mathcal{F}_n)}-1\Bigg|\geq n^{-1/6}\right)\nonumber \\ &\lesssim \sqrt{n^{2\epsilon+1/3}\sup\limits_{(\mathcal{G},\mu)\in \mathcal{J}_{\theta}}\E\left[(\hS_n^2-\mbox{Var}(N_n|\mathcal{F}_n))^2\right]}+n^{\epsilon+\frac{1}{6}}\mbox{Var}(\tilde{\mathcal{N}}_n) \\ & \qquad \qquad +\sup\limits_{(\mathcal{G},\mu)\in \mathcal{J}_{\theta}}\P(n^{\epsilon}\mbox{Var}(\tilde{N}_n|\mathcal{F}_n)\leq 1/2)\nonumber \\ &\overset{(i)}{\lesssim} \sqrt{n^{2\epsilon+\frac{1}{3}}\sup\limits_{(\mathcal{G},\mu)\in \mathcal{J}_{\theta}}\max\left\{\frac{1}{n},\E(\tilde{a}-a)^2,\E(\tilde{b}-b)^2,\E(\tilde{c}-c)^2\right\}} \\ & \qquad \qquad +o(1)+\sup\limits_{(\mathcal{G},\mu)\in \mathcal{J}_{\theta}}\P(n^{\epsilon}\mbox{Var}(\tilde{N}_n|\mathcal{F}_n)\leq 1/2)\nonumber \\ &\lesssim n^{\epsilon-1/3}+o(1)+\sup\limits_{(\mathcal{G},\mu)\in \mathcal{J}_{\theta}}\P(n^{\epsilon}\mbox{Var}(\tilde{N}_n|\mathcal{F}_n)\leq 1/2)
	\end{align}
	where the last line follows from the fact that $\max\left\{\E(\tilde{a}-a)^2,\E(\tilde{b}-b)^2,\E(\tilde{c}-c)^2\right\}\lesssim n^{-1}$ uniformly over $\mathcal{J}_{\theta}$, which once again, is a consequence of standard U-statistics theory (see~\cite[Theorem 12.3]{van1998}) on observing that $\tilde{a}$, $\tilde{b}$ and $\tilde{c}$ are U-statistics which are unbiased estimates of $a$, $b$ and $c$ respectively. Here, $(i)$ follows from~\eqref{eq:project}. Further, the third term on the right hand side of~\eqref{eq:slutskystep} is the same as the second term in~\eqref{eq:bound1}. We have already proved that this term converges to $0$ using~\eqref{eq:claimtoprove}~and~\eqref{eq:proveclaim}. Combining this observation with~\eqref{eq:supconv},~\eqref{eq:rewrite}~and~\eqref{eq:project} completes the proof.
	\qed
	
	\subsection{Proof of~\cref{prop:nullcltesscomp}}\label{pf:nullcltlesscomp}
	The proof is very similar to that of~\cref{theo:nullclt}. We will only highlight the important differences here. Throughout this proof $(i,j)\in\mathcal{E}(\tilde{\mathcal{G}}_n)$ will be simplified as $i\sim j$, and all indices are to be interpreted modulo $n$ as stated in the statement of the proposition. The first important difference is that between $\hS_n$ and $\hS_{n,\mathrm{lin}}^2$. To understand this, we will first calculate $\mathrm{Var}(\linn)$. Note that $\E \linn=0$ following the same calculation as in~\eqref{eq:conmeanzero}. Next we break down $\E(\linn)^2=\Gamma_1^{\ml}-2\Gamma_2^{\ml}+\Gamma_3^{\ml}$ in the same way as was done in~\eqref{eq:vareq} with $\E N_n^2$ being expressed as $\Gamma_1-2\Gamma_2+\Gamma_3$. Note that $\Gamma_1=\Gamma_1^{\ml}$. Also $\Gamma_3^{\ml}$ is only a function of $\mu_Y$ which makes it easier to simplify, and it turns out to be:
	\begin{equation}\label{eq:t3}
	\Gamma_3^{\ml}=a+2b+c(n-3).
	\end{equation}
	For $\Gamma_2^{\ml}$, we will first define the following two quantities which arise naturally,
	\begin{align*}
	T_{+1}:=\frac{1}{n}\sum_{i=1}^n \frac{1}{\tilde{d}_i}\mathbf{1}(i\sim (i+1)),\quad T_{-1}:=\frac{1}{n}\sum_{i=1}^n \frac{1}{\tilde{d}_i}\mathbf{1}(i\sim (i-1))
	\end{align*}
	As $\sup_{(\tilde{G},\tilde{\mu})\in\mathcal{J}_{\theta}} \sup_{i,j}\P(i\sim j)\lesssim n^{-1}$, the following property of $T_{+1}$ and $T_{-1}$ follows trivially,
	\begin{align}\label{eq:smallcount}
	\sup_{\nst} \max\{\E T_{+1}^2,\E T_{+1},\E T_{-1}^2,\E T_{-1}\}\lesssim n^{-1}.
	\end{align}
	Recall the definitions of $a$, $b$ and $c$ from the proof of~\cref{theo:nullclt}. After some tedious simplification, the expression for $\Gamma_2^{\ml}$ turns out to be:
	\begin{equation}\label{eq:t2}
	\Gamma_2^{\ml}=a(\E T_{+1}+\E T_{-1})+b(4-2\E T_{+1}-2\E T_{-1})+c(n-4+\E T_{+1}+\E T_{-1}).
	\end{equation}
	By using~\eqref{eq:t2},~\eqref{eq:t3},~\eqref{eq:maint1},~\eqref{eq:smallcount} and setting $g_1:=\E[\tilde{d}_1^{-1}]$, $g_2:=\E[\tilde{d}_1^{-1}\sum_{j\sim 1} \tilde{d}_j^{-1}]$ and $g_3:=\E[\tilde{d}_1^{-1}\sum_{j} \tgn(1,j)/\tilde{d}_j]$, we get:
	\begin{equation}\label{eq:linvar}
	\sup_{\nst}\Big|\mathrm{Var}(\linn)-a(g_1+g_3+1)-b(g_2-2g_1-2g_3-3)-c(2+g_1+g_3-g_2)\Big|\lesssim n^{-1}.
	\end{equation}
	Therefore, $\mathrm{Var}(\linn)=\mathcal{O}(1)$ as in~\cref{theo:nullclt}.
	
	The next step involves constructing the dependency graph to establish a CLT using~\cite[Theorem 2.7]{Baldi1989} (also se~\cref{prop:cltdep}). Observe that:
	$$\linn=\frac{1}{n}\sum_{i=1}^n \left(\frac{1}{\tilde{d}_i}\sum_{j:j\sim i} (K(Y_i,Y_j)-K(Y_i,Y_{i+1}))\right).$$
	Let us construct a dependency graph, say $\mathcal{D}(\tilde{\mathcal{G}}_n)$ depending on $\tilde{\mathcal{G}}_n$ as follows: given $i\neq j$, we say that there is an edge between two vertices $V_i$ and $V_j$ in $\mathcal{D}(\tilde{\mathcal{G}}_n)$ iff there is a path of length $\leq 2$ joining $X_i$ and $X_j$ in $\mathcal{G}_n$ or if $j=i-1,i,i+1$. Using this dependency graph, one can repeat the same set of calculations as in the proof of~\cref{theo:nullclt} right up to~\eqref{eq:anova1} where a small observation needs to be made. 
	
	Recall that $\mathcal{F}_n$ was defined as the $\sigma$-field generated by $\{X_1,\ldots ,X_n\}$ and $\mathcal{H}_n$ was defined as the $\sigma$-field generated by the unordered set $\{Y_1,\ldots ,Y_n\}$. In~\cref{theo:nullclt}, we used that $\E[\sum_{i\neq j} K(Y_i,Y_j)|\mathcal{H}_n]=\sum_{i\neq j} K(Y_i,Y_j)$ which follows because $\sum_{i\neq j} K(Y_i,Y_j)$ is permutation invariant, which $\sum_{i=1}^n K(Y_i,Y_{i+1})$ is unfortunately not. However, it is easy to check that $\E[n^{-1}\sum_{i=1}^n K(Y_i,Y_{i+1})|\mathcal{H}_n]=(n(n-1))^{-1}\sum_{i\neq j} K(Y_i,Y_j)$ which is all we needed anyway. Therefore, one can proceed from~\eqref{eq:anova1} to get:
	\begin{align}\label{eq:newconvareq}
	\mathrm{Var}(\linn|\mathcal{F}_n)&=a(\tilde{g}_1+\tilde{g}_3+1-2T_{-1}-2T_{+1})+b(\tilde{g}_2-2\tilde{g}_1-2\tilde{g}_3+4T_{-1}+4T_{+1})\nonumber \\&+c(2+\tilde{g}_1+\tilde{g}_3-\tilde{g}_2-2T_{-1}-2T_{+1}).
	\end{align}
	Using~\eqref{eq:newconvareq}, the rest of the proof follows verbatim from the proof of~\cref{theo:nullclt} if we can show that $\sup_{\nst}\max\{\E(\tilde{a}^{\ml}-a)^2,\E(\tilde{b}^{\ml}-b)^2,\E(\tilde{c}^{\ml}-c)^2\}\lesssim n^{-1}$. This follows from simple moment computations or one could use the Efron-Stein inequality, see~\cite{Boucheron2005} or equivalently~\cref{prop:Genef} with $q=2$. Similar moment computations have been carried out multiple times in this paper (e.g., see~\cref{lem:projerrorctrl}) and we omit the details for brevity.\qed
	\subsection{Proof of~\cref{theo:rateres}}\label{pf:rateres}
	In this proof, we will use $C$ to denote absolute constants which might change from one line to another. Recall the definition of $\bar{T}_n$ from~\eqref{eq:mainterm}. Also note that when $\mathcal{G}$ is the $k_n$-NNG, then one can chose $q_n$ (from assumption (A2)) as $4t_n$. As a result, by~\eqref{eq:perturb}, we have:
	\begin{align}\label{eq:variancebound}
	\mathrm{Var}(\bar{T}_n)\lesssim \frac{t_n^2}{nk_n^2}\E\left(\max_{1\leq i\leq n} \lVert K(\cdot,Y_i)\rVert_{\hk}^4 \right)\lesssim \frac{t_n^2}{nk_n^2}\inf_{\gamma>0}\left(\gamma^2+n\int_{\gamma}^{\infty} t\P(K(Y_1,Y_1)\geq t)\,dt\right)
	\end{align}
	where the last line follows from a simple union bound. Next we control the bias, i.e., $\E\bar{T}_n-\E\lVert \E[K(\cdot,Y)|X]\rVert_{\hk}^2$. Recall the definition of $g(\cdot)$ from assumption (R3). Observe that:
	\begin{align}\label{eq:biasbound}
	&|\E\bar{T}_n-\E\lVert \E[K(\cdot,Y)|X]\rVert_{\hk}^2|=|\E\left[\langle g(X_1),g(X_{N(1)})\rangle_{\hk}-\lVert g(X_1)\rVert_{\hk}^2\right]|\nonumber \\ &\;\;\;\;\lesssim \E\left[\left(1+\rho(X_1,\xs)^{\beta_1}+\rho(X_{N(1)},\xs)^{\beta_1}\right)\rho(X_1,X_{N(1)})^{\beta_2}\right]\nonumber \\ &\;\;\;\;\lesssim \sqrt{\E\left[1+\rho(X_1,\xs)^{2\beta_1}+\rho(X_{N(1)},\xs)^{2\beta_1}\right]}\sqrt{\E\rho(X_1,X_{N(1)})^{2\beta_2}}
	\end{align}
	where the first line uses the lipschitz type assumption on $g(\cdot,\cdot)$ (see (R3)) and and the last line follows from the Cauchy-Schwartz inequality. As $\rho(X_1,X_{N(1)})\leq \rho(X_1,\xs)+\rho(X_{N(1)},\xs)$, both terms of the right hand side of the above display are $\mathcal{O}(1)$ using assumption (R2) and~\cref{lem:techlemnbd}.
	
	Next, by another application of Cauchy-Schwartz inequality and~\cref{lem:techlemnbd}, for any $C>0$, we have:
	\begin{align*}
	\E\rho(X_1,X_{N(1)})^{2\beta_2}&\lesssim \E\left[\rho(X_1,X_{N(1)})^{2\beta_2}\mathbf{1}\left(\max\{\rho(X_1,\xs),\rho(X_{N(1)},\xs)\}\leq C(\log{n})^{1/\alpha}\right)\right]\\ &+\sqrt{\P(\rho(X_1,\xs)\geq C(\log{n})^{1/\alpha})+\P(\rho(X_{N(1)},\xs)\geq C(\log{n})^{1/\alpha})}.
	\end{align*}
	The second term on the right hand side of the above display is bounded by $\lesssim n^{-2}$ by choosing $C>0$ large enough, by using Assumption (R2). In order to bound the first term on the right hand side above, let us define $\delta_{1,n}:=Ck_n\log{n}/n:=\delta_{2,n}$. Next, fix $\epsilon\in (\epsilon_n,C(\log{n})^{1/\alpha})$. Let $N\equiv N(\mu_X,\tX_n,\epsilon,\delta_{1,n})$ and $\B_1,\ldots ,\B_N$ denote the $\mu_X$-covering number and $\mu_X$-cover respectively. Further, let $\mathcal{A}$ denote the sub-collection of $\B_i$'s such that $\mu_X(\B_i)\leq \delta_{2,n}$ for all $\B_i\in\mathcal{A}$. Set $\mathcal{B}:=\cup_{i=1}^N \B_i$. Next, observe the following sequence of inequalities:
	\begin{align}\label{eq:callearly}
	&\;\;\;\;\P(\rho(X_1,X_{N(1)})\geq\epsilon,\max\{\rho(X_1,\xs),\rho(X_{N(1)},\xs)\}\leq C(\log{n})^{1/\alpha})\nonumber \\ &\lesssim \P\left(\rho(X_1,X_{N(1)})\geq \epsilon,X_1,X_{N(1)}\in \tX_n\cap \B\right)+2\delta_{1,n}\nonumber \\ &\overset{(a)}{\lesssim} \P\left(\rho(X_1,X_{N(1)})\geq \epsilon,X_1,X_{N(1)}\in \tX_n\cap \B, X_1,X_{N(1)}\in \cup_{\B_i\notin \mathcal{A}}\B_i\right)+2\delta_{1,n}\nonumber \\ &\;\;+2\delta_{2,n}N(\mu_X,\tX_n,\epsilon,\delta_{1,n})\nonumber \\ &\overset{(b)}{\lesssim} n^{k_n+1}(1-\delta_{2,n})^{n-k_n}+4\delta_{2,n}N(\mu_X,\tX_n,\epsilon,\delta_{1,n})\lesssim n^{-2}+4\delta_{2,n}N(\mu_X,\tX_n,\epsilon,\delta_{1,n}).
	\end{align}
	For (a), we use the fact that $$\P(X_1\in \cup_{\B_i\in\mathcal{A}}\mathcal{B}_i)\leq \sum_{\B_i\in\mathcal{A}} \P(X_i\in\B_i)\leq \delta_{2n}|\mathcal{A}|.$$
	In order to establish (b), note that:
	\begin{align*}
	&\;\;\;\P\left(\rho(X_1,X_{N(1)})\geq \epsilon,X_1,X_{N(1)}\in \tX_n\cap \B, X_1,X_{N(1)}\in \cup_{\B_i\notin \mathcal{A}}\B_i\right)\\ &\leq \P\left(\exists i,j_1,\ldots ,j_{n-k_n-1}\ \mathrm{all}\ \mathrm{distinct}: X_i\in \cup_{\B_i\notin \mathcal{A}}\B_i,\  \min_{1\leq l\leq n-k_n-1} \rho(X_i,X_{j_l})\geq\epsilon\right)\\ &\leq n^{k_n+1}\left(1-\inf_{\B_i\notin \mathcal{A}}\mu_X(\B_i)\right)^{n-k_n}.
	\end{align*}
	The last line of the above display uses standard combinatorial arguments which imply that the number of ways of choosing $n-k_n$ points out of $n$ points is bounded above by $n^{k_n}$, followed by the union bound. The last line of (b) also uses the fact that $C$ is large and $k_n=o(n/\log{n})$. Now by expressing expectations as tail integrals, we get:
	\begin{align*}
	&\E\rho(X_1,X_{N(1)})^{2\beta_2}\mathbf{1}(\max\{\rho(X_1,\xs),\rho(X_{N(1)},\xs)\}\leq C(\log{n})^{1/\alpha})\\ &=2\beta_2\int_0^{2C(\log{n})^{1/\alpha}} \epsilon^{2\beta_2-1} \P(\rho(X_1,X_{N(1)})\geq\epsilon,\max\{\rho(X_1,\xs),\rho(X_{N(1)},\xs)\}\leq C(\log{n})^{1/\alpha})\\ &\lesssim \epsilon_n^{2\beta_2}+\beta_2\delta_{2,n}\int_{\epsilon_n}^{2C(\log{n})^{1/\alpha}} \epsilon^{2\beta_2-1}N(\mu_X,\tX_n,\epsilon,\delta_{1,n})\,d\epsilon.
	\end{align*}
	Combining the above display with~\eqref{eq:biasbound}, we get:
	\begin{align}\label{eq:finbiasbd}
	\big|\E\bar{T}_n-\E\lVert\E[K(\cdot,Y)|X]\rVert_{\hk}^2\big|\lesssim \epsilon_n^{\beta_2}+\sqrt{\nu_{1,n}}
	\end{align}
	where $\nu_{1,n}$ is as defined in~\cref{theo:rateres}. Note that the other terms (apart from $\bar{T}_n$) in $\kmac$ are standard U-statistics which concentrate around their mean at a $\mathcal{O}_{\P}(n^{-1/2})$ rate under the assumptions of~\cref{theo:rateres} (see~\cite[Theorem 12.3]{vaart1996}). Consequently, by combining~\eqref{eq:variancebound},~\eqref{eq:finbiasbd} with~\cref{prop:basicmom} completes the proof.
	
	\subsection{Proof of~\cref{prop:spGauss}}\label{pf:spGauss}
	We will be using~\cref{lem:equivform} repeatedly in this proof.
	
	\noindent {\bf Proof of (a)}: Note that $T_{\alpha}(\cdot)$ is location and scale invariant. Therefore, without loss of generality, we can assume that $\theta_1=\theta_2=0$ and $\sigma_1^2=\sigma_2^2=1$. By plugging in the characteristic functions of Gaussian distributions, we get:
	$$\int_{\R}\frac{\mme|\hat{f}_{Y|X}(t)-\hat{f}_{Y}(t)|^2}{|t|^{1+\alpha}}\,dt=\int_{\R}\frac{\exp(-t^2)[\exp(\rho^2t^2)-1]}{|t|^{1+\alpha}}\,dt.$$
	The integrand is a strictly increasing function of $\rho$, which implies that the left hand side is a strictly increasing function of $\rho$ as well. Finally note that the denominator in~\eqref{eq:mainrep} is free of $\rho$. An application of~\cref{lem:equivform} then completes the proof for $0<\alpha<2$. The same result for $\alpha=2$ will follow from part (c). The fact that $T_{\alpha}(\mu_{\rho})$ satisfies (P1), (P2) and (P3) follows from the observation that $K(\cdot,\cdot)$ as in~\cref{rem:Euclid} is characteristic, followed by an application of~\cref{theo:kacmas}.
	
	\noindent {\bf Proof of (b)}: When $\alpha=1$, we observe that the numerator of $T_1(\mu_{\rho})$ simplifies as:
	\begin{equation}\label{eq:spGauss1}
	\int_{\R}\frac{\mme|\hat{f}_{Y|X}(t)-\hat{f}_{Y}(t)|^2}{t^2}\,dt=\int_{\R} \frac{\exp(-t^2)}{t^2}\left(\sum_{k\geq 1} \frac{(\rho^2t^2)^k}{k!}\right)\,dt =\rho^2 G(\rho),
	\end{equation}
	where $G(\rho)\coloneqq \int_{\R}\exp(-t^2)\left(\sum_{k=1}^{\infty} \frac{(t^2\rho^2)^{k-1}}{k!}\right)\,dt$. Clearly $G(\cdot)$ is a nondecreasing function in $|\rho|$ and as a result,
	$$\int_{\R}\frac{\mme|\hat{f}_{Y|X}(t)-\hat{f}_{Y}(t)|^2}{t^2}\,dt\leq \rho^2 G(1)=\rho^2\int_{\R} \frac{1-\exp(-t^2)}{t^2}\,dt=\rho^2\int_{\R}\frac{1-|\hat{f}_{Y}(t)|^2}{t^2}\,dt.$$
	The right hand side of the above display involves the denominator of $T_1(\mu_{\rho})$. This completes the proof of part (2).
	
	Next, let us define a function $F:\R\to\R$ such that $F(\rho)$ equals the left hand side of~\eqref{eq:spGauss1}. Note that $T_1(\mu_{\rho})=F(\rho)/F(1)$, and $F(0)=F'(0)=0$. Therefore $F(\rho)=\int_0^{\rho}\int_0^x F''(z)\,dz\,dx$. Observe that, by an application of the dominated convergence theorem, we have:
	\begin{align*}
	F''(z)&=\frac{d^2}{dz^2}\int_{\R}\frac{\exp(-t^2)(\exp(z^2t^2)-1)}{t^2}\,dt\\&=\int_{\R} (4z^2t^2+2)\exp(-t^2(1-z^2))\,dt=\frac{2\sqrt{\pi}}{(1-z^2)^{3/2}}.
	\end{align*}
	Using this value of $F''(z)$, we get:
	$$F(\rho)=\int_0^{\rho}\int_0^x \frac{2\sqrt{\pi}}{(1-z^2)^{3/2}}=\int_0^{\rho^2} \frac{\sqrt{\pi}}{\sqrt{1-x}}=2\sqrt{\pi}(1-\sqrt{1-\rho^2}).$$
	Therefore $F(1)=2\sqrt{\pi}$ and $T_1(\mu_{\rho})=1-\sqrt{1-\rho^2}$.
	
	In the previous discussion, we showed that $G(\rho)=T_1(\mu_{\rho})/\rho^2$ is a monotonic function in $|\rho|$. Therefore,
	$$\inf_{|\rho|\neq 0}\frac{\sqrt{T_1(\mu_{\rho})}}{|\rho|}=\lim_{|\rho|\to 0}\frac{\sqrt{1-\sqrt{1-\rho^2}}}{|\rho|}=\lim_{|\rho|\to 0}\frac{1}{|\rho|}\sqrt{\frac{\rho^2}{2}+o(\rho^2)}=2^{-1/2}.$$
	
	\noindent {\bf Proof of (c)}: First, let us obtain a more general expression for $T_2$. Note that:
	\begin{align*}
	T_2=\frac{\E \lVert Y_1-Y_2\rVert_2^2-\E \lVert Y'-\tilde{Y'}\rVert_2^2}{\E \lVert Y_1-Y_2\rVert_2^2}=\frac{\E\langle Y',\tilde{Y'}\rangle -\E \langle Y_1,Y_2\rangle}{\E\lVert Y_1\rVert_2^2-\E \langle Y_1,Y_2\rangle}.
	\end{align*}
	For the multivariate Gaussian case, first we can standardize so that both the marginals have mean $0$ and variance $1$. In this setting, $\E Y_1Y_2=0$, $\E Y_1^2=1$ and $\E Y'\tilde{Y'}=\rho^2$. Plugging these in the above display completes the proof. \qed
	
	\begin{remark}\label{rem:notalpha2}
		Although $T_2$ satisfies $\mathrm{(P1)}$-$\mathrm{(P3)}$ from~\cref{sec:popkmac} in the bivariate normal case, the same is not true in general for distributions on $\R^{d_1+d_2}$. $T_2$ will always satisfy $\mathrm{(P1)}$ and $\mathrm{(P3)}$ (see the proof of~\cref{prop:energyeq} for details), but not necessarily $\mathrm{(P2)}$. To see this, note that for bivariate random variables, $T_2$ can be simplified as follows:
		$$T_2=\frac{\mathrm{Var}(\E[Y|X])}{\mathrm{Var}(Y)}.$$
		So for instance, if $X\sim \mathcal{U}[0,1]$ and $Y|X=x\sim \mathcal{U}([-x,x])$, then the corresponding $T_2$ is always $0$, although $X$ and $Y$ are clearly dependent. 
	\end{remark}

	\subsection{Proof of~\cref{prop:energyeq}}\label{pf:energyeq}
	We will prove a more general result with $T_{\alpha}$  (defined as in~\eqref{eq:propmeas}), for $0<\alpha<2$ instead of $T\equiv T_1$. We therefore redefine the corresponding $T_n$ (from~\eqref{eq:firststat}) as:
	$$T_n:=1-\frac{n^{-1}\sum_{i=1}^n d_i^{-1}\sum_{j:(i,j)\in \emgn} \lVert Y_i-Y_j\rVert_2^{\alpha}}{(n(n-1))^{-1}\sum_{i\neq j}\lVert Y_i-Y_j\rVert_2^{\alpha}}.$$
	Now $T_n$ as defined above is not equal to $\kmac$ with $K(\cdot,\cdot)$ as in~\cref{rem:Euclid}. It can be shown that $T_n-\kmac\overset{\mathbb{P}}{\longrightarrow}0$. However, we will take an alternate route and use a characteristic function representation (as in~\cref{lem:equivform}) to directly prove $T_{\alpha}\to \eta_K(\mu)$. We will work under the assumption that $\E \lVert Y\rVert_2^{2\alpha+\epsilon}<\infty$ for some $\epsilon>0$ (cf. the assumption in~\cref{prop:energyeq} for $\alpha=1$).
	
	In the first part, we show that $T_{\alpha}$ satisfies (P1)-(P3) from~\cref{sec:popkmac}. Note that a regular conditional probability always exists on $\R^d$ for Borel probability measures (see~\cite[Theorem 2.1.15 and Exercise 5.1.16]{Durrett2019}).
	
	Recall the notation from~\cref{lem:equivform}. Note that $|\hat{f}_{Y|X}(t)|^2\leq 1$ for all $t\in\R^{d_2}$ a.s. A direct application of~\cref{lem:equivform} then proves (P1).
	
	Next, if $\mu=\mu_{X}\otimes\mu_{Y}$, then $\mu_{Y|X}=\mu_{Y}$ for $\mu_X$-a.e.~$X$. As a result $\mme\lVert Y'-\tilde{Y'}\rVert_2^{\alpha}=\mme\lVert Y_1-Y_2\rVert_2^{\alpha}$ and consequently, $T_{\alpha}(\mu)=0$. Now suppose that $T_{\alpha}(\mu)=0$. By~\cref{lem:equivform}, we have $\hat{f}_{Y|X}(t)=\hat{f}_{Y}(t)$ for $\mu_X$-a.e.~$X$ and Lebesgue almost every $t\in\R^{d_2}$. This implies $\mu=\mu_{X}\otimes\mu_{Y}$, consequently establishing (P2).
	
	If $Y=g(X)$ $\mu$-a.e., for some measurable function $g:\R^{d_1}\to\R^{d_2}$, then $Y'=\tilde{Y'}=g(X')$ a.s., which implies $T_{\alpha}(\mu)=1$. For the other direction, assume $T_{\alpha}(\mu)=1$. This implies $Y'=\tilde{Y'}$ a.s. Define a function $f:\R^{d_2}\to\R^{d_2}$ such that, for $y=(y_1,y_2,\ldots ,y_{d_2})\in \R^{d_2}$, $f(y)=(\tanh(y_1),\ldots ,\tanh(y_{d_2}))$. Note that $f(\cdot)$ is bounded and invertible. Therefore,
	$$0=\mme\lVert f(Y')-f(\tilde{Y}')\rVert^2=\mme \mathrm{Var}(f(Y')|X').$$
	The above display implies $f(Y')=\mme[f(Y')|X']$ for $\mu_X$-a.e.~$X'$, which further implies $Y'=f^{-1}(\mme[f(Y')|X'])$ a.s. This proves (P3).
	
	In the second part, we want to show that $T_n\overset{\mathbb{P}}{\longrightarrow} T_{\alpha}$. Following our argument in~\cref{theo:consis}, we only need to show the following: $$\E\lVert Y_1-Y_{N(1)}\rVert_2^{\alpha}\overset{n\to\infty}{\longrightarrow}\E \lVert Y'-\tilde{Y'}\rVert_2^{\alpha}.$$
	The proofs of the other parts go through verbatim as in the proof of~\cref{theo:consis}. In order to establish the above, note that, by~\eqref{eq:refbefore}, we get:
	\begin{align*}
	\E\lVert Y_1-Y_{N(1)}\rVert_2^{\alpha}&=\frac{1}{C(d_2,\alpha)}\E\left[\int_{\R^{d_2}}\frac{1-\E(\exp(it^{\top}Y_1)|X_1)\bar{\E(\exp(it^{\top}Y_{N(1)})|X_{N(1)})}}{\lVert t\rVert_2^{d+\alpha}}\,dt\right]\\ &=\frac{1}{C(d_2,\alpha)}\int_{\R^{d_2}}\frac{1-\E g_{X_1}(t)\bar{g_{X_{N(1)}}(t)}}{\lVert t\rVert_2^{d+\alpha}}\,dt,
	\end{align*}
	where $g_x(t):=\E[\exp(it^{\top}Y|X=x)]$ for $x\in\R^{d_1}$, $t\in\R^{d_2}$. By~\cref{lem:techlemmconv}, $g_{X_{N(1)}}(t)\overset{\P}{\longrightarrow} g_{X_1}(t)$ for each $t\in\R^{d_2}$. As $|g_x(t)|\leq 1$ uniformly over $t\in\R^{d_2}$ and $x\in\R^{d_1}$, by the bounded convergence theorem $\E g_{X_{N(1)}}(t)\to \E g_{X_1}(t)$ as $n\to\infty$.
	Using this observation, the following convergence,
	$$\frac{1}{C(d_2,\alpha)}\int_{\R^{d_2}}\frac{1-\E g_{X_1}(t)\bar{g_{X_{N(1)}}(t)}}{\lVert t\rVert_2^{d+\alpha}}\,dt\overset{n\to\infty}{\longrightarrow}\frac{1}{C(d_2,\alpha)}\int_{\R^{d_2}}\frac{1-\E |g_{X_1}(t)|^2}{\lVert t\rVert_2^{d+\alpha}}\,dt$$
	follows by using standard arguments involving characteristic functions as in~\cite[Equations 2.20-2.24]{Gabor2007}~or~\cite[Section F.1, steps 1-3]{Deb19}. Next note that the right side of the display above equals $\E\lVert Y'-\tilde{Y'}\rVert_2^{\alpha}$ by~\cite[Lemma 1]{Gabor2013}, thereby completing the proof.

	\subsection{Proof of~\cref{prop:energyeqfree}}\label{pf:energyeqfree}
	The proof of this result directly follows from the proof of~\cref{theo:consis} and~\cite[Theorem 2.2]{azadkia2019simple}. First note that a regular conditional probability always exists on $\R^d$ for Borel probability measures (see~\cite[Theorem 2.1.15 and Exercise 5.1.16]{Durrett2019}). Also as $K(\cdot,\cdot)$ is bounded, all the moment assumptions from~\cref{theo:consis} are satisfied by default. The only other thing to check is whether assumptions (A1)-(A3) hold for the 1-NNG under no assumptions on $\mu_X$ provided a random tie-breaking mechanism is used (as discussed in the prequel to~\cref{prop:energyeqfree}). All these $3$ properties were established for the 1-NNG in~\cite[Lemmas 10.3, 10.4 and 10.9]{azadkia2019simple}. This gives the result. 
	
	\subsection{Proof of~\cref{thm:Consistency}}\label{pf:Consistency}
	
	First we state an important result about the consistency of the empirical multivariate ranks  (as in~\cref{def:empquanrank}), i.e., $\hat{R}_n(\cdot)$ yields a consistent estimate of $R(\cdot)$ (see~\cite{Deb19} for a proof).
	\begin{prop}\label{prop:consisrank}
		Suppose $\mu\in\mathcal{P}_{ac}(\R^d)$ and $n^{-1}\sum_{i=1}^n \delta_{h_i^d}\overset{w}{\longrightarrow}\nu$. Then $n^{-1}\sum_{i=1}^n \lVert \hat{R}_n(X_i)-R(X_i)\rVert_2\overset{a.s.}{\longrightarrow}0$.
	\end{prop}

	\emph{Proof of (i)}: By~\cite[Proposition 2.2 (ii)]{Deb19}, $(\hat{R}_n^X(X_1),\ldots ,\hat{R}_n^X(X_n))$ (or $(\hat{R}_n^Y(Y_1),\ldots ,\hat{R}_n^Y(Y_n))$) is distributed uniformly over all $n!$ permutations of the set $\mathbf{\mathcal{H}}_n^{d_1}$ (or $\mathbf{\mathcal{H}}_n^{d_2}$). When $\mu=\mu_X\otimes\mu_Y$, clearly $(\hat{R}_n^X(X_1),\ldots ,\hat{R}_n^X(X_n))$ and $(\hat{R}_n^Y(Y_1),\ldots ,\hat{R}_n^Y(Y_n))$ are independent and the joint distribution is pivotal. As $\rkmac$ is a function of $(\hat{R}_n^X(X_1),\ldots ,\hat{R}_n^X(X_n),\hat{R}_n^Y(Y_1),\ldots ,\hat{R}_n^Y(Y_n))$, it also has a pivotal distribution.
	
	\emph{Proof of (ii)}: First let us prove that $\rpk$ satisfies (P1)-(P3) from~\cref{sec:popkmac}. As $R^Y(\cdot)$ is uniformly bounded coordinate-wise by $1$ and $K(\cdot,\cdot)$ is characteristic, by~\cref{theo:consis}, $\rpk\in [0,1]$ which yields (P1). To prove (P2), suppose that $X$ and $Y$ are independent, then $R^Y(Y)$ and $X$ are independent and so $\rpk=0$ (by~\cref{theo:consis}). And if $\rpk=0$, then $R^Y(Y)$ and $R^X(X)$ are independent (by~\cref{theo:consis}) and so are $Y$ and $X$ (see the proof of~\cite[Lemma 3.1(b)]{Deb19}). For (P3), if $Y$ is a measurable function of $X$, then $R^Y(Y)$ is a measurable function of $X$ and so $\rpk=1$ (by~\cref{theo:consis}). Finally, if $\rpk=1$, then $R^Y(Y)$ is a measurable function of $X$ (by~\cref{theo:consis}), i.e., there exists a measurable function $h:\R^{d_1}\to\R^{d_2}$ such that $R^Y(Y)=h(X)$. Let $Q^Y(\cdot)$ be the quantile map from~\cref{def:popquanrank}. By~\cref{prop:Mccan}, we get $Y=Q^Y(h(X))$, $\mu$-a.e., and thus $Y$ is a measurable function of $X$. This completes the proof.
	
	Next we will show that $\rkmac\overset{\P}{\longrightarrow}\rpk$. As $n^{-1}\sum_{i=1}^n \delta_{h_i^{d_2}}$ converges weakly to $\mathcal{U}[0,1]^{d_2}$, the following conclusions are easy consequences of the Portmanteau Theorem:
	\begin{align}\label{eq:Promitexp}
	&\frac{1}{n}\sum_{i=1}^n K(\hat{R}_n^Y(Y_i),\hat{R}_n^Y(Y_i))\overset{n\to\infty}{\longrightarrow}\E K(R^Y(Y_1),R^Y(Y_1)),\nonumber  \\&
	\frac{1}{n(n-1)}\sum_{i\neq j, 1}^n K(\hat{R}_n^Y(Y_i),\hat{R}_n^Y(Y_j))\overset{n\to\infty}{\longrightarrow}\E K(R^Y(Y_1),R^Y(Y_2)).
	\end{align}
	In the above displays, both terms on the right hand side are deterministic because both of them are permutation-invariant and $(\hat{R}_n(Y_1),\ldots ,\hat{R}_n(Y_n))$ is some permutation of the fixed set $\mathcal{H}_n^{d_2}$. Consequently the above convergence is a deterministic result.
	By an application of~\cref{prop:consisrank} and Assumption (A3), it further suffices to show that:
	$$\mathcal{Z}_n:=\frac{1}{n}\sum_{i}d^{-1}_i\sum_{j:(i,j)\in \rmgn}K(R^Y(Y_i), R^Y(Y_j))\overset{\mathbb{P}}{\longrightarrow}\E K(R^Y(Y'),R^Y(\tilde{Y'})).$$
	
	In order to establish the above, by an application of Chebyshev's inequality, it suffices to show that:
	\begin{enumerate}
		\item[(a)] $\E \mathcal{Z}_n^2\overset{n\to\infty}{\longrightarrow} \left[\E K(R^Y(Y'),R^Y(\tilde{Y'}))\right]^2.$
		\item[(b)] $\E \mathcal{Z}_n\overset{n\to\infty}{\longrightarrow} \E K(R^Y(Y'),R^Y(\tilde{Y'})).$
	\end{enumerate}
	
	We first prove (a). For the sake of simplicity, we write 
	\begin{align*}
	\mathbb{E}\mathcal{Z}_n^2 & = (\mathbf{I})+ (\mathbf{II}) + (\mathbf{III})
	\end{align*}
	where 
	\begin{align*}
	(\mathbf{I}) &:= \frac{1}{n^2}\E\left[\sum_{i}\sum_{j:(i,j)\in \rmgn}\frac{1}{d_i}\left(\frac{1}{d_i}+\frac{1}{d_j}\right)\mathbb{E}\Big[\big(K(R^Y(Y_i), R^Y(Y_j))\big)^2\Big|R^X(X_i),R^X(X_j)]\right] \\
	(\mathbf{II}) &:= \frac{1}{n^2}\E\Bigg[\sum_{i}\sum_{j\neq k: (i,j),(i,k)\in \rmgn}\left(\frac{1}{d_i}+\frac{1}{d_j}\right)\left(\frac{1}{d_i}+\frac{1}{d_k}\right)\\& \times\mathbb{E}\Big[K(R^Y(Y_i), R^Y(Y_j))K(R^Y(Y_i), R^Y(Y_k))|R^X(X_i),R^X(X_j),R^X(X_k)\Big]\Bigg]\\
	(\mathbf{III}) &:=\frac{1}{n^2} \E\Bigg[\sum_{i\neq j} \sum_{k\neq \ell: (k,i), (\ell,j)\in\rmgn}\frac{1}{d_id_j} \mathbb{E}\Big[K(R^Y(Y_i),R^Y(Y_k))K(R^Y(Y_j), R^Y(Y_\ell))|\\ &R^X(X_i),R^X(X_j),R^X(X_k),R^X(X_l)\Big]\Bigg] .  
	\end{align*}
	We now claim and prove that 
	\begin{align}\label{eq:ppconv}
	(\mathbf{III}) \to \Big(\mathbb{E}\Big(\mathbb{E}\Big[K\big(R^Y(\tilde{Y}_1),R^Y(\tilde{Y}_2)\big)|X\Big]\Big)\Big)^2
	\end{align}
	where $\tilde{Y}_1$ and $\tilde{Y}_2$ are independent samples from the conditional distribution $Y|X$. 
	as $n$ goes to $\infty$. For this, we define 
	\begin{align*}
	(\widetilde{\mathbf{III}}) := \frac{1}{n^2}\sum_{i\neq j} \sum_{k\neq \ell: (k,i),(\ell,j)\in\rmgn} &\frac{1}{d_i}\frac{1}{d_j}\mathbb{E}\Big[\mathbb{E}\Big[K\big(R^Y(Y_i),R^Y(Y^{\prime}_i)\big)|R^X(X_i)\Big]\\&\times \mathbb{E}\Big[K\big(R^Y(Y_j),R^Y(Y^{\prime}_j)\big)|R^X(X_j)\Big]\Big]
	\end{align*}
	where $Y^{\prime}_i$ (resp. $Y^{\prime}_j$) is a sample from the conditional distribution $Y|X_i$ (resp. $Y|X_j$) independent of $Y_i$ and $Y_j$.  
	A simple application of the fact that $K(\cdot,\cdot)$ is continuous a.e., $R^Y(\cdot)$ is bounded coordinate-wise, shows that:
	\begin{align}
	|(\mathbf{III}) -(\widetilde{\mathbf{III}})|&\lesssim C\frac{t_n}{nr_n}\sum_{i}\frac{1}{d_i}\mathbb{E}\Big[\sum_{k:(k,i)\in \rmgn}\big|  \mathbb{E}\big[K(R^Y(Y_i),R^Y(Y_k))\big|R^X(X_i),R^X(X_k)\big]\nonumber \\ &- \mathbb{E}\big[K(R^Y(Y_i),R^Y(Y^{\prime}_i))\big|R^X(X_i)\big]\big|\Big].  \label{eq:Error}
	\end{align}
	We show that the right hand side of the above inequality converges to $0$ as $n \to \infty$. By our assumption, $\mathbb{E}\big[K\big(R^Y(Y_1), R^Y(Y_2)\big)|R^X(X_1)=x_1,R^X(X_2)=x_2\big]$ is uniformly $\beta$-H{\"o}lder continuous w.r.t. $x_1,x_2$. Applying this to \eqref{eq:Error} shows 
	\begin{align}
	\limsup_{n\to \infty}\text{r.h.s. of \eqref{eq:Error}} &\lesssim \limsup_{n\to \infty}\frac{t_n}{nr_n}\sum_{i=1}^{n}\frac{1}{d_i}\mathbb{E}\Big[\sum_{k:(k,i)\in \rmgn} \big\|R^X(X_i)- R^X(X_k)\big\|^{\beta} \Big]\nonumber \\
	&\leq \limsup_{n\to \infty} \frac{t_n}{nr_n} \mathbb{E}\Big[\sum_{i=1}\frac{1}{d_i}\sum_{k:(k,i)\in\rmgn}\big\|\hat{R}_n^X(X_i)-\hat{R}_n^X(X_k)\big\|^{\beta}\Big] . \label{eq:errorbd}
	\end{align} 
	where the last inequality follows once again from~\cref{prop:consisrank}. Notice that $\sum_{i=1}\sum_{k:(k,i)\in\rmgn}\big\|\hat{R}_n^X(X_i)-\hat{R}_n^X(X_k)\big\|^{\beta}=\sum_{e\in\rmgn} |e|^{\beta}$ where $|e|$ denotes the length of the edge $e$. The right hand side of the above display converge to $0$ by our assumption.   
	This shows $|(\mathbf{III})-(\widetilde{\mathbf{III}})|$ converges to $0$ as $n$ tends to $\infty$. On other hand, we define 
	\begin{align*}
	(\widetilde{\mathbf{I}}) &:= \frac{1}{n^2} \E\Bigg[\sum_{i=1}^{n} \sum_{j:(j,i)\in \rmgn} \Bigg(\frac{1}{d^2_i}\Big[\big(K(R^Y(Y_i), R^Y(Y^{\prime}_i))\big)^2\Big]\\& \qquad \qquad + \frac{1}{d_id_j}\mathbb{E}\Big[K(R^Y(Y_i), R^Y(Y^{\prime}_i))K(R^Y(Y_j), R^Y(Y^{\prime}_j))\Big]\Bigg)\Bigg]\\
	(\widetilde{\mathbf{II}}) &:=\frac{1}{n^2}\E\Bigg[\sum_{i=1}^{n} \sum_{(j,i),(k,i)\in\rmgn}\Bigg[ \frac{1}{d_jd_k}K(R^Y(Y_k), R^Y(Y^{\prime}_k))K(R^Y(Y_j), R^Y(Y^{\prime}_j))\\ & \quad +\frac{1}{d_i^2}K^2(R^Y(Y_i),R^Y(Y_i'))+\frac{2}{d_id_j}K(R^Y(Y_i),R^Y(Y_i'))K(R^Y(Y_j),R^Y(Y_j'))\Bigg]\Bigg].
	\end{align*}
	Note that 
	\begin{align}
	(\widetilde{\mathbf{I}}) +(\widetilde{\mathbf{II}}) + (\widetilde{\mathbf{III}}) &= \frac{1}{n^2} \mathbb{E}\Big[\Big(\sum_{i=1}^n K(R^Y(Y_i), R^Y(Y^{\prime}_i))\Big)^2\Big]\nonumber \\& \to \Big(\mathbb{E}\big[K(R^Y(Y_1), R^Y(Y^{\prime}_1))\big]\Big)^2 \label{eq:sumconv}
	\end{align}
	where the last convergence follows from the strong law of large numbers and the dominated convergence theorem. Owing to the fact that Assumption (A3) holds and $K(\cdot,\cdot)$ is continuous a.e., it is easy to check that $(\mathbf{I}),(\mathbf{II}),(\widetilde{\mathbf{I}}),(\widetilde{\mathbf{II}})$ are converging to $0$ as $n \to \infty$. Combining this with \eqref{eq:ppconv} and \eqref{eq:sumconv} completes the proof of (a). 
	
	Now we move on to (b). By the towering property of the conditional expectation, we have 
	\begin{align}
	\E \mathcal{Z}_n&=\mathbb{E}\Big[\frac{1}{n}\sum_{i}d^{-1}_i\sum_{j:(i,j)\in \mathcal{E}(\mathcal{G}_n)}K(R^Y(Y_i), R^Y(Y_j))\Big]\nonumber \\
	& = \mathbb{E}\Big[\frac{1}{n}\sum_{i}d^{-1}_i\sum_{j:(i,j)\in \mathcal{E}(\mathcal{G}_n)}\mathbb{E}\big[K(R^Y(Y_i), R^Y(Y_j))|R^X(X_i),R^X(X_j)\big]\Big].\label{eq:Tower}
	\end{align}
	By the $\beta$-H{\"o}lder continuity of $\mathbb{E}\big[K(R^Y(Y_i), R^Y(Y_j))\big|R^X(X_i)=x, R^X(X_j)=y\big]$ as a function of $x$ and $y$, there exists $C>0$ such that 
	\begin{align*}
	\Big|&\mathbb{E}\big[K(R^Y(Y_i), R^Y(Y_j))\big|R^X(X_i), R^X(X_j)\big]\\& -\mathbb{E}\big[K(R^Y(Y_i), R^Y(Y^{\prime}_j))\big|R^X(X_i)\big]\Big|\lesssim \|R^X(X_i)- R^X(X_j)\|^{\beta}
	\end{align*}
	where $Y_i,Y^{\prime}_i$ are two independent samples from the conditional distribution of $Y$ given $X_i$. Applying the above inequality, we get 
	\begin{align*}
	\Big|&\text{r.h.s. of \eqref{eq:Tower}}- \frac{1}{n}\mathbb{E}\big[\sum_{i=1}^{n}K(R^Y(Y_i), R^Y(Y^{\prime}_i))\big]\Big|\\& \lesssim \mathbb{E}\Big[\frac{1}{n}\sum_{i=1}^{n}\sum_{j:(i,j)\in \mathcal{E}(\mathcal{G}_n)}\frac{1}{d_i}\|R^X(X_i)- R^X(X_j)\|^{\beta}\Big].
	\end{align*} 
	The proof of (b) can now be completed using the same steps as those in the proof of (a) starting from~\eqref{eq:errorbd}.    	
	
	\subsection{Proof of~\cref{prop:Chacon}}\label{pf:Chacon}
	
	The crucial observation in this proof is that $K(y_1,y_2)=|y_1|+|y_2|-|y_1-y_2|=\min\{y_1,y_2\}$ for $y_1,y_2\in [0,\infty)$. Now, when $d_2=1$, $R^Y(\cdot)$ is simply the cumulative distribution function of $Y$; we will call it $F_Y$ to stick with conventional notation. Next note that, $\rpk$ can be simplified as:
	\begin{equation}\label{eq:temprpk}
	\rpk=\frac{\E\min\{F_Y(Y'),F_Y(\tilde{Y'})\}-\E\min\{F_Y(Y_1),F_Y(Y_2)\}}{1/2-\E\min\{F_Y(Y_1),F_Y(Y_2)\}}.
	\end{equation}
	As $\min\{a,b\}=\int_0^{\infty} \mathbf{1}(t\leq a)\mathbf{1}(t\leq b)\,dt$ for $a,b\in\R$, an application of the dominated convergence theorem yields: $$\E\min\{F_Y(Y'),F_Y(\tilde{Y'})\}=\int \E [\left(\P(F_Y(Y)\geq t|X)\right)^2]\,dt=\int \E[\left(\P(Y\geq t|X)\right)^2]\,d\mu_Y(t).$$ and similarly, $$\E\min\{F_Y(Y_1),F_Y(Y_2)\}=\int \left(\P(Y\geq t)\right)^2\,d\mu_Y(t).$$ Plugging the above expressions into the expression of $\rpk$ in~\eqref{eq:temprpk} gives: $$\rpk =\frac{\int\left(\E[\left(\P(Y\geq t|X)\right)^2]-\left(\P(Y\geq t)\right)^2\right)\,d\mu_Y(t)}{\int\left(\P(Y\geq t)-\left(\P(Y\geq t)\right)^2\right)\,d\mu_Y(t)}=\xi(\mu).$$ This completes the proof.
	\qed

	\subsection{Proof of~\cref{theo:ranknullclt}}\label{pf:ranknullclt}
	Define $C_n:=(n(n-1))^{-1}\sum_{i\neq j} K(\hat{R}_n^Y(Y_i),\hat{R}_n^Y(Y_j))$ and recall that $C_n$ is a deterministic quantity as was explained in the comment after~\eqref{eq:Promitexp}. Also $\E K(\hat{R}_n^Y(Y_1),\hat{R}_n^Y(Y_2))=C_n$. Let $\mathcal{F}_n:=\sigma(X_1,\ldots ,X_n)$, i.e., the $\sigma$-field generated by $(X_1,\ldots ,X_n)$. Consequently, note that:
	\begin{align}\label{eq:rcenter}
	\E[\nrk|\mathcal{F}_n]&=\sqrt{n}\left(\frac{1}{n}\sum_{i=1}^n \frac{1}{\tilde{d}_i}\sum_{j:(j,i)\in\rmgn}\E [K(\hat{R}_n^Y(Y_i),\hat{R}_n^Y(Y_j))]-C_n\right)\nonumber \\&=\sqrt{n}C_n\left(\frac{1}{n}\sum_{i=1}^n \frac{1}{\tilde{d}_i}\sum_{j:(j,i)\in\rmgn} 1-1\right)=0.
	\end{align}
	By the same calculation as in~\eqref{eq:mainvareq}, we get:
	\begin{align}\label{eq:rankmainvareq}
	\mbox{Var}(\nrk|\mathcal{F}_n)=\left(\tilde{g}_1+\tilde{g}_3-\frac{2}{n-1}\right)(\hat{a}-2\hat{b}+\hat{c})+(\tilde{g}_2-1)(\hat{b}-\hat{c}),
	\end{align}
	where 
	\begin{align*}
	&\hat{a}:=\frac{1}{n(n-1)}\sum_{(i,j)\ \mathrm{distinct}} K^2(\hat{R}_n^Y(Y_i),\hat{R}_n^Y(Y_j))\\ &\hat{b}:=\frac{1}{n(n-1)(n-2)}\sum_{(i,j,l)\ \mathrm{distinct}}K(\hat{R}_n^Y(Y_i),\hat{R}_n^Y(Y_j))K(\hat{R}_n^Y(Y_i),\hat{R}_n^Y(Y_l))\\ &\hat{c}:=\frac{1}{n(n-1)(n-2)(n-3)}\sum_{(i,j,l,m)\ \mathrm{distinct}} K(\hat{R}_n^Y(Y_i),\hat{R}_n^Y(Y_j))K(\hat{R}_n^Y(Y_l),\hat{R}_n^Y(Y_m)).
	\end{align*}
	Now $\hat{a},\hat{b},\hat{c}$ are clearly $\mathcal{O}(1)$ and $\tilde{g}_1,\tilde{g}_2,\tilde{g}_3$ are $\mathcal{O}(1)$ by using assumption (A3) (same as in~\eqref{eq:boundcount}). Next observe that the right hand side of~\eqref{eq:rankmainvareq} is a deterministic quantity. Therefore by combining~\eqref{eq:rcenter}~and~\eqref{eq:rankmainvareq}, we get $\mbox{Var}(\nrk)=\mathcal{O}(1)$.
	
	In order to establish the CLT, define
	\begin{equation}\label{eq:hajekpop}
	\nrp:=\sqrt{n}\left(\frac{1}{n}\sum_{i=1}^n \frac{1}{\tilde{d}_i}\sum_{j:(j,i)\in\rmgn} K(R^Y(Y_i),R^Y(Y_j))-\frac{1}{n(n-1)}\sum_{i\neq j} K(R^Y(Y_i),R^Y(Y_j))\right).
	\end{equation}
	By~\cref{theo:nullclt}~and~\cref{lem:hajekrep}, it suffices to show that 
	$$\sup_{\tilde{\mathcal{G}}\in\mathcal{J}_{\theta}}\Bigg|\frac{\mbox{Var}(\nrp)}{\mbox{Var}(\nrk)}-1\Bigg|\overset{n\to\infty}{\longrightarrow}0.$$
	As in~\eqref{eq:mainvareq}, we have:
	\begin{equation}\label{eq:popvar}
	\mbox{Var}(\nrp)=\E(\nrp)^2=\left(\tilde{g}_1+\tilde{g}_3-\frac{2}{n-1}\right)(\tilde{a}-2\tilde{b}+\tilde{c})+(\tilde{g}_2-1)(\tilde{b}-\tilde{c}).
	\end{equation}
	Therefore, 
	\begin{align*}
	&\sup_{\tilde{\mathcal{G}}\in\mathcal{J}_{\theta}}\Bigg|\frac{\mbox{Var}(\nrp)}{\mbox{Var}(\nrk)}-1\Bigg|\lesssim \sup_{\tilde{\mathcal{G}}\in\mathcal{J}_{\theta}}\Bigg[\left(\tilde{g}_1+\tilde{g}_3-\frac{2}{n-1}\right)(\tilde{a}-\hat{a}-2\tilde{b}+2\hat{b}+\tilde{c}-\hat{c})\\ &\qquad \qquad +(\tilde{g}_2-1)(\tilde{b}-\hat{b}-\tilde{c}+\hat{c})\Bigg]\overset{n\to\infty}{\longrightarrow}0.
	\end{align*}
	This completes the proof. \qed
	
	\section{Some Technical Lemmas}\label{sec:techlem}
	\begin{lemma}[H\'ajek representation]\label{lem:hajekrep}
		Suppose $\mu=\mu_X\otimes\mu_Y$, $\mu_X\in\mathcal{P}_{ac}(\R^{d_1})$ and $\mu_Y\in\mathcal{P}_{ac}(\R^{d_2})$. Recall the definition of $\nrp$ from~\eqref{eq:hajekpop} and that of $\nrk$ from~\eqref{eq:recallnrk}. Consider the following subclass of graph functionals given by:
		\begin{align*}
		&\mathcal{I}_{\theta} :=\{\tilde{\mathcal{G}}: r_n^{-1}t_n\leq \theta\ \mbox{and } \;   \trgn:=\tilde{\mathcal{G}}(\hat{R}_n^X(X_1),\ldots ,\hat{R}_n^X(X_n))\},
		\end{align*}
		for $\theta\in (0,\infty)$. Then the following convergence holds, for fixed $\theta$:
		$$\sup_{\tilde{\mathcal{G}}\in\mathcal{I}_{\theta}}\E[\left(\nrk-\nrp\right)^2]\overset{n\to\infty}{\longrightarrow}0.$$
	\end{lemma}
	\begin{proof}
		To prove the above, let us first introduce the notion of the \emph{resampling distribution}. Note that when $\tilde{\mu}=\tilde{\mu}_X\otimes\tilde{\mu}_Y$, the joint distribution of $(X_1,Y_1),\ldots ,(X_n,Y_n)$ is the same as the joint distribution of $(X_1,Y_{\sigma(1)}),\ldots ,(X_n,Y_{\sigma(n)})$ where $\sigma$ is a random permutation of the set $\{1,2,\ldots,n\}$ which is drawn independent of the $X_i$'s and $Y_i$'s. 
		
		The expressions for $\E(\nrk)^2$ and $\E(\nrp)^2$ has already been presented in~\eqref{eq:rankmainvareq}~and~\eqref{eq:popvar} respectively. Therefore, in the sequel, we will only focus on the term $\E[\nrk\nrp]$ where we will use the \emph{resampling distribution} as discussed above. Towards this direction, let us define:
		\begin{align*}
		&\underline{a}:=\frac{1}{n(n-1)}\sum_{(i,j)\ \mathrm{distinct}} K(\hat{R}_n^Y(Y_i),\hat{R}_n^Y(Y_j))K(R^Y(Y_i),R^Y(Y_j))\\ &\underline{b}:=\frac{1}{n(n-1)(n-2)}\sum_{(i,j,l)\ \mathrm{distinct}}K(\hat{R}_n^Y(Y_i),\hat{R}_n^Y(Y_j))K(R^Y(Y_i),R^Y(Y_l))\\ &\underline{c}:=\frac{1}{n(n-1)(n-2)(n-3)}\sum_{(i,j,l,m)\ \mathrm{distinct}} K(\hat{R}_n^Y(Y_i),\hat{R}_n^Y(Y_j))K(R^Y(Y_l),R^Y(Y_m)).
		\end{align*}
		Also to simplify notation, we will use the symbol $j\sim i$ for $(j,i)\in\rmgn$ and let $\hat{Z}_i:=\hat{R}_n^Y(Y_i)$ and $Z_i:=R^Y(Y_i)$. With the above notation, observe that $\E[\nrk\nrp]$ can be simplified as follows: 
		\begin{align*}
		&\E\Bigg[\left(\frac{1}{n}\sum_{i=1}^n \frac{1}{\tilde{d}_i}\sum_{j\sim i} K(\hat{R}_n^Y(Y_{\sigma(i)}),\hat{R}_n^Y(Y_{\sigma(j)}))\right)\left(\frac{1}{n}\sum_{i=1}^n \frac{1}{\tilde{d}_i}\sum_{j\sim i} K(R^Y(Y_{\sigma(i)}),R^Y(Y_{\sigma(j)}))\right)\\ &-\left(\frac{1}{n}\sum_{i=1}^n \frac{1}{\tilde{d}_i}\sum_{j\sim i} K(\hat{R}_n^Y(Y_{\sigma(i)}),\hat{R}_n^Y(Y_{\sigma(j)}))\right)\left(\frac{1}{n(n-1)}\sum_{i\neq j} K(R^Y(Y_i),R^Y(Y_j))\right)\\ &-\left(\frac{1}{n}\sum_{i=1}^n \frac{1}{\tilde{d}_i}\sum_{j\sim i} K(R^Y(Y_{\sigma(i)}),R^Y(Y_{\sigma(j)}))\right)\left(\frac{1}{n(n-1)}\sum_{i\neq j} K(\hat{R}_n^Y(Y_i),\hat{R}_n^Y(Y_j))\right)\\ &+\left(\frac{1}{n(n-1)}\sum_{i\neq j}K(\hat{R}_n^Y(Y_{\sigma(i)}),\hat{R}_n^Y(Y_{\sigma(j)}))\right)\left(\frac{1}{n(n-1)}\sum_{i\neq j} K(R^Y(Y_{\sigma(i)}),R^Y(Y_{\sigma(j)}))\right)\Bigg]\\ &\overset{(*)}{=}(\tilde{g}_1+\tilde{g}_3)\E\left[K(\hat{Z}_{\sigma(1)},\hat{Z}_{\sigma(2)})K(Z_{\sigma(1)},Z_{\sigma(2)})\right]+(3+\tilde{g}_2-2\tilde{g}_1-2\tilde{g}_3)\E\Big[K(\hat{Z}_{\sigma(1)},\hat{Z}_{\sigma(2)})\\ &K(Z_{\sigma(1)},Z_{\sigma(3)})\Big]+(n-3+\tilde{g}_1+\tilde{g}_2+\tilde{g}_3)\E\left[K(\hat{Z}_{\sigma(1)},\hat{Z}_{\sigma(2)})K(Z_{\sigma(3)},Z_{\sigma(4)})\right]-\frac{2}{n-1}\\ &\E\left[K(\hat{Z}_{\sigma(1)},\hat{Z}_{\sigma(2)})K(Z_{\sigma(1)},Z_{\sigma(2)})\right]-\frac{4(n-2)}{n-1}\cdot\E\left[K(\hat{Z}_{\sigma(1)},\hat{Z}_{\sigma(2)})K(Z_{\sigma(1)},Z_{\sigma(3)})\right]\\ &-\frac{(n-2)(n-3)}{n-1}\cdot\E\left[K(\hat{Z}_{\sigma(1)},\hat{Z}_{\sigma(2)})K(Z_{\sigma(3)},Z_{\sigma(4)})\right]\\ &=\left(\tilde{g}_1+\tilde{g}_3-\frac{2}{n-1}\right)(\underline{a}-2\underline{b}+\underline{c})+(\tilde{g}_2-1)(\underline{b}-\underline{c}).
		\end{align*}
		Here $(*)$ follows from~\eqref{eq:maint1},~\eqref{eq:maint2}~and~~\eqref{eq:maint3}.
		By~\cref{prop:Mccan}, $\hat{a},\underline{a}$ converge to $\tilde{a}$ in $L^1$; same holds for $\hat{b},\underline{b},\tilde{b}$ and $\hat{c},\underline{c},\tilde{c}$. Also $\hat{a},\underline{a},\tilde{a},\hat{b},\underline{b},\tilde{b},\hat{c},\underline{c},\tilde{c}$ do not depend on the graph functional $\mathcal{G}$. Therefore, the following holds:
		\begin{align*}
		&\sup_{\tilde{\mathcal{G}}\in\mathcal{I}_{\theta}}\E\left(\nrk-\nrp\right)^2=\sup_{\tilde{\mathcal{G}}\in\mathcal{I}_{\theta}}\Bigg[\left(\tilde{g}_1+\tilde{g}_3-\frac{2}{n-1}\right)\{\hat{a}-2\underline{a}+\tilde{a}+4\underline{b}-2\tilde{b}-2\hat{b}\\ &\qquad \qquad  +\hat{c}-2\underline{c}+\tilde{c} \} +(\tilde{g}_2-1)(\hat{b}-2\underline{b}+\tilde{b}-\hat{c}+2\underline{c}-\tilde{c})\Bigg]\overset{n\to\infty}{\longrightarrow}0.
		\end{align*}
		This completes the proof.
	\end{proof}
	\begin{lemma}\label{lem:techlemnbd}
		Let $f:\X\to [0,\infty)$ be any arbitrary measurable function. Recall the construction of the random variables $N(i)$, $1\leq i\leq n$ from assumption (A1). If $\mathcal{G}_n$ satisfies (A3), then there exists a constant $C>0$ (depending on $f$) such that the following holds:
		$$\E f(X_{N(1)})\leq C\E f(X_1).$$
	\end{lemma}
	\begin{proof}
		The key argument in this proof is the symmetry in the construction of $\mathcal{G}_n$ and the exchangeability of the $X_i$'s. Recall the definition of $t_n$ from assumption (A3). Observe that:
		\begin{align*}
		\E f(X_{N(1)})&=\frac{1}{n}\sum_{i=1}^n \E f(X_{N(i)})\\ &=\frac{1}{n}\sum_{i=1}^n \E\left[\frac{1}{d_i}\sum_{j=1}^n f(X_j)\mathbbm{1}((i,j)\in\emgn)\right]\\ &\leq \frac{1}{nr_n}\E\left[\sum_{j=1}^n f(X_j)\sum_{i=1}^n \mathbbm{1}((i,j)\in\emgn)\right]\\ &\leq \frac{t_n}{r_n}\cdot\frac{1}{n}\E\left[\sum_{i=1}^n f(X_j)\right]\leq C\E f(X_1)
		\end{align*}
		where the last line follows from Assumption (A3). This completes the proof.
	\end{proof}
	
	\begin{lemma}\label{lem:techlemmconv}
		Let $f:\X\to\mathcal{H}$ be a measurable function and assume that $\mathcal{H}$ is a second countable Hilbert space. Then provided $\mathcal{G}_n$ satisfies (A1) and (A3), we have $\lVert f(X_1)-f(X_{N(1)})\rVert_{\mathcal{H}}\overset{\P}{\longrightarrow}0$. 
	\end{lemma}
	\begin{proof}
		Fix an arbitrary $\epsilon>0$. By~\cref{prop:lusin} (see~\cite{lusin1912proprietes}), there exists a compact set $K$ such that $f$ restricted to $K$ is a continuous function and $\mu_X(K^c)<\epsilon$. Further, note that for any $\delta>0$, we have:
		\begin{align*}
		&\;\;\;\P\left(\lVert f(X_1)-f(X_{N(1)})\rVert_{\mathcal{H}}>\delta\right)\\ &\leq \P\left(\lVert f(X_1)-f(X_{N(1)})\rVert_{\mathcal{H}}>\delta, X_1\in K, X_{N(1)}\in K\right)+\P(X_1\in K^c)+\P(X_{N(1)}\in K^c).
		\end{align*}
		The first term on the right hand side of the above display converges to $0$ as $n\to\infty$ by combining the continuous mapping theorem with Assumption (A1). Moreover, by construction, $\P(X_1\in K^c)\leq \epsilon$. Also,~\cref{lem:techlemnbd} implies that $\P(X_{N(1)}\in K^c)\leq C\P(X_1\in K^c)\leq C\epsilon$. Therefore,
		$$\limsup \P\left(\lVert f(X_1)-f(X_{N(1)})\rVert_{\mathcal{H}}>\delta\right)\leq (C+1)\epsilon.$$
		This completes the proof as $\epsilon>0$ is arbitrary.
	\end{proof}
	
	\begin{lemma}\label{lem:projerrorctrl}
		Recall the notation from~\cref{theo:nullclt}. Assume that $\E K^2(Y_1,Y_2)<\infty$. Set $\theta:=\E K(Y_1,Y_2)$ and $h(Y_i):=2\E[K(Y,Y_i)|Y_i]$ for $i=1,2,\ldots ,n$. Then there exists a universal constant $D>0$ such that:
		\begin{equation}\label{eq:simplify}
		n\E\left[\frac{1}{\sqrt{n}(n-1)}\sum_{i\neq j} K(Y_i,Y_j)-\frac{1}{\sqrt{n}}\sum_{i=1}^n (h(Y_i)-\theta)\right]^2\leq D\left(\E K^2(Y_1,Y_2)+\theta^2\right).
		\end{equation}
	\end{lemma}
	\begin{proof}
		Note that the left hand side of the above display is equivalent to
		$$n\E\left[\frac{1}{\sqrt{n}(n-1)}\sum_{i\neq j} (K(Y_i,Y_j)-\theta)-\frac{2}{\sqrt{n}}\sum_{i=1}^n (\E[K(Y,Y_i)|Y_i]-\theta)\right]^2.$$
		In this proof, we will use $\lesssim$ to hide universal constants. Observe that:
		\begin{equation}\label{eq:chterm1}
		\E\left[\frac{2}{\sqrt{n}}\sum_{i=1}^n (\E[K(Y,Y_i)|Y_i]-\theta)\right]^2=4\E\left(\E[K(Y,Y_1)|Y_1]-\theta\right)^2.
		\end{equation}
		\begin{equation}\label{eq:chterm2}
		\E\left[\frac{1}{\sqrt{n}(n-1)}\sum_{i\neq j} (K(Y_i,Y_j)-\theta)\right]^2=\frac{1}{n-1}\E( K(Y_1,Y_2)-\theta)^2+\frac{4(n-2)}{n-1}\E(\E[K(Y,Y_1)|Y_1]-\theta)^2.
		\end{equation}
		\begin{equation}\label{eq:chterm3}
		\frac{4}{n(n-1)}\E\left[\left(\sum_{i\neq j} (K(Y_i,Y_j)-\theta)\right)\left(\sum_{l=1}^n (\E[K(Y,Y_l)|Y_l]-\theta)\right)\right]=-8\E(\E[K(Y,Y_1)|Y_1]-\theta)^2.
		\end{equation}
		Using~\eqref{eq:chterm1},~\eqref{eq:chterm2},~\eqref{eq:chterm3}~and~\eqref{eq:simplify}, completes the proof.
	\end{proof}
	\begin{lemma}\label{lem:equivform}
		Suppose $(X,Y)\sim\mu$. Further, given any $t\in\R^{d_2}$, set $\hat{f}_{Y|X}(t)\coloneqq \mme[\exp(it^{\top}Y)|X]$ and $\hat{f}_{Y}(t)\coloneqq \mme[\exp(it^{\top}Y)]$. Then we have:
		\begin{equation}\label{eq:mainrep}
		T_{\alpha}(\mu)= \left(\int_{\R^{d_2}}\frac{\mme|\hat{f}_{Y|X}(t)-\hat{f}_{Y}(t)|^2}{\lVert t\rVert^{d_2+\alpha}_2}\,dt\right) \Big/ \left(\int_{\R^{d_2}}\frac{1-|\hat{f}_{Y}(t)|^2}{\lVert t\rVert^{d_2+\alpha}_2}\,dt\right).
		\end{equation}
	\end{lemma}
	\begin{proof}
		The proof of this lemma is based on~\cite[Lemma 1]{Gabor2013} which states that there exists a constant $C(d,\alpha)$ such that for all $x\in\R^d$, the following identity holds:
		\begin{equation}\label{eq:refbefore}
		\int_{\R^d} \frac{1-\cos\langle t,x\rangle}{\lVert t\rVert^{d+\alpha}_2}\,dt=C(d,\alpha)\lVert x \rVert^{\alpha}_2
		\end{equation}
		if $0<\alpha<2$. The above display combined with the dominated convergence theorem implies.
		$$C(d_2,\alpha)\, \mme\lVert Y'-\tilde{Y'}\rVert^{\alpha}_2= \int_{\R^{d_2}}\frac{1-\mme\cos\langle t,Y'-\tilde{Y'}\rangle}{\lVert t\rVert^{d_2+\alpha}_2}\,dt= \int_{\R^{d_2}}\frac{1-\mme|\hat{f}_{Y|X}(t)|^2}{\lVert t\rVert^{d_2+\alpha}_2}\,dt.$$
		Similar calculations as above also yield,
		$$\mme\lVert Y_1-Y_2\rVert^{\alpha}_2=\frac{1}{C(d_2,\alpha)}\int_{\R^{d_2}}\frac{1-|\hat{f}_{Y}(t)|^2}{\lVert t\rVert^{d_2+\alpha}_2}\,dt.$$
		By plugging the above displays in the definition of $T_{\alpha}$ in~\eqref{eq:propmeas} completes the proof.
	\end{proof}

	\section{Auxiliary Results}\label{sec:auxres}
	This section contains some results which we have been used earlier in the manuscript. Some of these results are well-known and we have added references for their proofs. Some others are elementary probability exercises and we leave their details to the reader.
	\begin{prop}[A generalized Efron-Stein inequality, {see~\cite[Theorem 2]{Boucheron2005}}]\label{prop:Genef}
		Suppose $W_1,W_2,\ldots ,W_n$ are independent random variables taking values in some normed linear space $\mathcal{W}$ and $F:\mathcal{W}^n\to\R$ be a measurable function. Further assume $\tilde{W}_i$'s are independent copies of $W_i$'s. Let $S\coloneqq F(W_1,\ldots ,W_n)$ and $S_i\coloneqq F(W_1,\ldots ,W_{i-1},\tilde{W}_i,W_{i+1},\ldots ,W_n)$. Finally define $\lVert W\rVert_q\coloneqq \E(|W|^q)^{1/q}$. Then for all integers $q\geq 2$, there exists a constant $\kappa_q$ (depending only on $q$) such that the following holds:
		$$\lVert S-\E[S]\rVert_q\leq \kappa_q\Bigg\lVert \sqrt{\E\left[\sum_{i=1}^n (S-S_i)^2|(W_1,W_2,\ldots,W_n)\right]}\Bigg\rVert_q.$$
		The special case $q=2$ yields the Efron-Stein inequality (see~\cite{Efron1981}).
	\end{prop}
	
	\begin{prop}[Bounding moments of the maximum of random variables]\label{prop:maxmom}
		Suppose $W_1,W_2,\ldots ,W_n$ are i.i.d. random variables taking values in $\R$. Set $M_n\coloneqq \max_{1\leq i\leq n} |W_i|$. If $\E[|W_1|^{p+\epsilon}]<\infty$ for some $p\geq 1$ and $\epsilon>0$, then $\E[M_n^p/n]\overset{n\to\infty}{\longrightarrow} 0$.
	\end{prop}
	
	\begin{prop}\label{prop:basicmom}
		Suppose there exists a sequence of real valued random variables $A_n$ and $B_n$, a deterministic real-valued sequence $a_n$, $b\in \R^+$ and two sequences $v_{n,1}$ and $v_{n,2}$ diverging to $\infty$, such that $v_{n,1}|A_n-a_n|=\mathcal{O}_{\P}(1)$ and $v_{n,2}|B_n-b|=\mathcal{O}_{\P}(1)$. Also assume $A:=\sup_{n\geq 1} a_n<\infty$. Then the following holds:
		$$ (v_{n,1}\wedge v_{n,2})\left(\frac{A_n}{B_n}-\frac{a_n}{b}\right)=\mathcal{O}_{\P}(1).$$
	\end{prop}
	
	\begin{prop}[Lusin's Theorem,~\cite{Loeb2004}]\label{prop:lusin}
		Let $(\Omega_1,\mathcal{A}_1,\mu_1)$ be a Radon measure space of finite measure and $\Omega_2$ be a second countable topological space. Then given any Borel measurable function $f:\Omega_1\to\Omega_2$ and $\epsilon>0$, there exists a compact set $K$, such that $f$ restricted to $K$ is continuous and $\mu_1(\Omega_1\setminus K)<\epsilon$.
	\end{prop}
	
	\begin{prop}[CLT on dependency graphs, {see~\cite[Theorem 2.7]{Chen2004}}]\label{prop:cltdep}
		Suppose $\{X_i\}_{i\in \mathcal{V}}$ be random variables indexed by vertices in a dependence graph. Set $W=\sum_{i\in\mathcal{V}} X_i$. If $\E X_i=0$ for all $i$ and $\mbox{Var}(W)=1$, then we have:
		$$\sup_{z\in \R}|\mathbb{P}(W\leq z)-\Phi (z)|\leq 75D^{10}\sum_{i\in \mathcal{V}}\E |X_i|^3$$
		where $D$ is the maximum degree of the dependency graph and $\Phi(\cdot)$ is the standard Gaussian cumulative distribution function.
	\end{prop}
	
\end{document}